\documentclass[]{article}

\usepackage[a4paper, portrait, margin=1.1811in]{geometry}
\usepackage[english]{babel}
\usepackage[utf8]{inputenc}
\usepackage[T1]{fontenc}
\usepackage{helvet}
\usepackage{etoolbox}
\usepackage{graphicx}
\usepackage{titlesec}
\usepackage[font=footnotesize]{caption}
\usepackage{booktabs}
\usepackage{xcolor} 
\usepackage[colorlinks, citecolor=teal, linkcolor=teal, urlcolor=teal]{hyperref}
\usepackage{caption}
\graphicspath{ {./images/} }
\usepackage{scrextend}
\usepackage{fancyhdr}
\usepackage{authblk}
\usepackage{multicol}
\usepackage{multirow}
\usepackage{subfig}
\usepackage{float}

\usepackage{changes}
\usepackage{longtable}

\usepackage{amsthm}
\usepackage{amsmath}
\usepackage{amssymb}
\usepackage{amsfonts}
\usepackage{mathtools}
\usepackage{censor}

\usepackage{xcolor}
\usepackage[linesnumbered,ruled,vlined]{algorithm2e}
\SetKwComment{Comment}{/* }{ */}

\usepackage{makeidx}
\usepackage[acronym, nonumberlist, style=long, toc]{glossaries}

\usepackage{xkeyval, xfor, amsgen, array}
\makeglossaries

\newacronym{Adam}{Adam}{Adaptive Moment Estimation}
\newacronym[longplural={Artifical Neural Networks}, 
\glsshortpluralkey={ANNs}]{ANN}{ANN}{Artificial Neural Network}
\newacronym{API}{API}{Application Programming Interface}

\newacronym{bagging}{bagging}{Bootstrap Aggregating}

\newacronym{CDF}{CDF}{Cumulative Distribution Function}
\newacronym{CII}{CII}{Carbon Intensity Indicator}
\newacronym[longplural={Confidence Intervals}, \glsshortpluralkey={CIs}]{CI}{CI}{Confidence Interval}
\newacronym{COG}{COG}{Course Over Ground}
\newacronym{CSO}{CSO}{Cleaning Schedule Optimization}

\newacronym{$R^2$}{$\text{R}^2$}{Coefficient of Determination}
\newacronym[longplural={Decision Trees}, \glsshortpluralkey={DTs}]{DT}{DT}{Decision Tree}

\newacronym{DDM}{DDM}{Dry Dock Maintenance}

\newacronym{DSC}{DSC}{Days Since Cleaning}
\newacronym{DSDDM}{DSDDM}{Days Since Dry Dock Maintenance}
\newacronym{DSIWS}{DSIWS}{Days Since In Water Survey}
\newacronym{EPSRC}{EPSRC}{Engineering and Physical Sciences Research Council}
\newacronym{ET}{ET}{Extremely Randomized Trees}
\newacronym{EEOI}{EEOI}{Energy Efficiency Operating Indicator}

\newacronym{FOC}{FOC}{Fuel Oil Consumption}

\newacronym[longplural={Green House Gases}, \glsshortpluralkey={GHGs}]{GHG}{GHG}{Green House Gas}
\newacronym{GUI}{GUI}{Graphical User Interface}

\newacronym{HFO}{HFO}{Heavy Fuel Oil}
 
\newacronym{IMO}{IMO}{International Maritime Organization}
\newacronym[longplural={In-Water Surveys}, \glsshortpluralkey={IWSs}]{IWS}{IWS}{In Water Survey}

\newacronym{kNN}{kNN}{k-Nearest Neighbors}

\newacronym{LASSO}{LASSO}{Least Absolute Shrinkage and Selection Operator}
\newacronym{LinReg}{LinReg}{Linear Regression}
\newacronym{LFO}{LFO}{Light Fuel Oil}
\newacronym{LSHFO}{LSHFO}{Low Sulfur Heavy Fuel Oil}

\newacronym{ML}{ML}{Machine Learning}

\newacronym{MGO}{MGO}{Marine Gasoil}

\newacronym{MLP}{MLP}{Multilayer Perceptron}

\newacronym{MAE}{MAE}{Mean Absolute Error}
\newacronym{MAPE}{MAPE}{Mean Absolute Percentage Error}

\newacronym{MRV}{MRV}{Monitoring, Reporting, and Verification}

\newacronym{MSE}{MSE}{Mean Squared Error}

\newacronym{MedAE}{MedAE}{Median Absolute Error}

\newacronym[longplural={Principal Components}, \glsshortpluralkey={PCs}]{PC}{PC}{Principal Component}
\newacronym{PCA}{PCA}{Principal Component Analysis}

\newacronym{RF}{RF}{Random Forest}

\newacronym{RMSE}{RMSE}{Root Mean Squared Error}
\newacronym{RMSPE}{RMSPE}{Root Mean Squared Percentage Error}

\newacronym{SHAP}{SHAP}{SHapley Additive exPlanations}
\newacronym{SRSG}{SRSG}{Southampton Research Software Group}
\newacronym{STW}{STW}{Speed Through Water}
\newacronym{SMMI}{SMMI}{Southampton Marine and Maritime Institute}
\newacronym{SOG}{SOG}{Speed Over the Ground}
\newacronym{SVM}{SVM}{Support Vector Machine}
\newacronym{SVR}{SVR}{Support Vector Regression}

\newacronym{WUK}{WUK}{Water Under Keel}

\newacronym{XGB}{XGB}{eXtreme Gradient Boosting}

\newacronym{XAI}{XAI}{eXplainable Artificial Intelligence}

\usepackage{enumitem,amssymb}
\newlist{todolist}{itemize}{2}
\setlist[todolist]{label=$\square$}
\usepackage{pifont}

\usepackage{natbib}   
\setcitestyle{authoryear, open={((},close={)}}

\usepackage{tikz}
\usetikzlibrary{arrows.meta}
\usepackage[framemethod=TikZ]{mdframed}
\usepackage{comment}

\usepackage{orcidlink}
\usepackage[misc]{ifsym}
\usepackage{fontawesome}

\usepackage{chngcntr}

\usepackage{lmodern}
\usepackage[bottom]{footmisc}

\ifpdf
\hypersetup{
  pdftitle={Data-Driven Hull-Fouling Cleaning Schedule Optimization to Reduce Carbon Footprint of Vessels},
  pdfauthor={Samuel Ward, Marah-Lisanne Thormann, Julian Wharton, Alain Zemkoho}
}
\fi

\usepackage{fancyvrb}
\usepackage{listings}
\usepackage{inconsolata}

\definecolor{lightgray}{gray}{0.95}
\definecolor{darkgray}{gray}{0.4}
\definecolor{keywordblue}{rgb}{0.26, 0.26, 0.8}
\definecolor{stringred}{rgb}{0.7, 0.0, 0.0}
\definecolor{specialorange}{RGB}{230,120,20}

\lstdefinestyle{custom}{
    backgroundcolor=\color{lightgray},
    basicstyle=\ttfamily\scriptsize,
    keywordstyle=\color{keywordblue}\bfseries,
    commentstyle=\color{gray},
    stringstyle=\color{stringred},
    showstringspaces=false,
    breaklines=true,
    frame=single,
    rulecolor=\color{darkgray},
    framerule=0.5pt,
    xleftmargin=0.5em,
    xrightmargin=0.5em,
    keepspaces=true,
    columns=fullflexible,
    tabsize=4,
    gobble=0,
    upquote=false,
    morekeywords={None,NaN,nan,True,true,False,false}, 
    emph={None,NaN,nan,True,true,False,false},
    emphstyle=\color{specialorange}
}

\theoremstyle{definition}

\SetCommentSty{mycommfont}

\SetKwInput{KwInput}{Input}                
\SetKwInput{KwOutput}{Output}              

\newcommand{\bb}{\mathbf{b}}
\newcommand{\cc}{\mathbf{c}}
\newcommand{\uu}{\mathbf{u}}
\newcommand{\vv}{\mathbf{v}}
\newcommand{\ww}{\mathbf{w}}
\newcommand{\zz}{\mathbf{z}}
\newcommand{\bdelta}{\boldsymbol{\delta}}
\newcommand{\bDelta}{\boldsymbol{\Delta}}
\newcommand{\bLambda}{\boldsymbol{\Lambda}}

\mdfdefinestyle{framedbox}{innertopmargin=10pt,linecolor=gray!20,%
linewidth=2pt,topline=true,needspace=5\baselineskip,%
frametitleaboveskip=\dimexpr-\ht\strutbox\relax
}

\newcounter{defi}[section]
\renewcommand{\thedefi}{\arabic{section}.\arabic{defi}}
\newenvironment{defi}[2][]{%
\refstepcounter{defi}%
\ifstrempty{#1}%
{\mdfsetup{%
frametitle={%
\tikz[baseline=(current bounding box.east),outer sep=0pt]
\node[anchor=east,rectangle,fill=gray!20]
{\strut Definition~\thedefi};}}
}%
{\mdfsetup{%
frametitle={%
\tikz[baseline=(current bounding box.east),outer sep=0pt]
\node[anchor=east,rectangle,fill=gray!20]
{\strut Definition~\thedefi:~\textbf{#1}};}}%
}%
\begin{mdframed}[style=framedbox]\relax%
\label{#2} \vspace{-0.3cm}}{\end{mdframed}}

\newcounter{theo}[section]

\renewcommand{\thetheo}{\arabic{section}.\arabic{theo}}
\newenvironment{theo}[2][]{%
\refstepcounter{theo}%
\ifstrempty{#1}%
{\mdfsetup{%
frametitle={%
\tikz[baseline=(current bounding box.east),outer sep=0pt]
\node[anchor=east,rectangle,fill=gray!20]
{\strut Theorem~\thetheo};}}
}%
{\mdfsetup{%
frametitle={%
\tikz[baseline=(current bounding box.east),outer sep=0pt]
\node[anchor=east,rectangle,fill=gray!20]
{\strut Theorem~\thetheo:~#1};}}%
}%
\begin{mdframed}[style=framedbox]\relax%
\label{#2} \vspace{-0.3cm}}{\end{mdframed}}

\newenvironment{prop}[2][]{%
\refstepcounter{theo}%
\ifstrempty{#1}%
{\mdfsetup{%
frametitle={%
\tikz[baseline=(current bounding box.east),outer sep=0pt]
\node[anchor=east,rectangle,fill=gray!20]
{\strut Proposition~\thetheo};}}
}%
{\mdfsetup{%
frametitle={%
\tikz[baseline=(current bounding box.east),outer sep=0pt]
\node[anchor=east,rectangle,fill=gray!20]
{\strut Proposition~\thetheo:~#1};}}%
}%
\begin{mdframed}[style=framedbox]\relax%
\label{#2} \vspace{-0.3cm}}{\end{mdframed}}

\newcounter{prf}[section]

\newenvironment{prf}[2][]{%
\refstepcounter{prf}%
\ifstrempty{#1}%
{\mdfsetup{%
frametitle={%
\tikz[baseline=(current bounding box.east),outer sep=0pt]
\node[anchor=east,rectangle,fill=gray!20]
{\strut Proof};}}
}%
{\mdfsetup{%
frametitle={%
\tikz[baseline=(current bounding box.east),outer sep=0pt]
\node[anchor=east,rectangle,fill=gray!20]
{\strut Proof:~#1};}}%
}%
\begin{mdframed}[style=framedbox]\relax%
\label{#2} \vspace{-0.3cm}}{\qed\end{mdframed}}

\SetKwInput{KwInput}{Input}

\begin{document}

\pagestyle{fancy}
\lhead{Ward et al.}
\rhead{Biofouling Cleaning Schedule Optimization}
\setlength{\headheight}{32pt}

\newpage
\setcounter{page}{1}
\renewcommand{\thepage}{\arabic{page}}
	
\captionsetup[figure]{labelfont={bf},labelformat={default},labelsep=period,name={Figure }}	\captionsetup[table]{labelfont={bf},labelformat={default},labelsep=period,name={Table}}
\setlength{\parskip}{0.5em}



\fancypagestyle{plain}{
	\fancyhf{}
	\renewcommand{\headrulewidth}{0pt}
	\renewcommand{\familydefault}{\sfdefault}
	
	\lhead{\color{teal}\small \textbf{Preprint}\\ \color{black}
	\text{}\\ }
	
}

\makeatletter
\patchcmd{\@maketitle}{\LARGE \@title}{\fontsize{20}{22}\selectfont\@title}{}{}
\makeatother



\author{
    {\large \textbf{Samuel J. Ward}} {\normalsize \href{mailto:s.ward@soton.ac.uk}{\Letter} \href{https://www.southampton.ac.uk/people/5zv8fk/mr-samuel-ward}{\faHome} {\large \orcidlink{0009-0001-3084-7099}} \protect\\ \vspace{0.1cm}
     School of Mathematical Sciences, \protect\\ University of Southampton, \protect\\  SO17 1BJ Southampton, UK}  \vspace{0.2cm}
     
    {\large \textbf{Marah-Lisanne Thormann}} {\normalsize \href{mailto:m.-l.thormann@soton.ac.uk}{\Letter}  \href{https://www.southampton.ac.uk/people/5ztphy/ms-marah-thormann}{\faHome} \protect\\ \vspace{0.1cm}
     School of Mathematical Sciences, \protect\\ University of Southampton, \protect\\  SO17 1BJ Southampton, UK} \\ \vspace{0.1cm}

    {\large \textbf{Julian Wharton}} {\normalsize \href{mailto:J.A.Wharton@soton.ac.uk}{\Letter} \href{https://www.southampton.ac.uk/people/5wy6z8/professor-julian-wharton}{\faHome} {\large \orcidlink{0000-0002-3439-017X
}} \protect\\
    \vspace{0.1cm} National Centre for Advanced Tribology at Southampton (nCATS), \protect\\ University of Southampton, \protect\\  SO17 1BJ Southampton, UK} \vspace{0.2cm}
    
    {\large \textbf{Alain Zemkoho}} {\normalsize \href{mailto:a.b.zemkoho@soton.ac.uk}{\Letter} \href{https://www.southampton.ac.uk/~abz1e14/index.html}{\faHome} {\large \orcidlink{0000-0003-1265-4178}} \protect\\
     \vspace{0.1cm} School of Mathematical Sciences, \protect\\University of Southampton, \protect\\  SO17 1BJ Southampton, UK} 
}




\titlespacing\section{0pt}{12pt plus 4pt minus 2pt}{0pt plus 2pt minus 2pt}
\titlespacing\subsection{12pt}{12pt plus 4pt minus 2pt}{0pt plus 2pt minus 2pt}
\titlespacing\subsubsection{12pt}{12pt plus 4pt minus 2pt}{0pt plus 2pt minus 2pt}

\titleformat{\section}{\normalfont\fontsize{10}{15}\bfseries}{\thesection.}{1em}{}
\titleformat{\subsection}{\normalfont\fontsize{10}{15}\bfseries}{\thesubsection.}{1em}{}
\titleformat{\subsubsection}{\normalfont\fontsize{10}{15}\bfseries}{\thesubsubsection.}{1em}{}


\title{\vspace{-0.6cm} \textbf{Data-Driven Hull-Fouling Cleaning Schedule Optimization to Reduce  Carbon Footprint of Vessels} \vspace{0.2cm}}

\date{}  

\begingroup
\let\center\flushleft
\let\endcenter\endflushleft

\maketitle

\vspace{-0.8cm}

	

\vspace{0.2cm}

\noindent\rule{1.25cm}{0.4pt}  \rlap{\color{gray}\vrule}%
  \fboxsep1.5mm\colorbox[rgb]{1,1,1}{\raisebox{-0.4ex}{%
    \large\selectfont\sffamily\bfseries\abstractname}} \noindent\rule{11.5cm}{0.4pt}

    \noindent
    In response to climate change, the International Maritime Organization has introduced regulations to reduce greenhouse gas emissions from international shipping. 
    Compliance~with these regulations is increasingly expected from individual shipping companies and can be achieved through more frequent hull and propeller cleanings.
    Since this maintenance work entails high costs, vessel operators are confronted with the trade-off between lower fleet emissions and better economic performance.
    As a result, assessing whether fuel efficiency has declined sufficiently to warrant cleaning from an environmental and economic standpoint is a critical task. 
    To address this challenge, this paper proposes a novel dynamic-programming-based algorithm to optimize vessel cleaning schedules.
    The main components of the optimization problem are a data-driven function for the fuel consumption costs and an additive term that sums up the costs per vessel cleaning.
    In numerical experiments, we demonstrate the usefulness of our proposed methodology based on real-world sensor data from nine tramp trading vessels. 
    The results confirm that, over four years, fuel consumption can be reduced by up to 5\%, even when accounting for the costs of the additional cleaning events.
    \\ \\
    \textbf{Keywords}: Biofouling; Cleaning Schedule Optimization; Dynamic Programming; Fuel Consumption; Machine Learning
        
    \vspace{0.2cm}
  
    \noindent \textbf{Mathematics Subject Classification (2020)}: 90B06; 90B25; 90C27; 90C39

    \vspace*{\fill}

    \noindent \rule{15cm}{0.2pt}
    {\footnotesize \noindent 
    \textbf{Funding Acknowledgement}: 
    The work presented in this paper was funded by both the SMMI Research Collaboration Stimulus Fund 2022/23 and an EPSRC Impact Acceleration Account grant with reference 9060958. 
    Samuel Ward's PhD is jointly funded by DAS Ltd and the School of Mathematical Sciences at the University of Southampton.
    Marah-Lisanne Thormann receives a PhD Studentship from the School of Mathematical Sciences at the University of Southampton. 
    Alain Zemkoho is supported by the \href{https://www.ukri.org/councils/epsrc/}{EPSRC} with grant reference EP/X040909/1. 
    }

\endgroup

\newpage

\section{Introduction}
\label{sec:Intro}


``Global warming is not a prediction. It is happening [and refers to the long-term increase in Earth's average temperature]'' [\cite{hansen2012gameover}]. 
A key driver of this climate change is the accumulation of $\text{CO}_2$ and other anthropogenic \glspl{GHG} in the atmosphere [\citet{Nakamura2025}].
One of the major contributors to these emissions is the transportation industry [\citet{Din2023}; \citet[pp. 354--355]{CHAPMAN2007}], where international shipping plays a central role, since approximately 90\% of global goods are transported by sea [\citet[p.~50]{parviainen2018};  \citet[p. 362]{CHAPMAN2007}]. 
In 2012, global shipping emissions amounted to about 938 million tonnes of  $\text{CO}_2$, and 961 million tonnes of $\text{CO}_{2e}$ that combines $\text{CO}_2$, $\text{CH}_4$ and $\text{N}_2 \text{O}$ [\citet[p.~1]{HUANG2022}]. 
These emissions represented roughly 2.2\% of global anthropogenic \glspl{GHG} [\citet[p.~1]{HUANG2022}]. 
In response, the \gls{IMO} has set targets to reduce the emissions by at least 20\% by 2030, 70\% by 2040, and to achieve net-zero by 2050 [\citet[p. 6]{MEPC_Resolutions2023}]. 
In Europe, these objectives are further reinforced by the EU’s \gls{MRV} regulation, which mandates reductions in maritime $\text{CO}_2$ emissions and introduces penalties for non-compliance [\citet[p. 3]{MRV}].


To meet the emission targets of the industry, the \gls{IMO} has increasingly held individual shipping companies accountable.
For instance, vessel operators must now report performance ratings like the \gls{CII}, which ranks vessels on a scale from A to E based on their energy efficiency [\cite{IMO2016}]. 
A key component of the \gls{CII} is the vessel-specific fuel efficiency, which can be improved through lower operating speeds, advanced engine technology, and optimized hull design [\citet[p.~362]{CHAPMAN2007}]. 
In particular, the condition of the hull is crucial for the fuel efficiency, as a deteriorated surface can significantly increase hydrodynamic drag.
In general, the hull deterioration begins once vessels operate in water, which can be attributed to the build up of hull and propeller fouling [\citet[pp.~1--2]{DEMIREL2019}]. 
At present, the main strategy to mitigate this deterioration is to regularly perform costly hull inspections, cleanings, and repaintings [\cite{Granhag2023}]. 
Shipping companies therefore face the following trade-off: cleaning frequently to reduce environmental impact versus performing maintenance only as needed.
Reconciling these competing objectives is one of the major open issues in the industry, which requires the determination of the optimal cleaning time point without costly in-water inspections.
To address this challenge, this paper proposes a data-driven optimization algorithm based on dynamic programming, which identifies the optimal cleaning schedule by approximating the impact of hull fouling on fuel efficiency.

\subsection{Ship Performance Modeling}


Optimizing the energy efficiency of maritime transportation has long been a complex challenge, which is typically addressed through measures related to ship design and operation [\citet[p. 1]{LANG2022110387}]. 
While numerous indicators exist for operational measures, those derived from fuel consumption and speed remain the most prevalent [\citet[p. 545]{soner2019}]. 
Specifically, the \gls{FOC} is directly integrated into the \gls{CII} and represents more than a quarter of a vessel’s operating costs [\citet[p. 1]{GKEREKOS2019}]. 
As \gls{GHG} emissions are also directly proportional to this consumption [\citet[p. 351]{CORADDU2017}], \gls{FOC} is adopted as the primary metric for evaluating operational costs and energy efficiency within our data-driven optimization approach introduced in Section~\ref{sec:CSO}.


Given that \gls{FOC} serves as a key metric in this study, its accurate prediction under varying operational conditions is essential.
This requires a mathematical model capable of incorporating multiple variables and their interdependencies.
In the literature, modeling of \gls{FOC} is an active research area, with numerous recent studies employing \gls{ML}-based approaches for prediction [\citet{AGAND2023, ZHOU2022, moreira2021, YAN2020, soner2019, GKEREKOS2019, CORADDU2017}]. 
For example, \citet{GKEREKOS2019} conducted a comparative study employing a diverse set of \gls{ML} algorithms for \gls{FOC} modeling.
Across multiple performance metrics, \gls{SVM} and \gls{ET} achieved the highest predictive accuracy. 
Conversely, \citet{AGAND2023} reported \gls{XGB} as the top-performing model in their comparison, while \citet{YAN2020} identified \gls{RF} as the best performer.
These examples underscore that no consensus has been reached on an overall best modeling approach, primarily due to the heterogeneity of datasets, variable selection, and performance metrics across existing studies.
Accordingly, this paper provides a brief overview of commonly used \gls{ML} models in Subsection~\ref{sec:reg_models}, and later compares their performance in Subsection \ref{sec:selection_of_fuel_prediction_function} to identify the most suitable approach for predicting \gls{FOC} on the given empirical datasets.


Once an appropriate mathematical framework for predicting \gls{FOC} is established, the influence of the underlying variables can be assessed.
For linear models, this is typically achieved by interpreting regression coefficients, whereas for most of the other \gls{ML}-based approaches, interpretability remains considerably more challenging.
These so-called black-box algorithms often lack transparency for domain experts due to their reliance on numerous complex non-linear transformations [\citet[p. 154097]{Loyola2019}].
One way to mitigate this black-box characteristics is through tools from \gls{XAI}, which have been developed in recent years and are gaining increasing attention in the academic literature.
In particular, a widely adopted technique is the use of \gls{SHAP} values, which have been applied by several researchers to enhance the interpretability of \gls{FOC} models [\citet{handayani2023}; \citet{ma2023}; \citet{WANG2023}; \citet[ch. 9.6]{molnar2022}].
In Subsection~\ref{sec:shap}, this approach is explained in more detail, and in Subsection~\ref{sec:selection_of_fuel_prediction_function} it is applied to verify whether the independent variables exert the expected influence on \gls{FOC} predictions.
The main novelty of this aspect of our study lies in computing and evaluating \gls{SHAP} values for input variables approximating the effects of hull and propeller fouling on \gls{FOC}.

\subsection{Hull Fouling and its Impacts}


Before discussing suitable approaches for approximating and detecting hull fouling, we define the phenomenon as accumulation of living organisms on the submerged surface of a vessel. 
It begins with a thin biofilm of bacteria that roughens the hull and propellers.
Later in the process, larger macroalgae, plants and hard foulers (i.e., barnacles and mussels) will attach to the surface [\citet[p. 3]{alsawaftah2022}]. 
From an environmental perspective, hull fouling has two major negative impacts.
The first relates to invasive species carried within the ship's fouling, which can outcompete local species or transmit diseases between normally disconnected habitats [\citet{Lewis2009}; \citet[p. 116]{Floerl2003}].
For instance, \citet{Darwin1851} had already noted the spread of barnacles transported by trading ships more than a century ago. 
Similarly, a  review by \citet[p. 485]{Molnar2008} concluded that invasive marine species are a major threat to biodiversity, with profound ecological and economic consequences. 
In line with these findings, the \gls{IMO} has identified this aspect of biofouling ``... as one of the greatest threats to the ecological and the economic well being of the planet ...'' [\citet{IMO2023}].


In addition to the spread of invasive species, a second negative environmental impact of hull fouling originates from increased frictional resistance between the ship's hull and the surrounding water, which can raise \gls{FOC} by approximately 5–17\% [\citet{Adland2018, deHaas2023, bakka2022, Nikolaidis2022, CORADDU2019, CORADDU2019b, UZUN2019, DEMIREL2017}]. 
Depending on the fuel type, this may correspond to a substantial increase in emissions of sulfur, particles, and carbon dioxide to the atmosphere [\citet[p.~3]{Oliveira2022}]. 
This highlights the direct relevance of hull fouling to climate change and its conflict with the \gls{IMO}'s \gls{GHG} reduction targets.
Consequently, from an environmental perspective, maintenance activities such as hull inspections and cleanings should be carried out as frequently as practicable. 
However, since most of this maintenance work is very expensive [\citet[p. 1]{GUPTA2022}], the actual degree of fouling on a vessel often remains unknown in practice.
In particular, smaller shipping companies might neither have the budget for specialized equipment such as underwater cleaning robots nor to frequently hire professional divers to perform in-water inspections.


To address the financial constraints of hull inspections, an emerging research direction focuses on the use of inexpensive underwater cameras to enable reliable fouling assessment whenever vessels are in port.
For instance, \citet{bloomfield2021} and \citet{first2021} already demonstrated how images captured by these cameras can be combined with automated image classification techniques to quantify fouling. 
Although this approach appears promising, its integration into our \gls{FOC} model was not feasible because the vessel datasets available for this study do not contain sufficient imaging data.
To still be able to account for the hull fouling effect in our predictive models, we decided to approximate it by counting the \gls{DSC}, which was suggested by \citet{bakka2022} and \citet{LAURIE2021}.
A major advantage of this approach is the inexpensive derivation of the corresponding variable based on vessel-specific cleaning event records in the provided datasets. 
This \gls{DSC} variable can also be supplemented by the hours a vessel spends at anchor, or in low-speed and high-speed transit, which can easily be computed as well.
For these types of operational states, it is well established that they result in differing rates of biofouling build up [\citet{deHaas2023}; \citet[p. 350]{Lewis2009}], especially in the case of self-polishing paints [\citet{Bressy2009}]. 
Overall, we therefore construct a vector of biofouling measures rather than relying on a single variable, which would likely only partially capture the targeted latent effect.
To the best of our knowledge, this is a novel approach.
Once the \gls{FOC} prediction model is estimated, the influence of hull fouling can then be assessed by simulating a reset of this vector to zero.

\subsection{Hull Cleaning Events and FOC Optimization}\label{sec:hull_cleaning}


To compute the variable representing the \gls{DSC}, a cleaning event must be performed and recorded.
Over time, various hull-cleaning methods and devices, applied either on land or in water, have been introduced, as summarized in the comprehensive review by \citet{song2020}.
In principle, all these cleaning options could be incorporated as separate explanatory variables.
However, since this study is limited to the cleaning records available in the datasets, we differentiate between two types of events. 
The first type is dry-dock cleaning, where the vessel is removed from the water for extensive maintenance on land. 
During this process, the hull is typically treated with sandblasting and high-pressure water to remove fouling [\citet[p.~249]{HUA2018}].
Afterwards, a fresh anti-fouling paint may be applied. 
Traditionally, copper-based coatings dominated the market, but environmental concerns [\citet[pp. 9--10]{Townsin2003}] have driven a shift toward silicone-based alternatives [\cite{Srinivasan2007}] or non-anti-fouling paints, which both foul quicker but offer less initial resistance.


The second cleaning option considered in this study is in-water maintenance, where divers or specialized machines scrub and water-jet the hull as needed; see Figure~\ref{fig:in_water}. 
While this approach provides operational flexibility, its effectiveness is comparatively lower [\cite{Tamburri2020}].
Regardless of effectiveness, both cleaning methods impose costs on the operator. 
Dry-dock cleaning is particularly expensive, not only due to service charges but also because the vessel loses several days of operation that could otherwise generate revenue.
For example, \citet{Schultz2011} estimated the average cost of a cleaning for a US Naval vessel at approximately \$26,808. 
Other studies report costs ranging from \$5,000 to \$50,000, depending on  vessel size [\cite{Wu2022}].

\begin{figure}[H]
    \centering
    \includegraphics[width=0.3\linewidth]{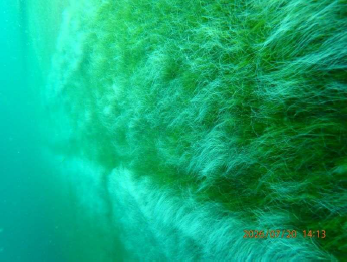}
    \hspace{1em}
    \includegraphics[width=0.3\linewidth]{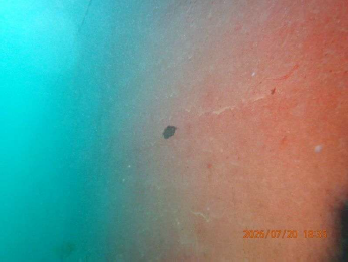}
    \caption{Photos before (left) and after (right) an in-water cleaning
    [Carisbrooke Shipping Ltd (2026)].}
    \label{fig:in_water}
\end{figure}


Given the available cleaning methods, the question remains when to perform them.
In practice, many ship operators adopt a policy of re-coating the hull every three to five years or when a significant performance drop is observed [\citet[p.~416]{song2020}; \citet[p.~255]{HUA2018}]. 
While this provides a reasonable heuristic, it is far from optimal. 
From a mathematical perspective, the procedure could be substantially improved through a simple sensitivity analysis that varies an input variable to assess how strongly predicted fuel consumption changes.
If a variable is identified as critical for reducing \gls{FOC}, a more advanced optimization framework can then be developed.
For example, \citet{Huotari2021} demonstrate an approach where voyage speeds are varied to minimize fuel consumption using convex optimization combined with Dijkstra’s algorithm. 
Similarly, \citet{Fagerholt2010} aim to reduce fuel emissions by optimizing vessel speed along shipping routes.
A different optimization approach is presented by \citet{GorenHuber2017}, who utilize a linear regression model to predict the log-transformed fuel consumption based on several input variables, including features related to propeller polishing, hull cleaning and dry dock events.
This data-driven model then serves as the core component of a mixed-integer linear program that aims to minimize operational costs by determining optimal values for the mentioned maintenance-related variables.
Inspired by these speed- and maintenance-based methodologies, we propose a novel optimization framework based on dynamic programming that determines the optimal time points for performing cleanings to mitigate the negative effects of hull fouling.


\subsection{Main Contributions \& Paper Outline}


This paper presents the following contributions, which, to the best of our knowledge, are each new to the literature: 
\vspace{-0.2cm}
\begin{enumerate}\setlength\itemsep{0.25em}
    \item We introduce a mathematical formulation for the problem of choosing an optimal hull cleaning schedule, which is presented in Section~\ref{sec:CSO} as~\eqref{eq:cleaning-2}.
    Based on this formulation, ship operating companies can schedule hull cleanings using vessel performance data while taking into account environmental and economic considerations.
    This reduces the dependence on expensive underwater surveys, which are typically used to assess the hull condition.

    \item We design a novel dynamic-programming-based algorithm that efficiently computes solutions of \eqref{eq:cleaning-2}.  
    The algorithm includes a data-driven function to predict fuel consumption, which can be represented by \gls{ML} models or physical relationships.     
    Both in theory and in our numerical analysis, we demonstrate the computational efficiency of our approach. 
    Specifically, the optimal cleaning schedules of vessels are determined over a four-year period (approx. 30{,}000 observations) in only two minutes of computation time.
    
    \item We use \gls{SHAP} values to analyze and verify how different biofouling metrics influence \gls{FOC} predictions of \gls{ML} models. 
    While previous studies used \gls{SHAP} values to analyze \gls{FOC} predictions, they have not used them to verify metrics that approximate the effect of biofouling. 
    Our results indicate that variables such as the \gls{DSC} can lead to plausible \gls{SHAP} patterns, with their average influence on \gls{FOC} being consistent with expectations.

    \item We prepare and present a large dataset of around $800{,}000$ ship sensor observations across ten commercial tramp trading vessels measured over $1{,}000$ voyages.
    This data has never been analyzed in the literature before and allows us to conduct a unique case study of the proposed methodology, which demonstrates that our optimized hull cleaning schedules can reduce \gls{FOC} by up to 5\%. 

    \item We describe how the entire pipeline is automated and deployed as a user-friendly web app to allow our industrial partner to interact with and verify our results on an ongoing basis.
\end{enumerate}
\vspace{-0.15cm}


The remaining sections of this paper are organized as follows.
Subsection \ref{sec:reg_models} briefly introduces several \gls{ML} algorithms that are later used to predict \gls{FOC}.
Afterwards, the core idea of \gls{SHAP} values is explained in Subsection~\ref{sec:shap}. 
In Section \ref{sec:CSO}, the data-driven optimization algorithm is then introduced, and important theoretical properties are derived. 
This is followed by the description of the empirical datasets in Section~\ref{sec:empirical_data}. 
Subsequently, Section \ref{sec:numerical_experiments} presents the results of our numerical studies that are performed based on the previously derived theoretical setup.
At the end of this section, we also describe the implementation and deployment of the web app.
Lastly, the impact of this work is then summarized in Section~\ref{sec:conclusion}.

\section{Fuel Predictions with Explainable Machine Learning}
\label{sec:fuel_prediction_with_explainable_ML}


To introduce our data-driven \gls{CSO} for ships that balances environmental and economic considerations, 
it is necessary to construct a data‑driven function $g: \mathbb{R}^k \rightarrow \mathbb{R}$ that predicts fuel consumption for a given ship profile with $k \in \mathbb{N}$ independent variables such as vessel speed, cargo load or trim. 
Accordingly, this section outlines the methodological foundation required to construct such a function and to interpret its predictions. 
Specifically, the first subsection briefly introduces common \gls{ML} algorithms used in the literature to estimate~$g$. 
Since most of these algorithms are black-box models, the second subsection presents an \gls{XAI} tool that is becoming increasingly popular in the literature on \gls{FOC} modeling. 
In the numerical application in Subsection~\ref{sec:selection_of_fuel_prediction_function}, this tool is then used to verify whether the features of the ship profile exert a plausible influence on the \gls{FOC} prediction.

\subsection{Preliminaries on Regression Models}\label{sec:reg_models}


The estimation of any regression model requires the availability of a dataset containing $m \in \mathbb{N}$ observations, each with $k$ covariates and a corresponding dependent variable.
In this paper, we denote this dataset by $\mathcal{D} \coloneqq \bigl \{ (\mathbf{x}_1, y_1), \ldots, (\mathbf{x}_m, y_m)  \bigr\}$, where $\mathbf{x}_t \coloneqq \begin{bmatrix} x_{t1}, \ldots, x_{tk} \end{bmatrix}\in \mathbb{R}^{1 \times k}$ represents the ship profile at time point $t \in \{1, \ldots, m\}$, and $\mathbf{y} \coloneqq \begin{bmatrix} y_1, \ldots, y_m \end{bmatrix}^\top \in \mathbb{R}_{+}^m$ denotes the observed fuel consumption. 
To refer to the collection of ship profiles for a given vessel, we use the following notation, sometimes called a design matrix,
\begin{equation*}
    \mathbf{X} \coloneqq   \begin{bmatrix}
                       x_{11} & x_{12} & \cdots & x_{1k} \\ 
                       x_{21} & x_{22} & \cdots & x_{2k} \\
                       \vdots & \vdots & \ddots & \vdots \\
                       x_{m1} & x_{m2} & \cdots & x_{mk} 
                    \end{bmatrix}
                =   \begin{bmatrix}
                       \mathbf{x}_{1}^\top   \\ 
                       \mathbf{x}_{2}^\top   \\
                       \vdots  \\
                       \mathbf{x}_{m}^\top 
                    \end{bmatrix} \in \mathbb{R}^{m \times k},
\end{equation*}
where each row corresponds to the profile of the vessel at a given time point, and each column contains the realizations of a variable $l$ with $l \in \{1, \ldots, k\}$. 
Lastly, the collection of all predicted fuel consumptions is denoted by $\hat{\mathbf{y}} \coloneqq \begin{bmatrix} \hat{y}_1, \ldots, \hat{y}_m \end{bmatrix}^\top \in \mathbb{R}^m$.


Based on the previously described model components, we can introduce a diverse set of regression estimators used in the literature to predict \gls{FOC} for a given ship profile $\mathbf{x}_t$.
In particular, we propose two simpler approaches, \gls{LinReg} and \gls{LASSO}, which generate predictions by linearly combining the input variables, i.e., $\hat{y}_t = \beta_1 \cdot x_{t1} + \ldots + \beta_m \cdot x_{tk}$, without and with variable selection through regularization, respectively. 
This has the advantage that the resulting regression coefficients $\boldsymbol{\beta} \coloneqq \begin{bmatrix} \beta_1, \ldots, \beta_k \end{bmatrix}^\top \in \mathbb{R}^k$ can be interpreted to assess the influence of individual covariates. 
Besides these two linear approaches, more advanced \gls{ML} techniques are frequently used in the literature, which apply several non-linear transformation to the input variables to obtain a prediction $\hat{y}_t$.
Among these techniques we select three popular tree-based ensemble estimators, \gls{RF}, \gls{ET}, and \gls{XGB}, whose core idea is to partition the observations into subgroups based on their covariate values. 
Another non-parametric \gls{ML} approach considered in this paper is \gls{kNN}. 
It predicts $\hat{y}_t$ for a new instance by identifying, based on a distance measure, a neighborhood of similar previously observed cases and aggregating their corresponding $y_t$-values.
Lastly, we also evaluate two parametric non-linear \gls{ML} models in our comparison shown in Subsection~\ref{sec:selection_of_fuel_prediction_function}, namely the \gls{MLP} and \gls{SVR}.
For the sake of brevity, we omit their detailed description here.
Interested reader can find an explanation of their core ideas in Appendix~\hyperref[subsec:Linear_Regression]{\ref*{subsec:Linear_Regression}}--\hyperref[subsec:XGBoost]{\ref*{subsec:XGBoost}}, which also provide additional information on the other models.
Because both the analysis of biofouling-variable effects and optimization formulation are agnostic to the choice of regression technique, we refer in what follows only to the predictor $g$, rather than to any specific model.  

\vspace{-0.15cm}

\subsection{Shapley Additive Explanations}\label{sec:shap}


In the previous subsection, different \gls{ML} algorithms were briefly introduced. 
While the predictions of the two linear methods discussed above are easy to interpret, the remaining methods apply more complex transformations, resulting in predictions that cannot be verified by analyzing regression coefficients [\citet[p. 33--36]{Gianfagna2021}]. 
One way of mitigating this black-box characteristics is to use modern \gls{XAI} tools [\citet[pp. 2--4]{Kamath2021}]. 
These approaches aim to construct a simpler explanation model $h: \mathbb{R}^{k} \rightarrow \mathbb{R}$ that can be any interpretable approximation of the original model denoted by $g$ [\citet[p. 2]{Lundberg2017}].


Within the landscape of \gls{XAI} tools, a distinction can be made between global and local explanation approaches [\citet[pp. 81--82]{Gianfagna2021}]. 
Local techniques aim to explain an individual prediction $g(\mathbf{x}_t)$, whereas global methods seek to summarize the effects of input variables across the entire dataset [\citet[Chapter~6]{molnar2022}]. 
One frequently used local model-agnostic approach are \acrfull{SHAP}, which aim to explain the prediction for an instance $\mathbf{x}_t$ by calculating the contribution of each variable to the model output [\citet[Chapter~9.6]{molnar2022}]. 
The idea is closely related to Shapley values, which were developed in game theory to fairly distribute the total payout among players in a coalition [\citet[p. 117]{Rothman2020}]. 
In this game-theoretical framework, the individual contribution $\phi_l$ of an arbitrary player $l$ to the total payout $f_x$ generated by a group of players $M \in \mathbb{N}$ can be calculated as
\begin{equation}\label{eq:shap}
    \phi_l = \sum_{S \subseteq N \setminus \{l\}} \frac{|S|!(M - |S| - 1)!}{M!}[f_x(S \cup \{l\}) - f_x(S)],
\end{equation}
where $N \setminus \{l\}$ are all coalitions of players that are possible without player $l$. 
Specifically, the individual contribution of player $l$ is calculated as the difference in payouts between coalitions with and without this player, i.e., $f_x(S \cup \{l\}) - f_x(S)$. 
These differences are then weighted via $\frac{|S|!(M - |S| - 1)!}{M!}$, which assigns a higher importance to cases in which player $l$ either acts nearly alone or joins a group with many others. 
To connect this framework with the prediction process of a black-box model, \gls{SHAP} replaces the players in a coalition with the individual variable realizations of an observation  $\mathbf{x}_t$, and redefines the payout as the corresponding model prediction $\hat{y}_t$ [\citet[p.~94]{Gianfagna2021}].
Under additional assumptions, the resulting \gls{SHAP} values are then the solution to  \eqref{eq:shap} [\citet[pp.~4--5]{Lundberg2017}]. 
For each observation $t$ and feature $l$, a corresponding $\phi_{tl}$ value can therefore be calculated, which can subsequently be interpreted as the contribution of that feature value to the difference between the instance-specific prediction $\hat{y}_t$ and the mean prediction [\citet[Chapter~9.6]{molnar2022}; \citet[pp. 191--192]{Kamath2021}]. 

\vspace{-0.2cm}

\section{Data-Driven Hull Fouling Cleaning Schedule Optimization}
\label{sec:CSO}

In the previous section, we set out the notation for training a \gls{ML} model that provides us with a predictor~$g(\mathbf{x}_t)$ to estimate the \gls{FOC} at time index~$t$.
This becomes a central component of the following formulation, as we wish to minimize the fuel costs predicted by $g$ plus the cleaning costs.

Subsection~\ref{sec:cso_setting} introduces parameters and decision variables, and Subsection~\ref{sec:cso_general} presents the optimization formulation that we call \ref{eq:cleaning}.
While this formulation provides a very general set-up, solving the resulting problem is difficult in practice, as we show that it is NP-hard.
In Subsection~\ref{sec:cso_fouling_state}, we then add an assumption that allows us to introduce the more tractable fouling-state formulation called~\ref{eq:cleaning-2}.
Subsequently, Subsection~\ref{sec:cso_dp} then derives a dynamic programming algorithm that solves this model exactly with a complexity that is quadratic in the number of voyages.
This second approach is the one we use in practice.

\subsection{Problem Setting and Notation}
\label{sec:cso_setting}


We consider a single vessel over a planning horizon of $n \in \mathbb{N}$ chronologically ordered voyages, which are indexed by $j=1, \dots, n$.
Refer to Figure~\ref{fig:typical_itinerary} for a typical itinerary.
The port of departure of each voyage offers an opportunity to clean the hull, which is modeled by the binary decision variables
\begin{equation*}
    z_j =
    \begin{cases}
    1 & \text{if the hull is cleaned before voyage } j, \\
    0 & \text{otherwise},
    \end{cases}
    \qquad j=1,\dots,n.
\end{equation*}
These decisions are collected in the cleaning schedule $\zz \coloneqq \begin{bmatrix} z_1, \ldots, z_n \end{bmatrix}^\top \in \{0,1\}^n$, and $\zz_{1:j} \coloneqq \begin{bmatrix} z_1, \ldots, z_j \end{bmatrix}^\top \in \{0,1\}^j$ denotes the cleaning history up to and including voyage~$j$.
A cleaning before voyage~$j$ incurs the cost $c_j \geq 0$, so that the total cleaning costs of a schedule are $\cc^\top \zz$ with $\cc \coloneqq \begin{bmatrix} c_1, \ldots, c_n \end{bmatrix}^\top \in \mathbb{R}^n_{+}$.
In practice, not every port may offer a cleaning service.
In this case, the corresponding decision is simply fixed to $z_j = 0$, which the algorithms below handle by skipping the option to clean before voyage~$j$.
The hourly observations recorded during voyage~$j$ are indexed by the set $\mathcal{T}_j \subseteq \{1, \ldots, m\}$, which contains $|\mathcal{T}_j|$ time points.


We partition each ship profile into two parts, $\mathbf{x}_t = (\uu_t, \mathbf{s}_t)$.
The vector $\mathbf{s}_t \in \mathbb{R}^d_{+}$ contains $d \in \mathbb{N}$ biofouling measures, such as the days or the hours at a certain speed since the last cleaning (see Subsection~\ref{sec:empirical_results_cso} for the measures used in this study), and therefore depends on the cleaning schedule.
The vector $\uu_t \in \mathbb{R}^{k-d}$ contains all remaining variables, such as the speed, the draught, or the weather conditions, which are treated as given and independent of the cleaning schedule.
We write $g(\uu_t, \mathbf{s}_t) \coloneqq g(\mathbf{x}_t)$ to make the notation more convenient.

\begin{figure}[H]
    \centering
    \includegraphics[width=0.8\linewidth]{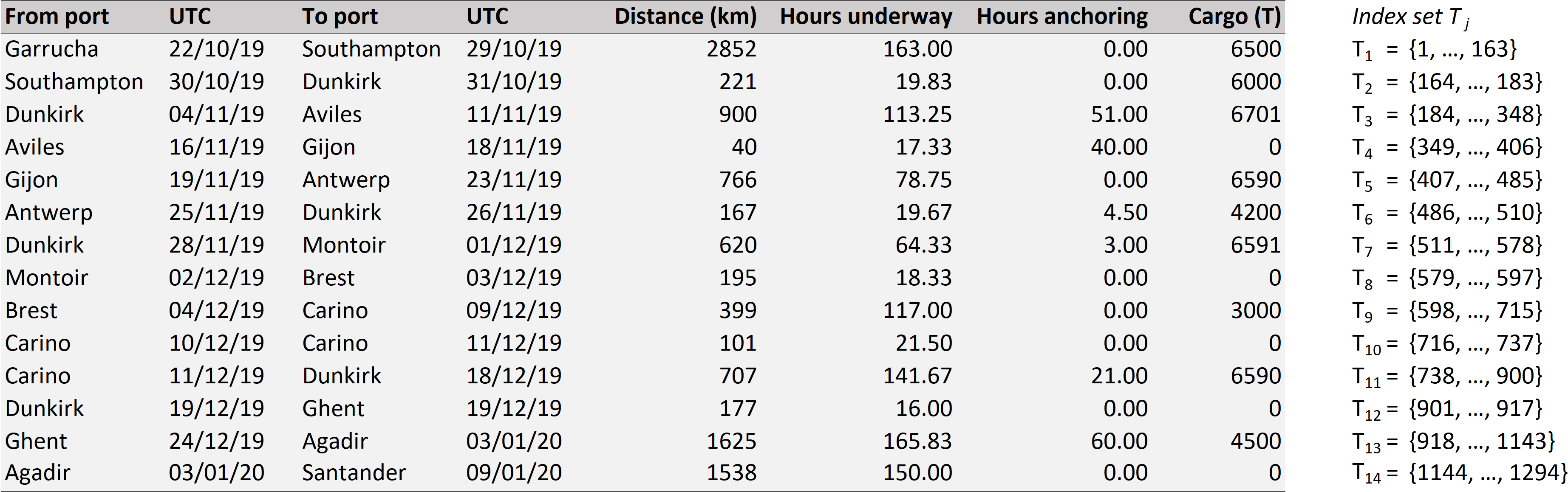}
    \caption{A snapshot of a typical itinerary of vessel $V_1$, based on anonymized real-world data; see Subsection~\ref{sec:ship_sensor_data}.
    Each row corresponds to a voyage~$j$ of our optimization model, and its port of departure is an opportunity to clean the hull.}
    \label{fig:typical_itinerary}
\end{figure}


We refer to the value $\bb_j \in \mathbb{R}^d_{+}$ of the biofouling measures at the first hour of voyage~$j$, after the cleaning decision $z_j$ has been taken, as the \emph{fouling state} of voyage~$j$.
During the voyage, the measures continue to accumulate, i.e., $\mathbf{s}_t = \bb_j + \bdelta_t$ for $t \in \mathcal{T}_j$, where $\bdelta_t \in \mathbb{R}^d_{+}$ is the accumulation since the start of voyage~$j$ and depends only on the operation of the vessel during this voyage.
For instance, if the measure \textit{HaS6} (hours at a speed in the interval 1-6 knots since the last cleaning) equals 500 at the start of voyage~$j$, then the hourly predictions use the values $500, 501, 502, \ldots$ while the vessel sails in this speed range.
However, if the vessel was not in this speed range, the values would remain constant e.g., $500, 500, 500, \ldots$.

We denote by $\bDelta_j \in \mathbb{R}^d_{+}$ the total accumulation from the start of voyage~$j$ to the start of voyage $j+1$, and by $\bar{\bb} \in \mathbb{R}^d_{+}$ the value of the measures at the start of voyage 1 if no cleaning takes place.
A cleaning resets a subset of the measures to zero, while the remaining measures are unaffected.
In our application, for example, an \gls{IWS} cleaning resets all measures related to in-water cleanings, whereas the days since the last dry-dock maintenance continue to increase.
This is encoded by the diagonal matrix $\mathbf{P} \in \{0,1\}^{d \times d}$, whose diagonal entry for a measure equals one if and only if this measure is reset by a cleaning.
Consequently, the fouling can be written in terms of the following recursion
\begin{equation}
\label{eq:fouling_dynamics}
    \bb_j = \bigl(\mathbf{I} - z_j \mathbf{P}\bigr)\bigl(\bb_{j-1} + \bDelta_{j-1}\bigr), \qquad j=1,\dots,n,
\end{equation}
where $\mathbf{I}$ denotes the identity matrix, $\bb_0 \coloneqq \mathbf{0}$, and $\bDelta_0 \coloneqq \bar{\bb}$.
That is, the measures simply accumulate if $z_j = 0$, while the measures selected by $\mathbf{P}$ are set to zero if $z_j = 1$.
If a cleaning resets all measures, i.e., $\mathbf{P} = \mathbf{I}$, then \eqref{eq:fouling_dynamics} yields $\bb_j = \mathbf{0}$ whenever $z_j = 1$.


Given the fouling state $\bb \in \mathbb{R}^d_{+}$ at the start of voyage~$j$, the predicted fuel costs of this voyage are
\begin{equation}
\label{eq:voyage_cost}
    F_j(\bb_j) \coloneqq \sum_{t \in \mathcal{T}_j} p_t \cdot g\bigl(\uu_t, \bb_j + \bdelta_t\bigr),
\end{equation}
where $p_t \geq 0$ denotes the fuel price at time point $t$.
Since the observations are hourly and $g$ predicts the \gls{FOC} in kg/h, each summand corresponds to the predicted fuel costs of one hour.
Costs of a voyage that do not depend on the cleaning schedule, such as port fees, only add a constant to the objective functions below and are therefore omitted.
Note that, for the historically realized cleaning schedule, $\bb_j + \bdelta_t$ coincides with the recorded biofouling measures $\mathbf{s}_t$, whereas for any other schedule, \eqref{eq:voyage_cost} evaluates $g$ with biofouling accumulation different to that which was observed in the historical record.
Table~\ref{tab:cso_notation} summarizes the notation of this section.


\begin{table}[ht]
    \scriptsize
    \centering
    \captionsetup{justification=centering}
    \caption{Notation of the cleaning schedule optimization.}
    \label{tab:cso_notation}
    \begin{tabular}{lp{0.68\linewidth}}
        \toprule
        Notation                                  & Description \\
        \midrule
        $j=1,\dots,n$                        & $n$ is the number of voyages and each voyage is indexed with a $j$. \\
        $t\in\mathcal{T}_j$                  & Hourly time indices within voyage~$j$ see Figure~\ref{fig:typical_itinerary}. \\
        $\mathbf{x}_t = (\uu_t, \mathbf{s}_t)$  & Ship profile at time point $t$, split into variables $\uu_t$ independent of the cleaning schedule and biofouling measures $\mathbf{s}_t$. \\
        $g$                              & Data-driven \gls{FOC} prediction function (kg). \\
        $f_j$                              & Cost function (\$) based on $\{0,1\}^j$ cleaning schedules used in \eqref{eq:cleaning}.\\
        $F_j$                              & Cost function (\$) based on fouling state $\boldsymbol{b}_j$ cleaning scheduled \eqref{eq:cleaning-2}.\\
        $z_j$, $\zz$, $\zz_{1:j}$               & Binary cleaning decision. \\
        $c_j$, $\cc$                            & Cost of a cleaning before voyage~$j$. \\
        $\bb_j$, $\bar{\bb}$                    & Fouling state at the start of voyage~$j$ and initial fouling state. \\
        $\bdelta_t$,                & Accumulation of the biofouling measures, hour by hour, up to time point $t$. \\
        $\bDelta_j$ & Total accumulation of the biofouling measures over voyage~$j$. \\
        $\mathbf{P}$                            & Diagonal 0-1 matrix that selects the biofouling measures reset by a cleaning. \\
        $\bLambda_j$, $\bb_j^{(\ell)}$          & Accumulated fouling without cleanings and fouling state of voyage~$j$ if the most recent cleaning was before voyage $\ell$. \\
        $\ell_j(\zz)$                           & Index of the most recent cleaning up to voyage~$j$ ($0$ if there was none). \\
        $\Phi_j(\ell)$, $\pi_j(\ell)$           & Optimal cost-to-go and optimal decision of the dynamic program. \\
        \bottomrule
    \end{tabular}
\end{table}

\subsection{General Model and its Complexity}
\label{sec:cso_general}


In the most general setting, the fuel costs of voyage~$j$ may depend on the entire cleaning history $\zz_{1:j}$, but not on the future decisions $z_{j+1}, \ldots, z_n$.
We describe these costs by arbitrary functions $f_j: \{0,1\}^j \rightarrow \mathbb{R}$, for $j=1,\dots,n$, where the data of voyage~$j$, e.g., its ship profiles, is part of the definition of $f_j$.
This leads to the following optimization problem, whose objective function sums up the fuel and cleaning costs.

\vspace{0.2cm}

\begin{defi}[Problem CLEANING-1]{defi:cleaning1}
    Given $n \in \mathbb{N}$, cleaning costs $\cc \in \mathbb{R}^n_{+}$, and fuel cost functions $f_j: \{0,1\}^j \rightarrow \mathbb{R}$ for $j=1,\dots,n$, we define the problem
    \begin{equation}
    \label{eq:cleaning}
        \underset{\zz \in \{0,1\}^n}{\text{minimize}} \quad J(\zz) \coloneqq \underbrace{\sum_{j=1}^{n} f_j(\zz_{1:j})}_{\text{Fuel Costs}} + \underbrace{\sum_{j=1}^{n} c_j \cdot z_j}_{\text{Cleaning Costs}}.
        \tag{CLEANING-1}
    \end{equation}
\end{defi}

\vspace{-0.2cm}


A global minimum of \eqref{eq:cleaning} corresponds to a cleaning schedule with the overall lowest costs and is therefore the most desirable outcome.
Although this formulation is general and mathematically well-posed, it is hard to solve for two reasons.
First, Theorem~\ref{theo:np_hard} shows that \eqref{eq:cleaning} is NP-hard, which implies that no polynomial-time algorithm can be expected to compute its global minimum unless P~$=$~NP.
For this statement, we consider instances in which each $f_j$ is specified by an algorithm that evaluates it in time polynomial in the size of the instance, as is the case for trained regression models.
Second, if $f_1, \ldots, f_n$ can only be accessed through function evaluations, as is the case for black-box \gls{ML} models, then Proposition~\ref{prop:oracle} in the Appendix shows that every exact algorithm requires an exponential number of evaluations, regardless of whether P~$=$~NP.
This highlights the approach's impracticality. 

\vspace{0.2cm}


\begin{theo}[NP-hardness]{theo:np_hard}
    The decision version of \eqref{eq:cleaning}, i.e., ``Given an instance of \eqref{eq:cleaning} and a threshold $C \in \mathbb{R}$, does there exist a cleaning schedule $\zz \in \{0,1\}^n$ with $J(\zz) \leq C$?'', is NP-hard.
    This holds even if $\cc = \mathbf{0}$ and each $f_j$ is nonnegative and can be evaluated in $\mathcal{O}(j)$ time.
\end{theo}


\begin{prf}[Theorem~\ref{theo:np_hard}]{prf:np_hard}
    We reduce the well-known 0-1 knapsack decision problem [\citet{Mathews1896, Dantzig1957, Kellerer2004}] to the decision version of \eqref{eq:cleaning}.
    Considering the decision version instead of the optimization problem is the standard approach to establish NP-hardness [\citet[Theorem~2]{Papadimitriou1994}].

    An instance of the 0-1 knapsack problem of size $n$ is described by enumerating a set of items $1,\dots,n$.
    Each item $i$ has a value $v_i \in \mathbb{N}$ and a weight $w_i \in \mathbb{N}$.
    There is a maximum weight $W \in \mathbb{N}$ and target value~$\bar{v} \in \mathbb{N}$.
    The decision variable $z_i \in \{0, 1\}$ is 1 if item $i$ is included in the knapsack and 0 otherwise.
    The bold notation
    $\ww \coloneqq \begin{bmatrix} w_1, \ldots, w_n \end{bmatrix}^\top$,
    $\vv \coloneqq \begin{bmatrix} v_1, \ldots, v_n \end{bmatrix}^\top$, and
    $\zz \coloneqq \begin{bmatrix} z_1, \ldots, z_n \end{bmatrix}^\top$
    is used to stack those weights, values and 0-1 decisions, respectively, into vectors according to the order of the $n$ items.
    The decision problem then asks, ``Does there exist a selection of items $\zz$, whose values total to at least $\bar{v}$ but whose weights are at most $W$ for an instance of the problem~\eqref{eq:knapsack}?''.
    The proof that this problem is NP-complete was first given by \citet[p.~100]{Karp1972} through a sequence of reductions.
    \begin{align}
    \label{eq:knapsack}
    \begin{split}
    \underset{\zz}{\text{maximize}} \quad & \sum_{i=1}^{n} v_i \cdot z_i
    \\
    \text{subject to} \quad
    & \sum_{i=1}^{n} w_i \cdot z_i \leq W,
    \\
    & \zz \in \{0,1\}^n.
    \end{split}
    \tag{KNAPSACK}
    \end{align}
    Given such an instance $\left(\vv, \ww, \bar{v}, W\right)$ of the 0-1 knapsack decision problem for $n$ items, we construct an instance of \eqref{eq:cleaning} with $n$ voyages, cleaning costs $\cc \coloneqq \mathbf{0}$, threshold $C \coloneqq \sum_{j=1}^{n} v_j - \bar{v}$, and fuel cost functions
    \begin{equation*}
        f_j(\zz_{1:j}) \coloneqq v_j \cdot (1 - z_j) +
        \begin{cases}
            M & \text{if } \sum_{i=1}^{j} w_i \cdot z_i > W, \\
            0 & \text{otherwise},
        \end{cases}
        \qquad j=1,\dots,n,
    \end{equation*}
    where $M \coloneqq \sum_{j=1}^{n} v_j + 1$.
    In words, a cleaning before voyage~$j$ saves the fuel costs $v_j$, and the penalty $M$ is incurred as soon as the accumulated weight of the cleanings exceeds the capacity.
    Each $f_j$ is nonnegative and can be evaluated in $\mathcal{O}(j)$ time, so the construction takes polynomial time.

    Note that the same vector $\zz \in \{0,1\}^n$ represents both a selection of items and a cleaning schedule, i.e., item $j$ is selected if and only if the hull is cleaned before voyage~$j$.
    Since $\ww \geq \mathbf{0}$, the partial sums $\sum_{i=1}^{j} w_i \cdot z_i$ are non-decreasing in $j$, so the penalty is incurred for at least one voyage if and only if $\ww^\top \zz > W$.
    Hence, for every $\zz \in \{0,1\}^n$, we obtain
    \begin{equation*}
        J(\zz) = \sum_{j=1}^{n} v_j - \vv^\top \zz \quad \text{if } \ww^\top \zz \leq W,
        \qquad \text{and} \qquad
        J(\zz) \geq M > C \quad \text{if } \ww^\top \zz > W,
    \end{equation*}
    where the last inequality uses $\bar{v} \geq 0$.
    Consequently, $J(\zz) \leq C$ holds if and only if $\ww^\top \zz \leq W$ and $\vv^\top \zz \geq \bar{v}$, i.e., if and only if the selection $\zz$ satisfies the knapsack decision problem.
    The knapsack instance is therefore a yes-instance if and only if the constructed instance of \eqref{eq:cleaning} is a yes-instance, and a polynomial-time algorithm for the decision version of \eqref{eq:cleaning} would solve the 0-1 knapsack decision problem in polynomial time.
\end{prf}


To assess the computational effort of exact methods, we treat $f_1, \ldots, f_n$ as black boxes and count the number of their evaluations, which dominate the computation time in our application, since each evaluation requires $|\mathcal{T}_j|$ predictions of the \gls{ML} model $g$.
The simplest exact method is the brute-force search shown in Algorithm~\ref{alg:brute_force}, which exhaustively evaluates the objective function for all possible schedules.


\begin{algorithm}[H]
    \caption{Brute-Force Search}
    \label{alg:brute_force}
    \KwInput{cleaning costs $\cc$ and fuel cost functions $f_1, \ldots, f_n$}
    $J^* \gets \infty$\;
    \For{$\zz \in \{0,1\}^n$}{
        \label{line:brute_for_loop}
        $J \gets \sum_{j=1}^{n} f_j(\zz_{1:j}) + \sum_{j=1}^{n} c_j \cdot z_j$\;
        \label{line:brute_obj_evaluation}
        \If{$J < J^*$}{
            $J^* \gets J$\;
            $\zz^* \gets \zz$\;
        }
    }
    \Return $\zz^*$, $J^*$
\end{algorithm}


Algorithm~\ref{alg:brute_force} requires exactly $n \cdot 2^n$ evaluations of $f_1, \ldots, f_n$.
To emphasize the impracticality of solving \eqref{eq:cleaning} exactly, Proposition~\ref{prop:oracle} in the Appendix shows that no algorithm can do substantially better: every deterministic algorithm that is guaranteed to find a global minimum must evaluate each $f_j$ at all of its $2^j$ arguments, i.e., it requires at least $\sum_{j=1}^{n} 2^j = 2^{n+1} - 2$ evaluations, and this bound is tight.
Hence, even a small instance with only 20 voyages requires more than two million evaluations of the fuel cost functions, each of which in turn requires the prediction of the \gls{FOC} for every hour of the voyage.
In contrast, the instances in Section~\ref{sec:numerical_experiments} comprise up to 125 voyages, which further highlights the impracticality.

\subsection{Fouling-State Model}
\label{sec:cso_fouling_state}


The intractability of the above program comes from the fact that $f_j$ is allowed to depend on the cleaning history in an arbitrary way.
By adding the restriction that the cleaning history only affects the cost of voyage~$j$ through the biofouling measures, i.e., through the variable~$\bb_j$ determined by~$\zz_{1:j}$ according to equation~\eqref{eq:fouling_dynamics}, the optimization problem becomes easier.
More precisely, setting $f_j(\zz_{1:j}) \coloneqq F_j(\bb_j)$ in \eqref{eq:cleaning} leads to the following problem:

\vspace{0.2cm}


\begin{defi}[Problem CLEANING-2]{defi:cleaning_2}
    Given $n \in \mathbb{N}$, cleaning costs $\cc \in \mathbb{R}^n_{+}$, an initial fouling state $\bar{\bb} \in \mathbb{R}^d_{+}$, accumulations $\bDelta_1, \ldots, \bDelta_{n-1} \in \mathbb{R}^d_{+}$, a diagonal reset matrix $\mathbf{P} \in \{0,1\}^{d \times d}$, and voyage fuel cost functions $F_j: \mathbb{R}^d_{+} \rightarrow \mathbb{R}$ for $j=1,\dots,n$, we define the problem
    \begin{align}
    \label{eq:cleaning-2}
    \begin{split}
        \underset{\zz,\, \bb_1, \ldots, \bb_n}{\text{minimize}} \quad & \underbrace{\sum_{j=1}^{n} F_j(\bb_j)}_{\text{Fuel Costs}} + \underbrace{\sum_{j=1}^{n} c_j \cdot z_j}_{\text{Cleaning Costs}}
        \\
        \text{subject to} \quad & \bb_j = \bigl(\mathbf{I} - z_j \mathbf{P}\bigr)\bigl(\bb_{j-1} + \bDelta_{j-1}\bigr), \quad j=1,\dots,n,
        \\
        & \zz \in \{0,1\}^n,
    \end{split}
    \tag{CLEANING-2}
    \end{align}
    where $\bb_0 \coloneqq \mathbf{0}$ and $\bDelta_0 \coloneqq \bar{\bb}$.
\end{defi}

\vspace{-0.2cm}


Since the constraints of \eqref{eq:cleaning-2} uniquely determine $\bb_1, \ldots, \bb_n$ for every $\zz$, the fouling states are auxiliary state variables rather than independent decisions, and \eqref{eq:cleaning-2} is the special case of \eqref{eq:cleaning} with $f_j(\zz_{1:j}) = F_j(\bb_j)$.
In our application, $F_j$ is given by \eqref{eq:voyage_cost}, and it is reasonable to expect that $F_j$ is non-decreasing in the fouling state, i.e., that the same voyage does not become cheaper with more biofouling.
However, none of the following results requires this or any other property of $F_j$, so that arbitrary \gls{ML} models can be used to represent $g$.


The key property of \eqref{eq:cleaning-2} is that the fouling state of a voyage depends on the cleaning history only through the time point of the most recent cleaning.
To formalize this property, let $\bLambda_0 \coloneqq \mathbf{0}$ and $\bLambda_j \coloneqq \sum_{i=0}^{j-1} \bDelta_i$ for $j=1,\dots,n$ denote the accumulated fouling at the start of voyage~$j$ if the hull has never been cleaned.
For $0 \leq \ell \leq j \leq n$, we then define
\begin{equation}
\label{eq:state_last_cleaning}
    \bb_j^{(\ell)} \coloneqq \bLambda_j - \mathbf{P} \bLambda_\ell
\end{equation}
as the fouling state of voyage~$j$ if the most recent cleaning took place before voyage $\ell \geq 1$, or if no cleaning has taken place at all ($\ell = 0$).
Indeed, since the cleaning before voyage $\ell$, the reset measures have accumulated $\mathbf{P}(\bLambda_j - \bLambda_\ell)$, while the remaining measures have accumulated $(\mathbf{I} - \mathbf{P})\bLambda_j$ since the start of the planning horizon.
In particular, $\bb_j^{(j)} = (\mathbf{I} - \mathbf{P})\bLambda_j$ is the fouling state of voyage~$j$ directly after a cleaning.
For a schedule $\zz \in \{0,1\}^n$ and $j \in \{0, \ldots, n\}$, we further denote the index of the most recent cleaning up to voyage~$j$ by
\begin{equation*}
    \ell_j(\zz) \coloneqq \max\bigl(\{0\} \cup \{i=1,\dots j: z_i = 1\}\bigr),
\end{equation*}
where $\ell_j(\zz) = 0$ if no cleaning has taken place.
Proposition~\ref{prop:bellman}(i), in the following subsection, shows that every feasible solution of \eqref{eq:cleaning-2} satisfies $\bb_j = \bb_j^{(\ell_j(\zz))}$.

\subsection{Dynamic Programming Solution}
\label{sec:cso_dp}


Dynamic programming solves an optimization problem by recursively decomposing it into simpler sub-problems, whose solutions can be re-combined to find the overall solution [\citet[Chapter~6]{Dasgupta2006}; \citet[Chapter~14]{Cormen2022}].
This approach relies on the \emph{principle of optimality}, which requires that the remaining decisions of an optimal solution are optimal for the sub-problem that remains after the first decisions have been taken [\citet{Bellman1957}; \citet[Chapter~1]{Bertsekas2017}].
For \eqref{eq:cleaning}, the sub-problem that remains after the decisions $\zz_{1:j}$ have been taken depends on the entire history $\zz_{1:j}$, of which there are $2^j$, as we see in Proposition~\ref{prop:oracle}.
In contrast, by \eqref{eq:state_last_cleaning}, the sub-problem of \eqref{eq:cleaning-2} that remains before voyage~$j$ depends only on the index $\ell \in \{0, \ldots, j-1\}$ of the most recent cleaning, of which there are only~$j$.
Accordingly, for $1 \leq j \leq n+1$ and $0 \leq \ell < j$, we define the \emph{optimal cost-to-go}
\begin{equation}
\label{eq:cost_to_go}
    \Phi_j(\ell) \coloneqq \min_{\substack{\zz \in \{0,1\}^n: \\ \ell_{j-1}(\zz) = \ell}} \ \sum_{i=j}^{n} \Bigl( F_i\bigl(\bb_i^{(\ell_i(\zz))}\bigr) + c_i \cdot z_i \Bigr)
\end{equation}
as the minimal costs of the voyages $j, \ldots, n$, given that the most recent cleaning predating voyage~$j$ took place before voyage $\ell$, or never if $\ell = 0$.
For $j = n+1$, the sum is empty, so that $\Phi_{n+1}(\ell) = 0$.
Proposition~\ref{prop:bellman} shows that these values satisfy a recursion, known as the Bellman equation, from which an optimal cleaning schedule can be constructed.

\vspace{0.2cm}


\begin{prop}[Bellman equation]{prop:bellman}
    Consider an instance of \eqref{eq:cleaning-2}.
    \begin{enumerate}[label=(\roman*), noitemsep, topsep=0pt]
        \item For every $\zz \in \{0,1\}^n$ and $j=1,\dots,n$, the constraints of \eqref{eq:cleaning-2} yield \[\bb_j = \bb_j^{(\ell_j(\zz))}.\]
        \item The optimal value of \eqref{eq:cleaning-2} equals $\Phi_1(0)$, and for all $1 \leq j \leq n$ and $0 \leq \ell < j$, it holds that
        \begin{equation}
        \label{eq:bellman}
            \Phi_j(\ell) = \min \Bigl\{ \underbrace{c_j + F_j\bigl(\bb_j^{(j)}\bigr) + \Phi_{j+1}(j)}_{\text{clean before voyage } j}, \ \underbrace{F_j\bigl(\bb_j^{(\ell)}\bigr) + \Phi_{j+1}(\ell)}_{\text{do not clean}} \Bigr\},
        \end{equation}
        where $\Phi_{n+1}(\ell) = 0$ for $0 \leq \ell \leq n$.
        \item Let $\pi_j(\ell) \coloneqq 1$ if the first term in \eqref{eq:bellman} is less than or equal to the second term, and $\pi_j(\ell) \coloneqq 0$ otherwise.
        Define $\hat{\ell}_0 \coloneqq 0$ and, for $j = 1, \ldots, n$, set $\hat{z}_j \coloneqq \pi_j(\hat{\ell}_{j-1})$, as well as $\hat{\ell}_j \coloneqq j$ if $\hat{z}_j = 1$ and $\hat{\ell}_j \coloneqq \hat{\ell}_{j-1}$ otherwise.
        Then, $\hat{\zz}$ is an optimal solution of \eqref{eq:cleaning-2}.
    \end{enumerate}
\end{prop}


\begin{prf}[Proposition~\ref{prop:bellman}]{prf:bellman}
    (i) We prove the claim by induction on $j$, starting from $j = 0$.
    Since $\ell_0(\zz) = 0$ and $\bLambda_0 = \mathbf{0}$, we have $\bb_0 = \mathbf{0} = \bb_0^{(0)}$.
    For $j \geq 1$, abbreviate $\ell \coloneqq \ell_{j-1}(\zz)$ for conciseness, and assume the induction hypothesis $\bb_{j-1} = \bb_{j-1}^{(\ell)}$.
    Since $\bLambda_j = \bLambda_{j-1} + \bDelta_{j-1}$, it follows that
    \[
        \bb_{j-1} + \bDelta_{j-1} = \bLambda_{j-1} - \mathbf{P}\bLambda_\ell + \bDelta_{j-1} = \bLambda_j - \mathbf{P}\bLambda_\ell = \bb_j^{(\ell)}.
    \]
    If $z_j = 0$, then $\ell_j(\zz) = \ell$, and \eqref{eq:fouling_dynamics} gives $\bb_j = \bb_j^{(\ell)}$.
    If $z_j = 1$, then $\ell_j(\zz) = j$, and since $(\mathbf{I} - \mathbf{P})\mathbf{P} = \mathbf{0}$ for the diagonal 0-1 matrix $\mathbf{P}$, \eqref{eq:fouling_dynamics} gives $\bb_j = (\mathbf{I} - \mathbf{P})(\bLambda_j - \mathbf{P}\bLambda_\ell) = (\mathbf{I} - \mathbf{P})\bLambda_j = \bb_j^{(j)}$.
    In both cases, $\bb_j = \bb_j^{(\ell_j(\zz))}$, which completes the induction.

    (ii) By (i), the objective value of every $\zz \in \{0,1\}^n$ in \eqref{eq:cleaning-2} equals
    \begin{equation*}
        \sum_{i=1}^{n} \Bigl( F_i\bigl(\bb_i^{(\ell_i(\zz))}\bigr) + c_i \cdot z_i \Bigr).
    \end{equation*}
    Since $\ell_0(\zz) = 0$ for all $\zz$, minimizing over $\zz \in \{0,1\}^n$ yields $\Phi_1(0)$ by \eqref{eq:cost_to_go}.
    Now, fix $1 \leq j \leq n$ and $0 \leq \ell < j$, and let $\zz$ be a schedule with $\ell_{j-1}(\zz) = \ell$.
    For $i > j$, the index $\ell_i(\zz)$ depends only on $\ell_j(\zz)$ and $z_{j+1}, \ldots, z_i$.
    Hence, the costs $\sum_{i=j+1}^{n} \bigl( F_i(\bb_i^{(\ell_i(\zz))}) + c_i \cdot z_i \bigr)$ of the voyages after $j$ depend only on $\ell_j(\zz)$ and on $z_{j+1}, \ldots, z_n$, which can be chosen freely, so that their minimum is $\Phi_{j+1}(\ell_j(\zz))$ by \eqref{eq:cost_to_go}.
    If $z_j = 1$, then $\ell_j(\zz) = j$, and the costs of voyage~$j$ are $c_j + F_j(\bb_j^{(j)})$.
    If $z_j = 0$, then $\ell_j(\zz) = \ell$, and the costs of voyage~$j$ are $F_j(\bb_j^{(\ell)})$.
    Minimizing first over $z_{j+1}, \ldots, z_n$ and then over $z_j \in \{0,1\}$ yields \eqref{eq:bellman}.

    (iii) By construction, $\hat{\ell}_j = \ell_j(\hat{\zz})$ for all $j$, and $\hat{z}_j$ attains the minimum in \eqref{eq:bellman} for $\ell = \hat{\ell}_{j-1}$, i.e.,
    \begin{equation*}
        \Phi_j(\hat{\ell}_{j-1}) = F_j\bigl(\bb_j^{(\hat{\ell}_j)}\bigr) + c_j \cdot \hat{z}_j + \Phi_{j+1}(\hat{\ell}_j), \qquad j=1,\dots,n.
    \end{equation*}
    Summing these equations over $j=1,\dots,n$ and using (i) as well as $\Phi_{n+1}(\hat{\ell}_n) = 0$, the objective value of $\hat{\zz}$ equals $\sum_{j=1}^{n} \bigl( F_j(\bb_j^{(\hat{\ell}_j)}) + c_j \cdot \hat{z}_j \bigr) = \Phi_1(\hat{\ell}_0) - \Phi_{n+1}(\hat{\ell}_n) = \Phi_1(0)$, which is the optimal value by (ii).
\end{prf}


Proposition~\ref{prop:bellman} has an intuitive graph interpretation, which is illustrated in Figure~\ref{fig:dp_state_graph}.
Each pair $(j, \ell)$ with $0 \leq \ell < j \leq n+1$ is a node of a directed acyclic graph, and each node with $j \leq n$ has two outgoing arcs, one for each decision $z_j \in \{0,1\}$, whose lengths are the resulting costs of voyage~$j$.
Every path from node $(1, 0)$ to the end node corresponds to exactly one cleaning schedule, and its length equals the objective value of this schedule.
Hence, $\Phi_j(\ell)$ is the length of a shortest path from node $(j, \ell)$ to the end node, and \eqref{eq:bellman} is the well-known shortest-path recursion on acyclic graphs.
Shortest-path algorithms that rely on the same principle of optimality, such as Dijkstra's algorithm and the A* algorithm, are established tools in maritime applications, e.g., to optimize weather routing and speed profiles
[
    \citet{Fagerholt2010};
    \citet{Zis2020};
    \citet{Huotari2021}
].
Since the graph in Figure~\ref{fig:dp_state_graph} is acyclic and its nodes are naturally ordered by voyage, a single backward pass over the nodes suffices to compute all values $\Phi_j(\ell)$.
Structurally, \eqref{eq:cleaning-2} resembles the classical equipment replacement problem [\citet{Bellman1955}], in which a costly renewal resets the deterioration of a system.
The distinguishing features of \eqref{eq:cleaning-2} are that the voyage costs are not given in closed form but predicted from data by an \gls{ML} model, and that the fouling state is a vector of biofouling measures that is only partially reset by a cleaning.


\begin{figure}[ht]
    \centering

%
\begin{tikzpicture}[
    x=2.4cm, y=-1.3cm,
    state/.style={circle, draw=black!55, fill=white, minimum size=8mm, inner sep=0pt, font=\scriptsize},
    sink/.style={state, fill=black!8},
    keep/.style={-{Stealth[length=1.5mm]}, draw=black!35},
    clean/.style={keep, dashed},
    final/.style={keep, dotted},
    optimal/.style={-{Stealth[length=2mm]}, draw=blue, very thick},
    cost/.style={font=\scriptsize, text=blue!70!black, fill=white, inner sep=1pt},
    legendtext/.style={font=\scriptsize, anchor=west, inner sep=2pt},
]

    \foreach \jdx in {1,...,5}{
        \node[font=\footnotesize] at (\jdx, -0.6) {$j = \jdx$};
    }
    \foreach \ldx in {0,...,4}{
        \node[font=\footnotesize, anchor=east] at (0.6, \ldx) {$\ell = \ldx$};
    }

    \foreach \jdx in {1,...,5}{
        \pgfmathtruncatemacro{\lmax}{\jdx - 1}
        \foreach \ldx in {0,...,\lmax}{
            \node[state] (s\jdx\ldx) at (\jdx, \ldx) {$\jdx,\ldx$};
        }
    }
    \node[sink] (end) at (6, 2) {end};

    \foreach \jdx in {1,...,4}{
        \pgfmathtruncatemacro{\lmax}{\jdx - 1}
        \pgfmathtruncatemacro{\jnext}{\jdx + 1}
        \foreach \ldx in {0,...,\lmax}{
            \draw[keep] (s\jdx\ldx) -- (s\jnext\ldx);
            \draw[clean] (s\jdx\ldx) -- (s\jnext\jdx);
        }
    }
    \foreach \ldx in {0,...,4}{
        \draw[final] (s5\ldx) -- (end);
    }

    \draw[optimal] (s10) -- node[cost, above=1pt] {$F_1(\bb_1^{(0)})$} (s20);
    \draw[optimal] (s20) -- node[cost, right=4pt, pos=0.22, align=left] {$c_2 \, +$\\$F_2(\bb_2^{(2)})$} (s32);
    \draw[optimal] (s32) -- node[cost, above=1pt] {$F_3(\bb_3^{(2)})$} (s42);
    \draw[optimal] (s42) -- node[cost, above=1pt] {$F_4(\bb_4^{(2)})$} (s52);
    \draw[optimal, dotted] (s52) -- (end);
    \foreach \nd in {s10, s20, s32, s42, s52, end}{
        \draw[blue, very thick] (\nd) circle (4mm);
    }

    \draw[keep] (0.8, 4.8) -- (1.1, 4.8);
    \node[legendtext] at (1.1, 4.8) {Do not clean ($z_j = 0$): cost $F_j\left(\bb_j^{(\ell)}\right)$};
    \draw[clean] (0.8, 5.12) -- (1.1, 5.28);
    \node[legendtext] at (1.1, 5.2) {Clean ($z_j = 1$): cost $c_j + F_j\left(\bb_j^{(j)}\right)$};
    \draw[optimal] (3.8, 4.74) -- (3.95, 4.74) -- (4.05, 4.86) -- (4.25, 4.86);
    \node[legendtext, text=black] at (4.25, 4.8) {Optimal schedule $\hat{\zz} = [0, 1, 0, 0]^\top$};

\end{tikzpicture}
    \caption{State graph of the dynamic program in Proposition~\ref{prop:bellman} for $n = 4$ voyages.
    Node $(j, \ell)$ represents the situation before voyage~$j$ if the most recent cleaning took place before voyage $\ell$ ($\ell = 0$: no cleaning so far).
    Its two outgoing arcs correspond to the decisions $z_j \in \{0,1\}$.
    Every path from node $(1, 0)$ to the end node corresponds to one cleaning schedule, and its length equals the objective value of \eqref{eq:cleaning-2}.
    The optimal cost-to-go $\Phi_j(\ell)$ is the length of a shortest path from node $(j, \ell)$ to the end node.
    The highlighted path is a shortest path from node $(1, 0)$ and thus corresponds to an optimal schedule, which here cleans the hull only before voyage~2.}
    \label{fig:dp_state_graph}
\end{figure}


Algorithm~\ref{alg:dynamic_programming} implements Proposition~\ref{prop:bellman}.
Its backward pass starts at Line~\ref{line:outer_loop} and works backwards through the voyages $j = n, \ldots, 1$.
For each voyage, the costs of a cleaning before voyage~$j$ are computed once in Line~\ref{line:m_clean}, since they do not depend on $\ell$.
The inner loop starting at Line~\ref{line:inner_loop} then computes the costs of not cleaning in Line~\ref{line:m_not} for each possible index $\ell$ of the most recent cleaning, and stores the optimal cost-to-go $\Phi_j(\ell)$ together with the optimal decision $\pi_j(\ell)$.
In this manner, the previously computed values $\Phi_{j+1}(\cdot)$ are reused instead of re-evaluating the costs of all future voyages.
Finally, the forward pass starting at Line~\ref{line:forward_loop} recovers an optimal schedule as described in Proposition~\ref{prop:bellman}(iii).
Since only the decisions $\pi_j(\ell)$ are stored rather than complete schedules, the memory requirement is $\mathcal{O}(n^2)$.


\begin{algorithm}[H]
    \caption{Dynamic \acrfull{CSO}}
    \label{alg:dynamic_programming}
    \KwInput{cleaning costs $\cc$, initial fouling state $\bar{\bb}$, accumulations $\bDelta_1, \ldots, \bDelta_{n-1}$, reset matrix $\mathbf{P}$, and voyage fuel cost functions $F_1, \ldots, F_n$}
    $\bLambda_0 \gets \mathbf{0}$, \ $\bLambda_1 \gets \bar{\bb}$, and $\bLambda_j \gets \bLambda_{j-1} + \bDelta_{j-1}$ for $j = 2, \ldots, n$\;
    $\Phi_{n+1}(\ell) \gets 0$ for $\ell = 0, \ldots, n$\;
    \tcp{Backward pass: Bellman equation \eqref{eq:bellman}}
    \For{$j = n, \ldots, 1$}{
        \label{line:outer_loop}
        $\kappa \gets c_j + F_j\bigl(\bLambda_j - \mathbf{P}\bLambda_j\bigr) + \Phi_{j+1}(j)$ \tcp*{clean before voyage~$j$}
        \label{line:m_clean}
        \For{$\ell = 0, \ldots, j-1$}{
            \label{line:inner_loop}
            $\nu \gets F_j\bigl(\bLambda_j - \mathbf{P}\bLambda_\ell\bigr) + \Phi_{j+1}(\ell)$ \tcp*{do not clean before voyage~$j$}
            \label{line:m_not}
            $\Phi_j(\ell) \gets \min\{\kappa, \nu\}$\;
            \leIf{$\kappa \leq \nu$}{$\pi_j(\ell) \gets 1$}{$\pi_j(\ell) \gets 0$}
        }
    }
    \tcp{Forward pass: recover an optimal schedule}
    $\ell \gets 0$\;
    \For{$j = 1, \ldots, n$}{
        \label{line:forward_loop}
        $\hat{z}_j \gets \pi_j(\ell)$\;
        \lIf{$\hat{z}_j = 1$}{$\ell \gets j$}
    }
    \Return $\hat{\zz}$, $\Phi_1(0)$
\end{algorithm}


\begin{prop}[Correctness and complexity]{prop:dp_complexity}
    Algorithm~\ref{alg:dynamic_programming} returns an optimal solution $\hat{\zz}$ and the optimal value $\Phi_1(0)$ of \eqref{eq:cleaning-2}.
    It evaluates the voyage fuel cost functions $F_1, \ldots, F_n$ exactly $n(n+3)/2$ times.
    If $F_j$ is given by \eqref{eq:voyage_cost}, this corresponds to $\sum_{j=1}^{n} (j+1) \cdot |\mathcal{T}_j| \leq (n+1) \cdot \sum_{j=1}^{n} |\mathcal{T}_j|$ predictions of $g$.
\end{prop}


\begin{prf}[Proposition~\ref{prop:dp_complexity}]{prf:dp_complexity}
    The backward pass computes $\Phi_j(\ell)$ and $\pi_j(\ell)$ for $j = n, \ldots, 1$ and $0 \leq \ell < j$ according to \eqref{eq:bellman}, where $\bb_j^{(\ell)} = \bLambda_j - \mathbf{P}\bLambda_\ell$ by \eqref{eq:state_last_cleaning}.
    Since $j$ decreases, the values $\Phi_{j+1}(\cdot)$ are available whenever they are needed.
    The forward pass is the construction of Proposition~\ref{prop:bellman}(iii), so that $\hat{\zz}$ is optimal and $\Phi_1(0)$ is the optimal value by Proposition~\ref{prop:bellman}(ii).
    In iteration $j$ of the backward pass, $F_j$ is evaluated once in Line~\ref{line:m_clean} and $j$ times in Line~\ref{line:m_not}, which amounts to $\sum_{j=1}^{n} (j+1) = n(n+3)/2$ evaluations in total.
    Each evaluation of $F_j$ via \eqref{eq:voyage_cost} requires $|\mathcal{T}_j|$ predictions of $g$, which proves the last claim.
\end{prf}


For the instances in Section~\ref{sec:numerical_experiments} with up to 125 voyages, Algorithm~\ref{alg:dynamic_programming} therefore requires at most 8,000 evaluations of the voyage fuel costs, whereas Algorithm~\ref{alg:brute_force} would require $125 \cdot 2^{125} > 10^{39}$ evaluations.
Moreover, the $j + 1$ evaluations of $F_j$ in iteration $j$ are independent of each other, so that the corresponding predictions of $g$ can be computed in a single batch.
Overall, this procedure provides a practical way to solve \eqref{eq:cleaning-2}, as demonstrated in Section~\ref{sec:numerical_experiments}, where we apply the optimization framework to real-world datasets.

\section{Empirical Data Sources}
\label{sec:empirical_data}

After proposing a mathematical framework for optimizing the cleaning schedule of ships, this section introduces the empirical data used to evaluate its performance. 
Specifically, the empirical analysis relies on real-world vessel data obtained through a cooperation agreement with \href{https://carisbrooke.co/}{Carisbrooke Shipping Ltd.~\faExternalLink}, which is a UK-based shipping company that operates a worldwide-trading  fleet of over 27 modern dry cargo and multi-purpose vessels.
As part of this agreement, the industry partner granted us access to hourly voyage data for ten vessels from 2016 to 2025, which we refer to as $V_1,\dots,V_{10}$.  
While vessel $V_8$ was excluded from the empirical evaluation due to data irregularities, each of the remaining nine ships provides approximately $80{,}000$ observations. 
As these vessels do not operate based on a repetitive schedule and can be chartered by different customers, important ship characteristics like the routes, the loads and the time spent in harbor can vary greatly. 
This also affects the hull fouling, which develops differently over time for each vessel.
Thus, the results of our models are likely less comparable to studies based on ship data with fixed schedules, e.g., ferries. 
In the upcoming subsections, the company data is separated into two sources and described in more detail. 
Afterwards, we introduce relevant weather and marine variables, which were obtained from Open-Meteo~[\citet{Zippenfenig2023}].
This subsection also includes non-linear variable transformations that integrate company and weather data.

\subsection{Cleaning Reports}\label{sec:cleaning_reports}


The first data source consists of the cleaning reports, which are generated within the company to document the cleaning procedures carried out on each vessel.
Based on these reports, two types of manual cleaning events can be distinguished: \gls{IWS} and \gls{DDM}.
As explained in Subsection~\ref{sec:hull_cleaning}, divers are deployed during an \gls{IWS} to photograph and inspect the vessel, remove major biofouling, and perform necessary repairs.
In contrast, \gls{DDM} is a far more involved process during which the vessel is docked for several days. 
In this period, the ship's entire hull is grit-blasted to abrasively remove biofouling as well as any remaining paint [\citet[p. 249]{HUA2018}]. 
Once the hull has been cleaned, any necessary repairs are carried out and the surface is subsequently recoated [\citet[p. 3] {Adland2018}]. 
Figure~\ref{fig:after_dry_dock} shows one of the ten vessels after having been cleaned using this procedure.
Compared to an \gls{IWS}, a \gls{DDM} is more costly and time-consuming, but it is typically also more effective in removing biofouling [\citet[p. 5]{VALCHEV2022}].
For our empirical data, we use the cleaning records to construct two variables for each vessel: \gls{DSIWS} and \gls{DSDDM}.
Additionally, we create the variable \gls{DSC}, which combines \gls{DSIWS} and \gls{DSDDM} by counting the days since the most recent cleaning event.


\begin{figure}[H]
    \centering
    \includegraphics[width=0.4\linewidth]{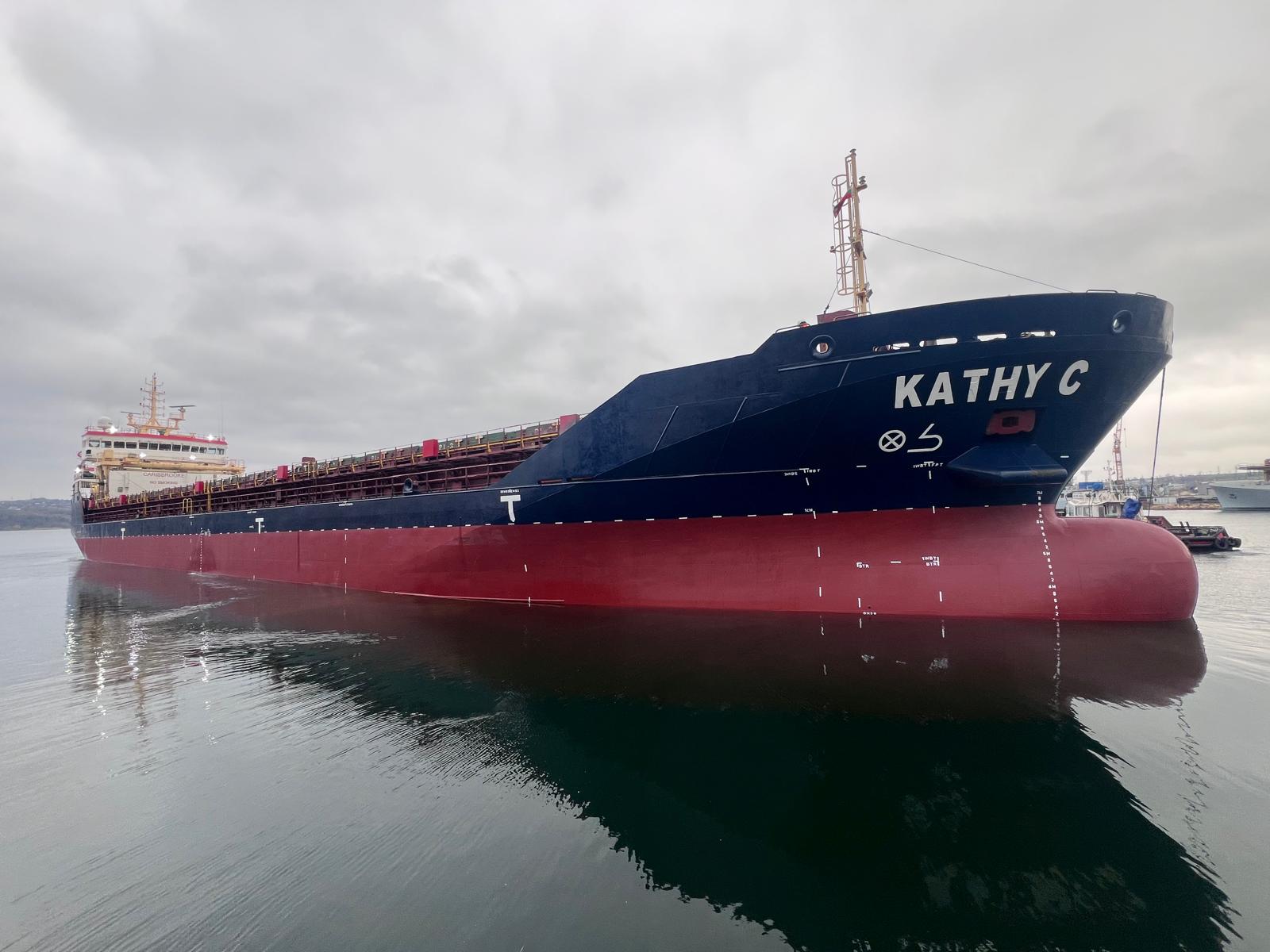}
    \caption{One of the ten sister ships included in the study, photographed immediately after \gls{DDM} [Carisbrooke Shipping Ltd (2025)].}
    \label{fig:after_dry_dock}
\end{figure}

\subsection{Ship Sensor Data}\label{sec:ship_sensor_data}


Nowadays, all ships of 5,000 gross tonnage and above must record fuel consumption data, broken down by fuel type, along with other specified information due to Regulation 22A of \citet[pp. 4--5]{IMO2016}. 
To comply with this requirement, our industry partner has installed a Marine Fuel Efficiency Monitoring System on each vessel. 
This system is our second data source and enables the tracking of the \gls{FOC} and several related variables, such as the \gls{STW}, \gls{SOG}, \gls{WUK} and \gls{COG}, on an hourly basis using multiple onboard sensors. 
Table \ref{tab:ship_sensor_data} provides a brief description of all relevant variables derived from this data source and specifies their respective data types. 
To link the ship sensor data to the variables from Subsection~\ref{sec:cleaning_reports}, we match the \textit{Date \& Time} entries with the time stamps of the cleaning records.


\begin{table}[H]
    \scriptsize
    \centering
    \captionsetup{justification=centering}
    \caption{Description of Ship Sensor Data Variables.}
    \label{tab:ship_sensor_data}
        \begin{tabular}{lll}
            \toprule
            Variable        & Type    (Unit) &  Description \\  \midrule
            Voyage ID       & Integer        & Unique identifier of the voyage per vessel.  \\ 
            IMO Number      & Integer        & A seven-digit integer unique to each ship.  \\ 
            Date \& Time    & Time Stamp     & The time stamp when the entry was created.   \\
            Position        & String         & Geographical coordinates of the ship when the entry was created.  \\
            \gls{FOC}       & Float    (kg/h)  & Total \gls{FOC} measured by flow meters on the fuel pipes. \\
            \gls{SOG}       & Float    (kn)    &  The average speed relative to the surface of the earth in knots.    \\
            \gls{STW}       & Float    (kn)    &  The average speed with respect to the water in knots.     \\
            Cargo Loaded    & Float    (tons)  &  Weight of cargo loaded onto the vessel.    \\
            Draught         & Float    (m)     &  Average depth of the vessel below the waterline in meters.  \\
            Trim            & Float    (m)     &  Average difference between the aft and the forward draft in meters.          \\
            \gls{WUK}       & Float    (m)     &  Average space between the ship's bottom and the seabed in meters. \\
            \gls{COG}       & Float    (\textdegree)  &  The actual direction of progress w.r.t. the earth surface in degrees.   \\
            Wind Speed      & Float    (m/s)          & The average speed of the wind in meters per second.        \\
            Wind Direction  & Float    (\textdegree)  & The direction of the wind in degrees.           \\
            Fuel Oil Type   & String         & The used fuel oil: HFO, LFO, LSHFO or MGO.        \\
            Propeller Shaft & Float    (rpm)  & Propeller shaft speed in rotations per minute.\\
            Pitch Propeller & Float    (\%)  & Propeller angle as percentage of maximum angle.\\
            \bottomrule
        \end{tabular}
\end{table}

\vspace{-0.2cm}


A subset of the features listed in Table \ref{tab:ship_sensor_data} are categorical or represent a time stamp. 
Since most \gls{ML} algorithms require numerical input, we transform these features using common preprocessing techniques. 
In particular, the \textit{Fuel Oil Type} is encoded based on four indicator variables, \gls{LFO}, \gls{LSHFO} and \gls{MGO}, with \gls{HFO} serving as the reference category. 
Similarly, season-specific indicator variables are generated from the \textit{Date \& Time} entries, where the winter months serve as a reference. 
An overview of the transformations is provided by Table~\ref{tab:ship_sensor_data_2}.
This summary also includes several additional variables based on \gls{STW} that are designed to partially capture the effect of hull fouling. 
Specifically, these variables measure the hours spent in certain speed ranges since the most recent cleaning in order to account for self-polishing effects.
The feature \textit{HaS0} is inspired by the findings of \citet{deHaas2023}, who showed that the number of anchorage days since the last vessel cleaning provides an alternative way to capture the hull fouling effect.  
Moreover, since we observe that the ship sensor data occasionally contains longer gaps between hourly observations due to maintenance work or other factors, we also create a feature named \textit{HU}, which counts these unaccounted hours. 
Ultimately, we also add lagged versions of the \gls{SOG} and \gls{STW} variables as additional features to the dataset, which simply take the corresponding value from the previous time index. 
This allows the models to better capture an acceleration/deceleration process and may also make the dataset more robust against one-off faulty readings.


\begin{table}[H]
    \scriptsize
    \centering
    \captionsetup{justification=centering}
    \caption{Transformations of Non-Numeric Variables.}
    \label{tab:ship_sensor_data_2}
        \begin{tabular}{lll}
            \toprule
            Variable & Type (Unit) &  Description \\  \midrule
            \gls{LFO}       & Integer & Indicator that is ``1'' if \gls{LFO}  is used, otherwise ``0''. \\
            \gls{LSHFO}     & Integer & Indicator that is ``1'' if \gls{LSHFO} is used, otherwise ``0''.\\
            \gls{MGO}       & Integer & Indicator that is ``1'' if \gls{MGO} is used, otherwise ``0''. \\
            Spring    & Integer & Indicator that is ``1'' if it is spring, otherwise ``0''. \\
            Summer    & Integer & Indicator that is ``1'' if it is summer, otherwise ``0''. \\
            Autumn    & Integer & Indicator that is ``1'' if it is autumn, otherwise ``0''. \\
            HU        & Float (h) & The unaccounted hours since cleaning.\\
            HaS0      & Float (h) & The hours at speed $[0, 1]$ knots since the last cleaning. \\
            HaS6      & Float (h) & The hours at speed $(1, 6]$ knots since the last cleaning.\\
            HaS9      & Float (h)& The hours at speed $(6, 9]$ knots since the last cleaning. \\
            HaS12     & Float (h) & The hours at speed $(9, \infty)$ knots since the last cleaning. \\
            \bottomrule
        \end{tabular}
\end{table}


Beyond the transformation of variables, we also perform common data quality and plausibility checks for the ship sensor data. 
This includes checking for missing or implausible values and, if necessary, performing an imputation. 
To limit the impact of extreme observations, we additionally apply winsorizing to selected variables. 
A first impression of the preprocessed data  is provided by Figure~\ref{fig:FOC_per_speed_and_DSC}, which illustrates the effect of the discretized \gls{DSDDM} on the \gls{FOC} through boxplots based on the data of vessels $V_1, V_3$ and $V_7$. 
Overall, the cleaned data for the selected vessels already indicates that the median \gls{FOC} increases as more time has passed since the last cleaning.
However, it is worth mentioning that such consistent plots cannot be produced for all vessels and we shall use more advanced techniques to properly model this relationship.


\begin{figure}[H] 
    \centering
    \includegraphics[width = \linewidth]{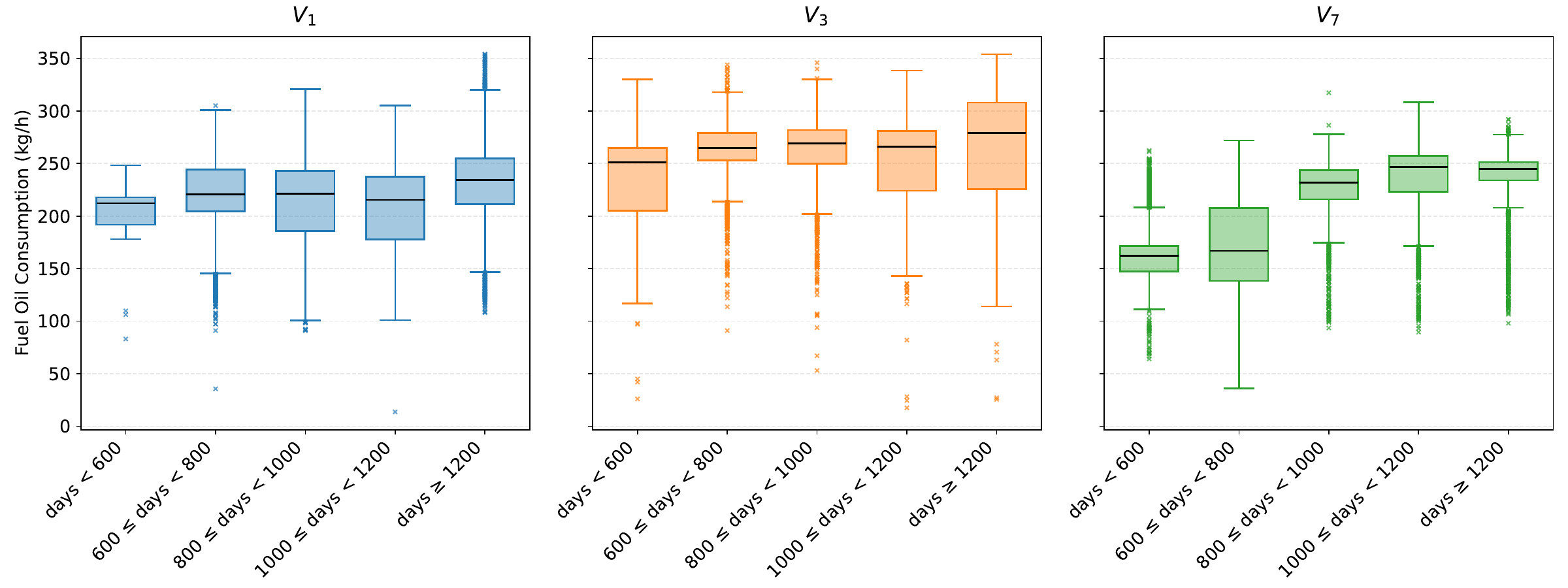}
    \caption{Three examples (left to right vessels $V_1$, $V_3$ and $V_7$) where the influence of the independent variable \acrfull{DSDDM} can be observed against the \acrfull{FOC} through box and whisker plots.}
\label{fig:FOC_per_speed_and_DSC}
\end{figure}

\subsection{Weather and Marine Data}


After presenting the ship sensor data, this subsection introduces the final data source: weather and marine variables obtained from Open-Meteo~[\citet{Zippenfenig2023}]. 
We chose this freely available data source for our empirical analysis, as the data provider used by the industry partner did not grant permission to use their system for research purposes. 
Table \ref{tab:Weather_Marine_Data} provides an overview of the obtained variables, which are linked to other data sources based on the \textit{Date \& Time} and \textit{Location} entries.
It is noteworthy that the open source data does not entirely match the time period of the ship sensor data, which is available from 2016.
Specifically, the weather variables can first be extracted for the second half of 2021, while the marine features become available starting in October of the same year.
More information about the data can be found in the documentation of the Historical Weather and Marine Forecast API on the Open-Meteo~[\citet{Zippenfenig2023}] website.


\begin{table}[H]
    \scriptsize
    \centering
    \captionsetup{justification=centering}
    \caption{Description of Weather and Marine Variables.}
    \label{tab:Weather_Marine_Data}
        \begin{tabular}{lll}
        \toprule
        Variable                    & Type (Unit) & Description \\
        \midrule
        Date \& Time                        & Time Stamp  & Time point of the data measurement.  \\
        Location                    & String     & Geographical coordinates of measured data. \\
        Weather code                & Float     & Weather condition based on the WMO code. \\
        Temperature 2m              & Float (\textdegree C) & Air temperature at 2 meters above ground.  \\
        Wind speed 10m              & Float (m/s) & Wind speed at 10 meters above ground. \\
        Wind direction 10m          & Float (\textdegree) & Wind direction at 10 meters above ground. \\
        Wave height                 & Float (m) & Wave height of significant mean waves.  \\
        Wave direction              & Float (\textdegree) & Mean direction of mean waves. \\
        Wave period                 & Float (s) & Period between mean waves. \\
        Ocean current velocity      & Float (m/s) & Velocity of ocean current. \\
        Ocean current direction     & Float (\textdegree) & Direction following the flow of the current. \\
        \bottomrule
        \end{tabular}
\end{table}


Similar to Subsection~\ref{sec:ship_sensor_data}, we derive additional variables based on the weather and marine features. 
For example, we observe that the \gls{FOC} can be sensitive to high wind speeds and ocean currents, particularly in the English Channel.
Depending on the situation, these forces may act in the direction of travel or against it.  
To include this effect, we introduce two variable transformations to our dataset, which are calculated in combination with the ship sensor data.
Specifically, we compute the new features based on
\begin{equation*}
    r_{\text{wind/current}} = v_\text{wind/current}\cdot\cos\left(|\theta_{\text{course}}-\theta_{\text{wind/current}}|\right)
\end{equation*}
where $\theta_{\text{course}}$ corresponds to the \gls{COG} variable from Table~\ref{tab:ship_sensor_data}, and $v_{\text{wind/current}}$ and $\theta_{\text{wind/current}}$ are the wind speed and direction (respectively ocean current velocity and direction) variables as described in Table~\ref{tab:Weather_Marine_Data}.


\begin{figure}[H]
    \centering
    \includegraphics[width=0.3\linewidth]{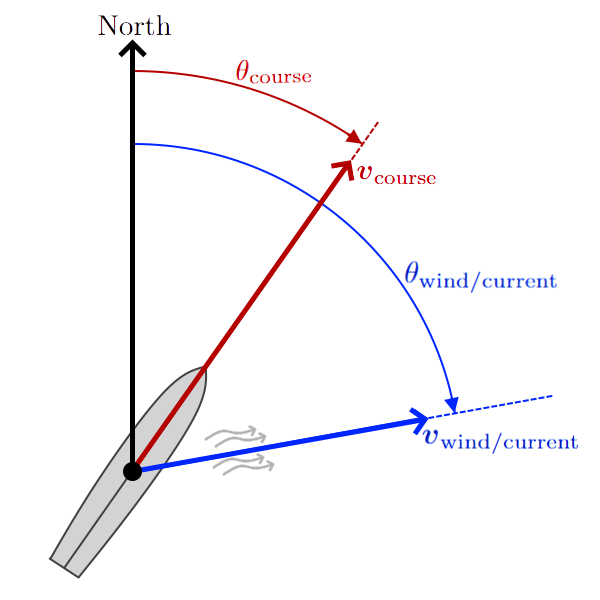}
    \caption{Visualization of the relative angle between \gls{COG} and wind/current direction.}
    \label{fig:course_wind}
\end{figure}

\section{Practical Application \& Numerical Studies}
\label{sec:numerical_experiments}


After explaining the mathematical framework for optimizing the cleaning schedule of ships and introducing the data sources, this section presents the findings of the empirical study. 
In particular, the first subsection focuses on comparing the performance of the \gls{ML} algorithms using the novel datasets and selects the best model to represent the data-driven function $g$.  
Based on the chosen \gls{ML} algorithm, the \gls{SHAP} values of several variables are then illustrated and interpreted.
Subsequently, the second subsection evaluates the results of the \gls{CSO} to empirically verify whether the application of the proposed algorithm leads to additional cleaning events in the considered time period that are economically justified and reduce fuel consumption.

\subsection{Selection of Fuel Prediction Function}
\label{sec:selection_of_fuel_prediction_function}


To compare the performance of the \gls{ML} algorithms introduced in Subsection~\ref{sec:reg_models}, hyperparameter tuning is required beforehand. 
In contrast to model parameters, hyperparameters are not learned during model estimation but need to be specified manually prior to training [\citet[pp.~35--36]{geron2022}].  
In our study, these parameters are chosen via a random search, which samples 100 different configurations from specified distributions [see Appendix~\hyperref[subsec:Linear_Regression]{\ref*{subsec:Linear_Regression}}--\hyperref[subsec:XGBoost]{\ref*{subsec:XGBoost}}] and re-estimates the model for each one.
The impact of different configurations can then be compared based on the model-specific validation scores.
An alternative to this procedure would be a grid search, which exhaustively explores the parameter space but becomes impractical for models with many hyperparameters, as in our case [\citet[p.~283]{bergstra2012}].


Having addressed the hyperparameter tuning, we now turn to the variable selection process in our study. 
Including all transformations and indicator variables presented in Section~\ref{sec:empirical_data}, the datasets contain a total of 54 variables. 
This high-dimensional setup can be problematic for three reasons: (i) the models may over-fit the data, (ii) training times may become long, and (iii) there may be high co-variance among variables.
A common approach to mitigate such issues is \gls{PCA}.
However, this method is incompatible with the type of analysis we wish to perform, as it removes the ability to study the effects of individual variables. 
To avoid this limitation, we apply forward stepwise feature selection [\citet[Section 3.3.2]{hastie2009}]. 
At each step, we trial adding one unselected variable at a time, run a new hyperparameter search, and then select the variable that results in the lowest validation score.


To perform the hyperparameter tuning and the variable selection, a validation set must be constructed.
For this purpose, we first split the instances for each vessel into a training and a testing set using an approximate 85-15 ratio to retain unseen data for the final model evaluation.
Specifically, we construct the subsets such that all the observations from a particular voyage [e.g., all 33 hourly instances from Portsmouth to Gijón] are assigned to either the training or the testing set, but never both, in order to avoid introducing data leakage into our comparison~[\citet{Rumala2023}].
Based on a similar procedure, we then further split the training data of each vessel into five folds to perform the feature selection and hyperparameter tuning as discussed previously.
This allows us to fit the model for each configuration five times by cycling through the folds one to five as the validation set while using the remaining four folds for training [see Figure~\ref{fig:k-fold-cross-val}].
The average validation score can then be used to assess the quality of a feature or hyperparameter choice [\citet[pp. 32--33]{Bishop2006}; \citet[pp. 241--242]{hastie2009}].


Consistent with the methodology described above, Figure~\ref{fig:foward_selection} illustrates the results of the variable selection for vessels $V_1$, $V_2$ and $V_3$ based on the \gls{XGB} model.
In this comparison, the \gls{RMSE} is used as performance metric and is shown on the y-axis [\citet{chai2014}].
The number of selected variables is shown on the x-axis, including the name of the variable selected at each step per vessel.
Overall, we find that selecting 13-15 features results in the best validation performance and that important variables like the \gls{STW} or \textit{Propeller Shaft} are selected first.
Moreover, we observe that several variables constructed to approximate the effect of hull fouling on the \gls{FOC} are selected across the different vessels.
For instance, the selection paths of $V_1$ and $V_3$ show that the \gls{DSDDM} and \gls{DSC} are chosen as fourth and third variable in the regression model, respectively.
For the hyperparameter selection, we find that the best configuration for each model varies across the nine vessels, but the optimized selection consistently offers a significant improvement over the default values.


To compare the performance of the \gls{ML} algorithms, we select multiple performance metrics, which are commonly used in the literature [\citet[p. 10]{GUPTA2022}; \citet[p.~50]{botchkarev2019}]. 
Table~\ref{tab:summary_R2} summarizes the model results for each vessel using the $\text{R}^2$ score computed based on the unseen test data.
In this overview, the best result for each vessel is highlighted in bold, which indicates that \gls{XGB} delivers the best performance regardless of the dataset.
Similar, though slightly worse, results are provided by the \gls{RF}, \gls{ET}, and \gls{SVR}.
This ranking suggests the presence of non-linear relationships between some input variables and the \gls{FOC}, as well as potential interaction effects, since the tree-based models in particular are very effective in identifying such patterns.
A first noticeable performance drop is then observed for the \gls{MLP}.
In principle, this model can also be very powerful in detecting non-linear effects and interactions between variables.
However, the predefined structure of the hidden layers might limit this capability for the data in our study, preventing the model from fully playing to its strengths.
Another performance drop is then observed for the two linear models, \gls{LASSO} and \gls{LinReg}, which further promotes the previously made suppositions about the non-linear data structure and interaction effects.
The overall worst $\text{R}^2$ scores for the test data are achieved by \gls{kNN}.
For this model, we observe that fewer than five variables are selected and that the number of neighbors is very close to the upper limit, which may lead to overly generalized \gls{FOC} predictions.


\begin{figure}[H]
    \centering
    \includegraphics[width=1\linewidth]{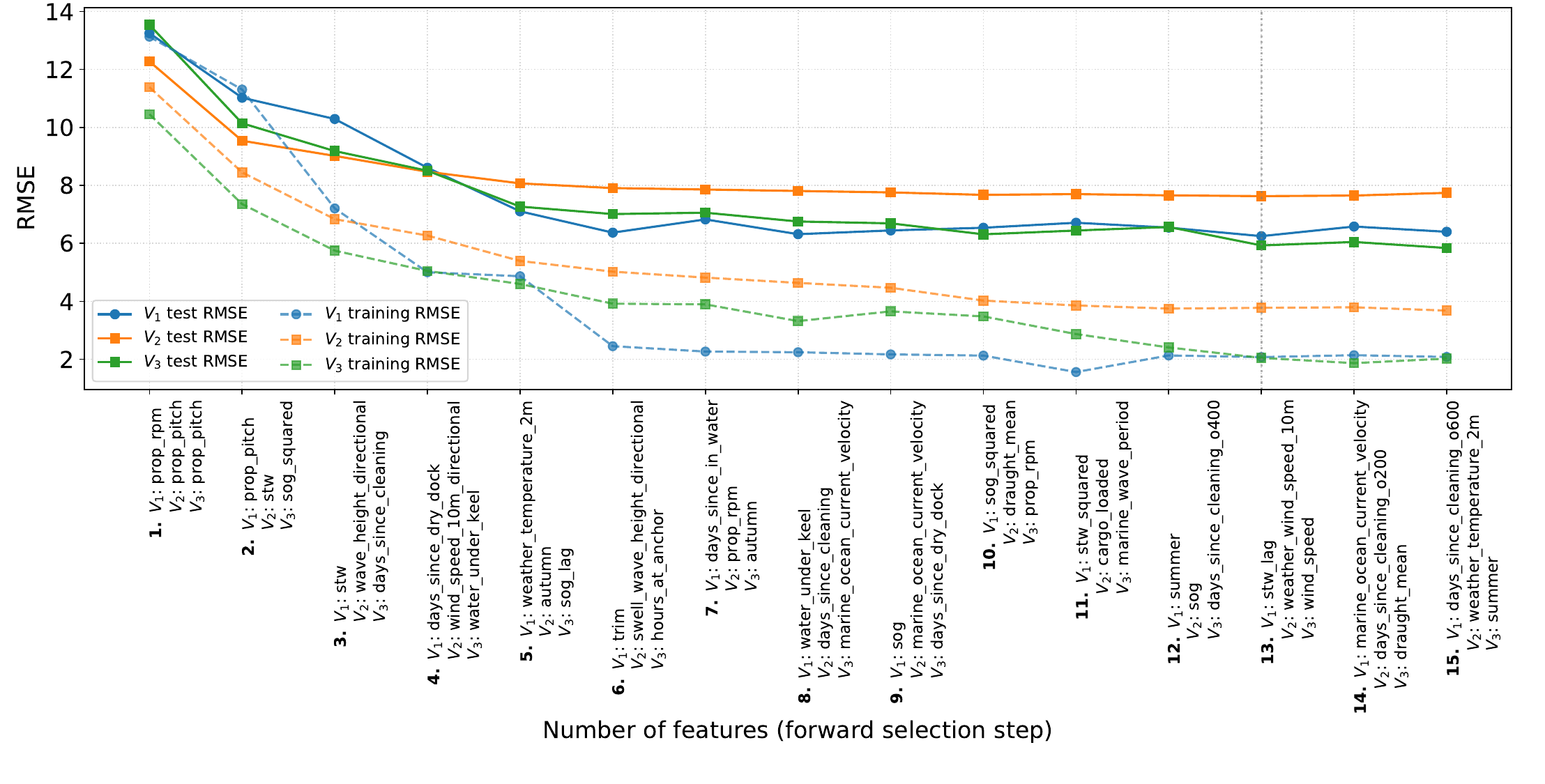}
    \caption{The forward feature selection process visualized through the training and test error of the \gls{XGB} model for vessels $V_1$, $V_2$ and $V_3$.}
    \label{fig:foward_selection}
\end{figure}

\vspace{-0.4cm}


Besides using the $\text{R}^2$ score as performance metric, we also compare the model predictions based on the \gls{RMSE} and the \gls{MAE}.
The corresponding results are shown in Table~\ref{tab:summary_RMSE} and Table~\ref{tab:summary_MAE}, respectively. 
While the \gls{RMSE} scores suggests an identical model ranking, the best \gls{MAE} scores for each vessel are distributed across the tree-based models.
This indicates that \gls{XGB} generates fewer very large prediction errors, since the $\text{R}^2$  and \gls{RMSE} penalize such mispredictions more strongly than the \gls{MAE}.


\begin{table}[H]
    \scriptsize
    \centering
    \captionsetup{justification=centering}
    \caption{$\text{R}^2$ scores of the \gls{ML} models on the test datasets.}
\begin{tabular}{l|ccccccccc}
\toprule
& \acrshort{LinReg}
& \acrshort{LASSO}
& \acrshort{kNN}
& \acrshort{MLP}
& \acrshort{SVR}
& \acrshort{ET}
& \acrshort{RF}
& \acrshort{XGB}\\
\midrule
$V_1$   & 0.5110 & 0.5148 & 0.3900 & 0.6929 & 0.9552 & 0.9671 & 0.9672 & \textbf{0.9686}\\
$V_2$   & 0.7107 & 0.7085 & 0.3669 & 0.6670 & 0.8817 & 0.9049 & 0.9038 & \textbf{0.9054}\\
$V_3$   & 0.7134 & 0.7139 & 0.3345 & 0.7732 & 0.9785 & 0.9854 & 0.9860 & \textbf{0.9866}\\
$V_4$   & 0.8980 & 0.8960 & 0.6568 & 0.9034 & 0.8438 & 0.8961 & 0.9101 & \textbf{0.9174}\\
$V_5$   & 0.7852 & 0.7842 & 0.3820 & 0.9260 & 0.9260 & 0.9774 & 0.9785 & \textbf{0.9799}\\
$V_6$   & 0.3608 & 0.3638 & 0.3123 & 0.6758 & 0.8907 & 0.9104 & 0.9210 & \textbf{0.9229}\\
$V_7$   & 0.7348 & 0.7330 & 0.5565 & 0.7875 & 0.9260 & 0.9303 & 0.9293 & \textbf{0.9394}\\
$V_9$   & 0.5436 & 0.5447 & 0.3296 & 0.8511 & 0.9881 & 0.9915 & 0.9916 & \textbf{0.9926}\\
$V_{10}$& 0.7601 & 0.7602 & 0.4100 & 0.6744 & 0.9673 & 0.9793 & 0.9796 & \textbf{0.9861}\\
\bottomrule
\end{tabular}

    \label{tab:summary_R2}
\end{table}

\vspace{-0.3cm}


Overall, the best model performance is provided by \gls{XGB}.
Thus, we select this model to represent the data-driven function $g$ for the remainder of this paper. 
To obtain a complete picture of the model performance, Table~\ref{tab:XGB_detailed} reports the performance scores on both the training and testing datasets for each vessel. 
The corresponding \glspl{CI} are shown in Appendix~\ref{sec:bootstrapping}.


\begin{table}[H]
    \scriptsize
    \centering
    \captionsetup{justification=centering}
    \caption{Detailed report of the \gls{XGB} performance in predicting \gls{FOC}.}
\begin{tabular}{l | ccc | ccc }
\toprule
& \multicolumn{3}{c}{Training} & \multicolumn{3}{c}{Test} \\
& \gls{RMSE} (kg/h) & \gls{MAE} (kg/h)    & $\text{R}^2$ score & \gls{RMSE} (kg/h) & \gls{MAE} (kg/h)   & $\text{R}^2$ score \\
\midrule
$V_1$   & 2.0070   & 1.2537   & 0.9973   & 6.2752   & 6.1702   & 0.9686 \\
$V_2$   & 3.8249   & 4.8752   & 0.9806   & 7.7433   & 7.5517   & 0.9054 \\
$V_3$   & 0.6237   & 0.6190   & 0.9996   & 4.0179   & 3.9315   & 0.9866 \\
$V_4$   & 2.7452   & 1.7175   & 0.9918   & 6.8672   & 5.5446   & 0.9174 \\
$V_5$   & 1.1335   & 1.1385   & 0.9986   & 6.6745   & 6.8213   & 0.9799 \\
$V_6$   & 1.7083   & 1.8577   & 0.9905   & 8.5962   & 7.8703   & 0.9229 \\
$V_7$   & 3.7351   & 2.0250   & 0.9836   & 8.2419   & 5.8392   & 0.9394 \\
$V_9$   & 2.5520   & 1.4508   & 0.9959   & 4.1277   & 2.9728   & 0.9926 \\
$V_{10}$& 1.8322   & 1.5519   & 0.9973   & 5.4880   & 5.0052   & 0.9861 \\
\midrule
\bottomrule
\end{tabular}
    \label{tab:XGB_detailed}
\end{table}

\vspace{-0.3cm}


After evaluating the \gls{ML} algorithms using performance metrics, we now examine \gls{SHAP} values to further verify the effects of the input variables on the \gls{FOC}.
For this purpose, we apply the theoretical framework introduced in Subsection~\ref{sec:shap} to compute \gls{SHAP} values based on the training datasets, where it becomes advantageous that \gls{XGB} was selected as the underlying model.
The tree‑specific Python implementation enables an efficient computation of \gls{SHAP} values, even for several thousand instances.
These values can then be visualized through \gls{SHAP} dependence plots, where the variable realizations and the \gls{SHAP} values are on the horizontal and vertical axis, respectively.
Figure~\ref{fig:plausi_shap_plots} illustrates six examples of such plots using the data of vessel $V_7$.
Each plot is automatically colored based on the strongest interaction with another variable, and a local smoothing function highlights the main trends of the point cloud. 
Overall, the dependence plots indicate plausible relationships between the input variables and \gls{FOC}. 
For instance, higher values of \textit{Pitch Propeller} or \textit{Wind Speed} are associated with higher \gls{SHAP} values.
This indicates that the difference between the instance-specific \gls{FOC} prediction and the mean prediction increases for higher pitches and wind speeds, which is in line with the expectation.
An opposite relationship is observed for \textit{Weather Temperature}, i.e., the warmer it is, the lower the \gls{SHAP} values become. 
This effect is further amplified when combined with higher \textit{Pitch Propeller} values.


\begin{figure}[H] 
    \includegraphics[width=.49\linewidth]{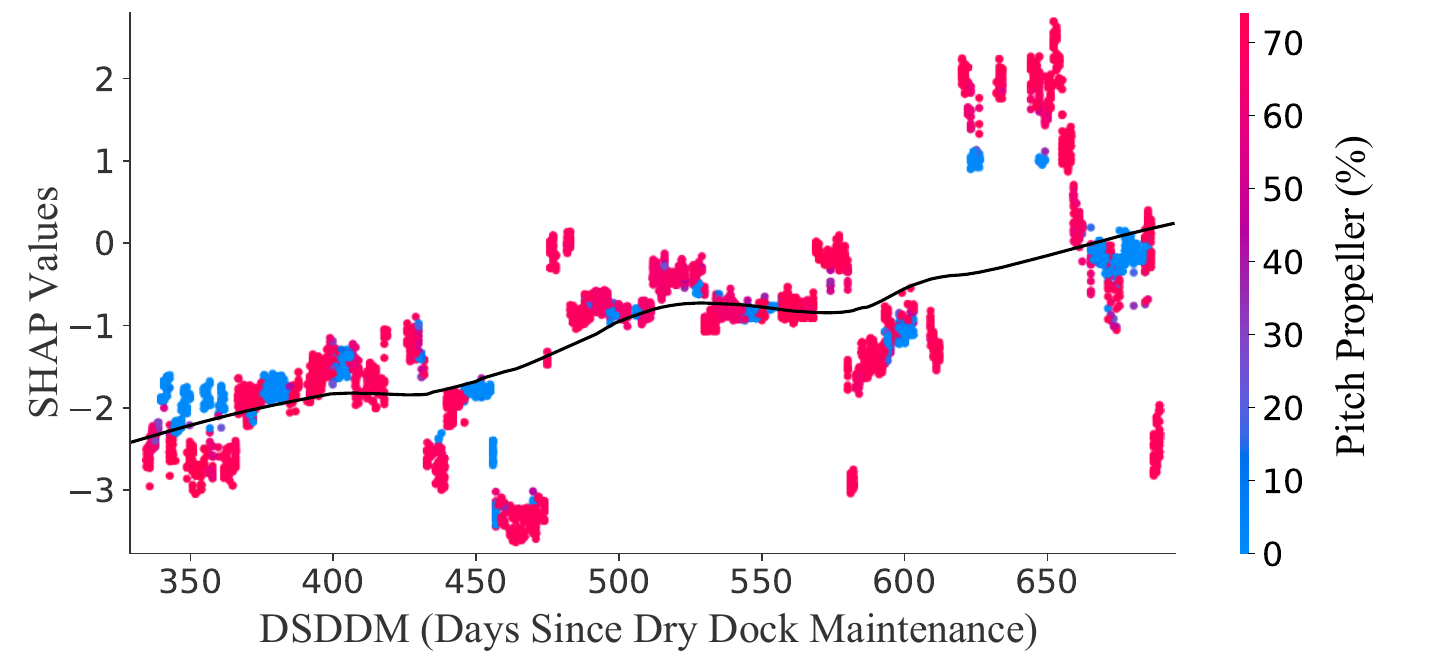}
    \includegraphics[width=.49\linewidth]{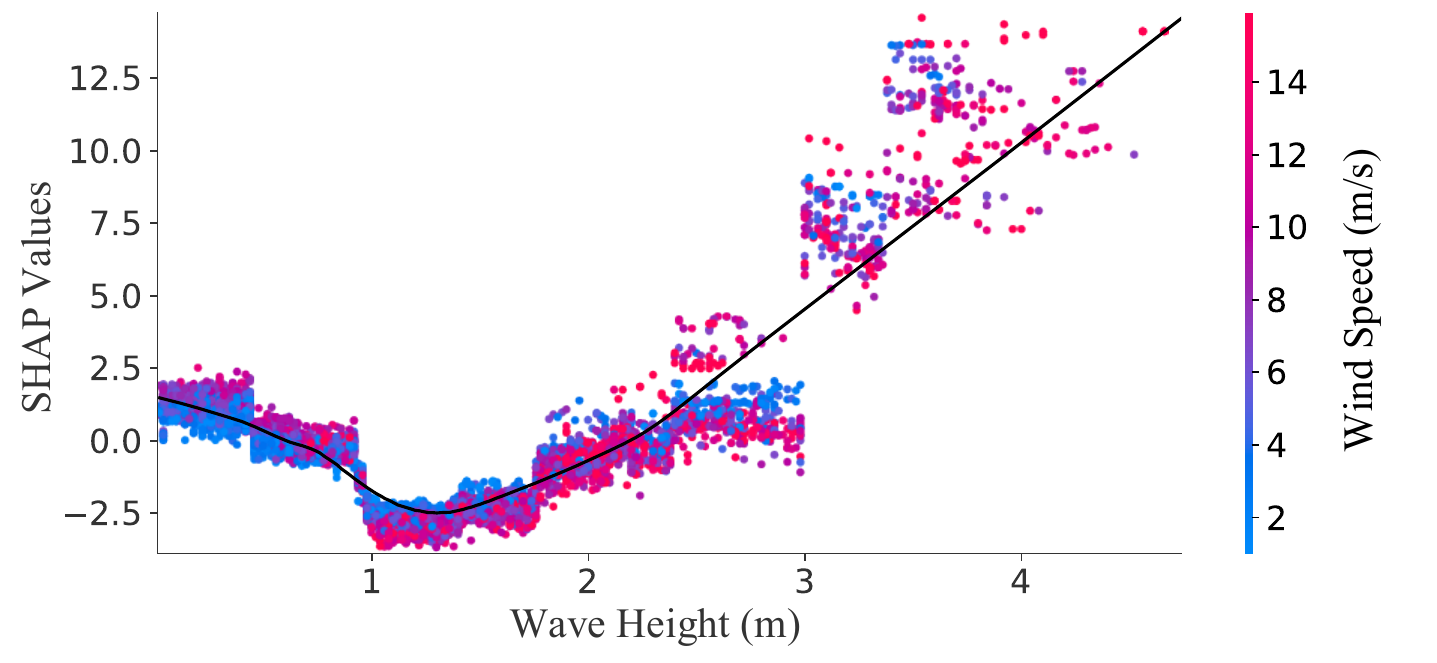}\\
    \includegraphics[width=.49\linewidth]{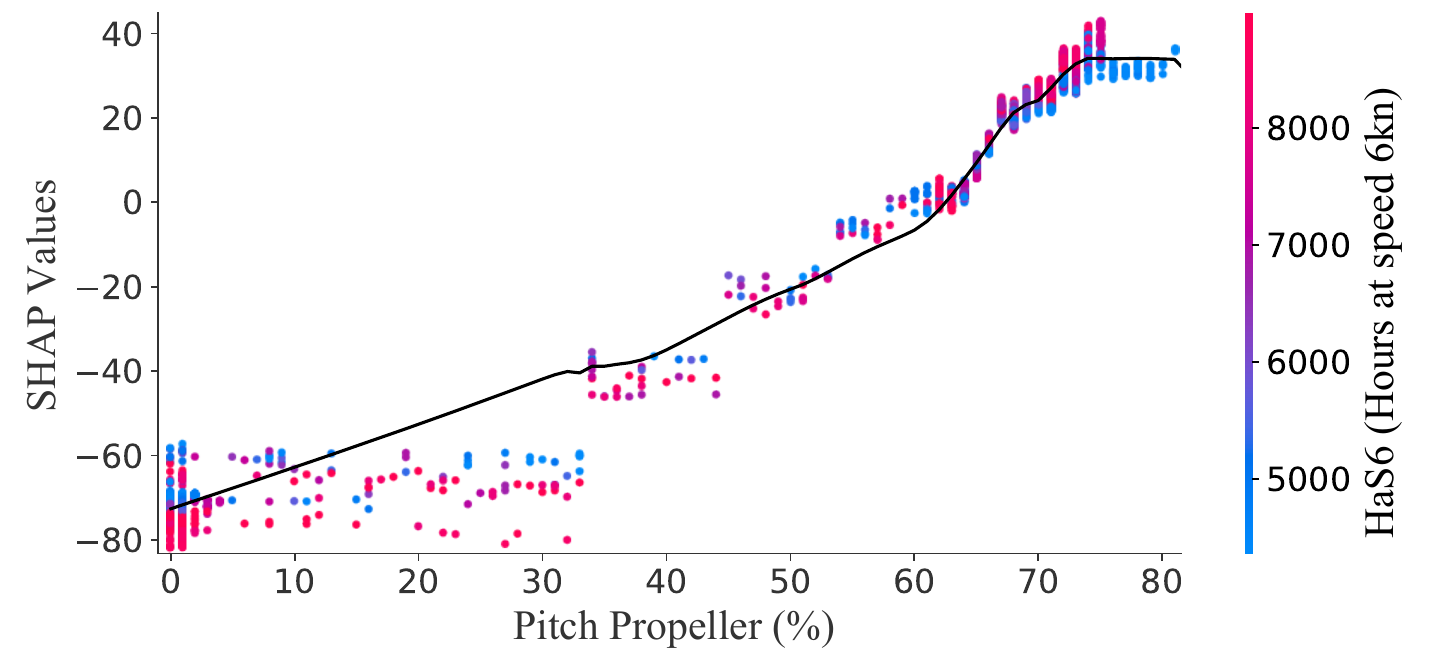}
    \includegraphics[width=.49\linewidth]{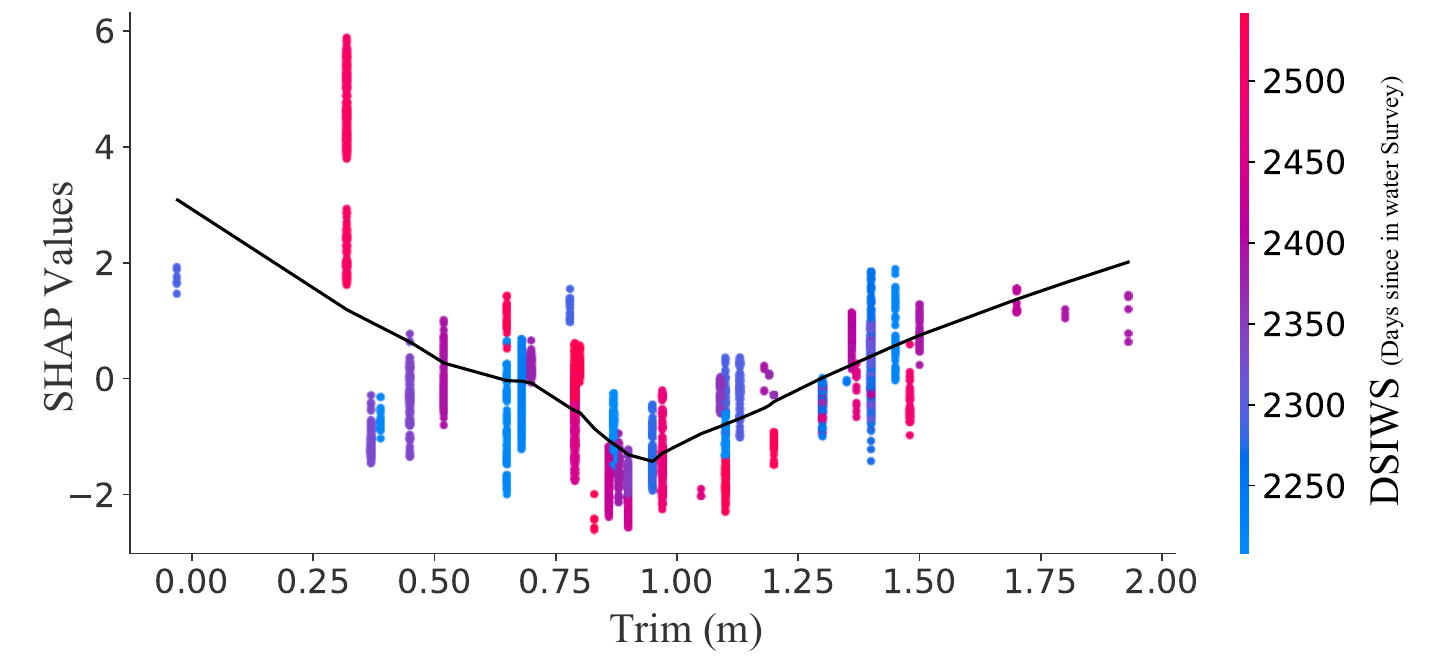}\\
    \includegraphics[width=.49\linewidth]{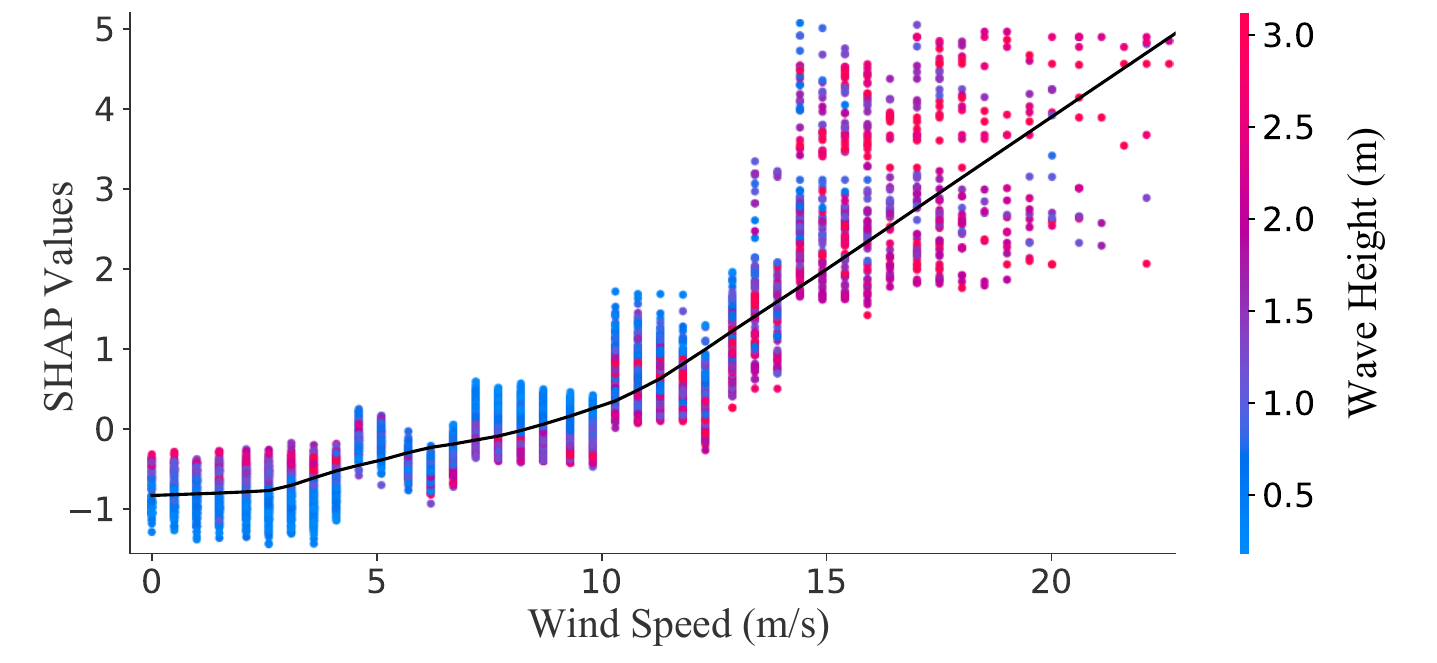}
    \includegraphics[width=.49\linewidth]{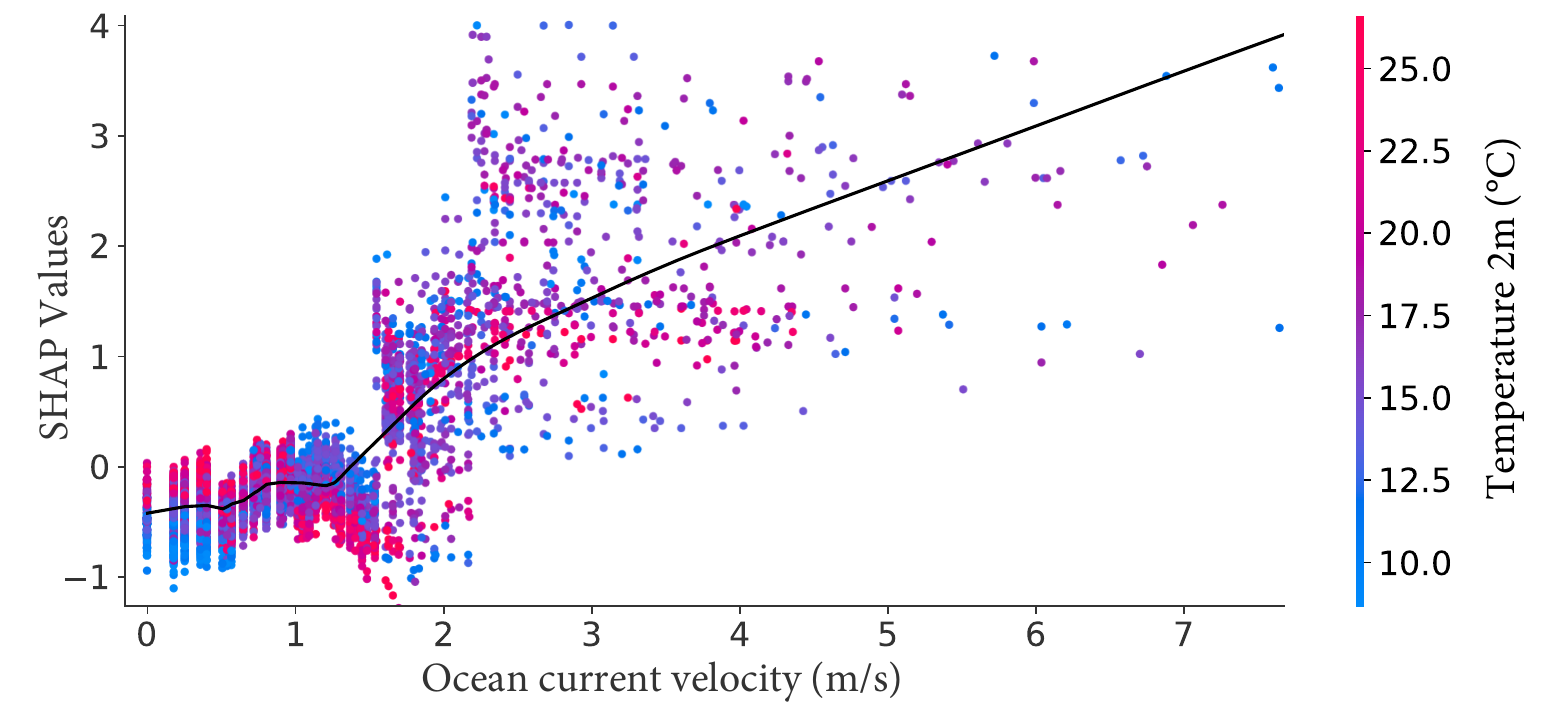}
    \caption{Six example \gls{SHAP} dependence plots using the \gls{XGB} model trained on the vessel $V_7$ data.
    On the horizontal axes are six different input variables and on the vertical axes are corresponding \gls{SHAP} values, which can be interpreted as the contribution of an individual feature value to the difference between the instance-specific \gls{FOC} prediction and the mean prediction.
    The coloring allows to check for interactions to a second input variable.  
    }
    \label{fig:plausi_shap_plots}
\end{figure}


Ultimately, we also examine the dependence plot of the \gls{DSDDM}, which is one of the selected biofouling measures.
Even though its \gls{SHAP} values are more scattered, the local smooth function still indicates the anticipated effect. 
We predominately observe that a higher count of the \gls{DSDDM} coincides with a greater difference between the instance-specific \gls{FOC} prediction and the mean prediction.
At this point, it is worth noting that not all generated \gls{SHAP} plots exhibit patterns as consistent as those shown in Figure~\ref{fig:plausi_shap_plots}.
To explain this behavior, a detailed inspection of the individual split points in the regression trees would be required.
However, this becomes increasingly challenging for deeper tree structures and is beyond the scope of this study.

\vspace{-0.3cm}


\subsection{Empirical Results of Cleaning Schedule Optimization}
\label{sec:empirical_results_cso}


After developing a prediction model for the \gls{FOC}, this subsection showcases how the optimization framework introduced in Section~\ref{sec:CSO} can be applied to determine optimal cleaning schedules for the vessels of our industrial partner.
In particular, we address the question of when a ship cleaning should be performed in order to minimize the overall operational expenditures, which consist of the cleaning and fuel costs.
For this purpose, we subset the data of each ship to an approximately four-year period that begins with a dry-dock cleaning event.
This has the advantage that the biofouling measures are all initiated at zero and that the Open-Meteo~[\citet{Zippenfenig2023}] data is predominantly available. 
We then further split the data of each vessel into chronologically ordered voyages $j=1,\dots,n$, where the hourly observations $t \in \mathcal{T}_j$ of voyage~$j$ provide the selected features that do not depend on the cleaning schedule, i.e., the vectors $\uu_t$ (e.g., \textit{Trim}, \textit{Cargo Load}, etc.).
To predict the \gls{FOC} for a given instance, we select the \gls{XGB} model as data-driven function $g$, which was discussed in more detail in  Subsection~\ref{sec:selection_of_fuel_prediction_function}.


Based on the previously described setup, it remains to determine the input parameters of the optimization problem shown in \eqref{eq:cleaning-2}.
In this study, the fouling state is a vector with $d = 7$ biofouling measures, defined as
$\bb_j =
    \left[
    \text{\gls{DSDDM}}, \
    \text{\gls{DSIWS}}, \
    \text{HU},          \
    \text{HaS0},        \
    \text{HaS6},        \
    \text{HaS9},        \
    \text{HaS12}
    \right]^\top$,
which keeps track of the ship's time spent in different operational states since the last cleaning event.
These values allow the regression models to learn how fouling is built up over time, as shown in Figure~\ref{fig:plausi_shap_plots} for the \gls{DSDDM} variable.
The accumulations $\bdelta_t$ and $\bDelta_j$ of these measures are then calculated according to the record of each voyage.
For example, if voyage $j$ spent ten days at anchor and five days sailing at 10 kn, the accumulation over the voyage could take the following values: $\bDelta_j = [15, 15, 0, 240, 0, 0, 120]^\top$.
As each considered period begins with a dry-dock maintenance, the initial fouling state is $\bar{\bb} = \mathbf{0}$.
The fuel costs $F_j$ of a voyage are computed via~\eqref{eq:voyage_cost}, i.e., the \gls{FOC} predicted by $g$ is multiplied with the fuel price $p_t$ as recorded on the website \url{https://shipandbunker.com/prices}.
Note that in our analysis, $c_j$ is set to $\$10{,}000$ at all ports, which is an estimate that was derived based on values from the literature in Subsection~\ref{sec:hull_cleaning} and can be fine-tuned in future applications.
In this manner, we have defined all components of an instance of \eqref{eq:cleaning-2}.


To implement the brute-force search and dynamic programming algorithm, introduced in Section~\ref{sec:CSO}, we utilized Python~3.12. 
The most up-to-date code can be accessed through the project's GitHub repository \href{https://github.com/SamuelWardPhD/Cleaning-Schedule-Optimization}{Cleaning-Schedule-Optimization \faExternalLink}, but a legacy code is also provided in Appendix~\ref{sec:python_imple}.
Note that since the vessel data contains confidential information, it is not included in this repository.
In our local version, the data of our industry partner is efficiently loaded and manipulated using the \texttt{Pandas} library, and all linear algebra is computed through \texttt{NumPy}. 
Based on this setup, we can then run both algorithms on the data for each ship. 
The brute-force search, i.e., Algorithm~\ref{alg:brute_force}, is given an hour to run but always failed to terminate in our case study. 
In contrast, the dynamic programming approach consistently terminated in under two minutes.
Therefore, the remainder of this subsection will only focus on the results of Algorithm~\ref{alg:dynamic_programming}.


Figure~\ref{fig:fuel_savings} summarizes the optimization results in terms of fuel savings across the vessels of our industry partner. 
Overall, in seven out of nine cases, the algorithm produces a cleaning schedule that includes at least one additional cleaning event over the four-year period.
In four out of nine cases, the algorithm even outputs a second cleaning event within this time period.
For the vessels $V_2$, $V_3$, $V_5$, $V_6$, $V_7$, $V_9$, $V_{10}$, the optimized cleaning schedules then translate into a reduction of fuel consumption by 2.9\%, 2.1\%, 1.9\%, 3.2\%, 5.0\%, 3.5\%, 1.3\%, respectively.
Only for the two vessels, $V_1$ and $V_4$, the algorithm could not recommend an additional cleaning as the corresponding fuel savings would have bee too small to offset the cleaning costs.


\begin{figure}[H]
    \centering
    \includegraphics[width = 4in]{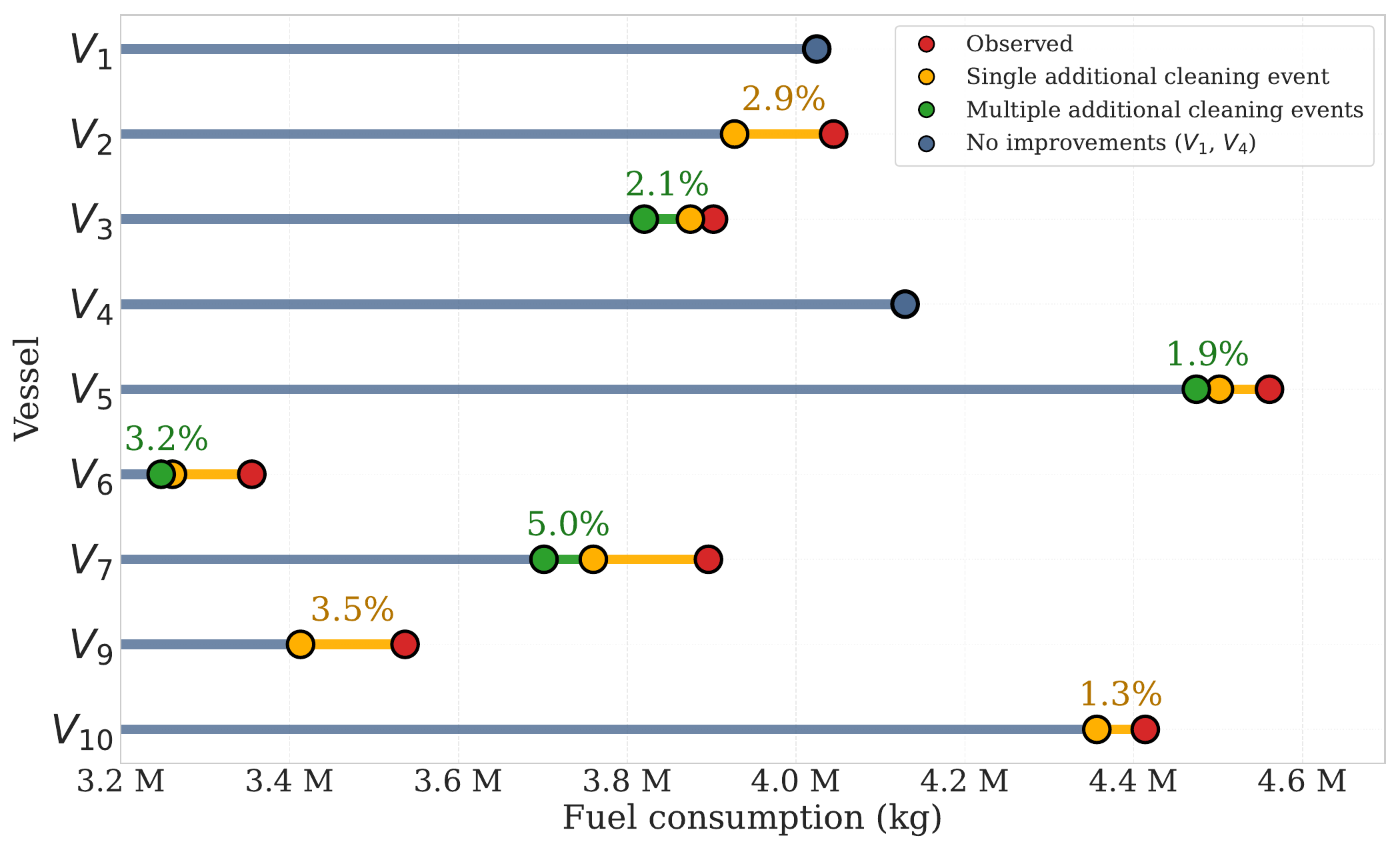}
    \caption{The fuel savings relative to the overall fuel consumption of the nine vessels in our study.}
\label{fig:fuel_savings}
\end{figure}

\vspace{-0.3cm}


As next step, the net savings of the company are examined based on the seven vessels, where the reduction in fuel costs outweighs the costs of cleaning. 
Specifically, our numerical results show that if the company had performed one extra cleaning on each of the aforementioned vessels during the four-year study period, this would have resulted in savings of approximately \$315{,}000. 
To visualize this financial impact, Table~\ref{tab:cleaning_optimisation_A} provides a detailed breakdown of the algorithmic output for the vessels $V_2$, $V_3$, $V_5$ and $V_9$.
It is interesting to note that a company could reasonably brute force a local solution to the problem for exactly one cleaning event without the innovation of Algorithm~\ref{alg:dynamic_programming}, which is shown for $V_3$ and $V_5$.
However, this naive approach would not unfold the full potential to save fuel and corresponding $\text{CO}_2$ emissions.
For instance, the results of $V_3$ demonstrate that fuel consumption would then only be reduced by 0.7\% instead of 2.1\%.


\vspace{-0.1cm}

\begin{table}[H]
\scriptsize
\caption{The output of our cleaning schedule optimization for four of the nine vessels.  
The percentages show the predicted savings in kilograms of fuel consumption (resp. costs in USD) as a percentage of the observed fuel consumption (resp. costs in USD).}
\vspace{-1em}
\label{tab:cleaning_optimisation_A}
\centering
\begin{tabular}{rrrrr}
\toprule
                         & $V_2$        & $V_3$         & $V_5$         & $V_9$      \\
    \midrule
\textbf{Parameterisation}\\
    Number of voyages   & 125           & 103           & 124           & 71      \\
    Hours of sailing    & 30,458        & 28,062        & 35,301        & 28,019  \\
    Start date          & 17 Jun 2021   & 14 Sep 2020   & 24 Apr 2020   & 07 Sep 2021  \\
    End date            & 28 Dec 2024   & 14 Sep 2024   & 24 Apr 2024   & 28 Dec 2024  \\
    \midrule
\textbf{No additional cleaning}\\
    Compute time (s)    & 1              & 1              & 1                 & 1             \\
    Fuel (kg)           & 4,044,410.50   & 3,902,133.58   & 4,560,798.18     & 3,536,718.12  \\
    Cost (\$)           & \$3,148,573.58 & \$3,037,810.99  & \$3,550,581.39   & \$2,753,335.06  \\
    \midrule
\textbf{One additional cleaning}\\
    Compute time (s)    & 70 & 54  & 65 & 25 \\
    Cleaning date       & 19 Jul 2022       & 05 Dec 2021    & 13 Jun 2022      & 12 Mar 2023  \\
    Fuel (kg)           & 3,927,073.92      & 3,874,779.62   & 4,501,393.87     & 3,413,095.09  \\
    Cost (\$)           & \$3,067,227.05    & \$3,026,515.94 & \$3,514,335.12   & \$2,667,094.52  \\
    $\Delta$ Fuel (kg)  & (2.9\%) \hfill 117,336.58   & (0.7\%) \hfill 27,353.96   & (1.3\%) \hfill 59,404.32      & (3.5\%) \hfill 123,623.04  \\
    $\Delta$ Cost (\$)  & (2.6\%) \hfill \$81,346.53  & (0.4\%) \hfill \$11,295.06 & (1.0\%) \hfill \$36,246.26    & (3.1\%) \hfill \$86,240.54  \\
    \midrule
\textbf{Multiple additional cleaning}\\
    Compute time (s)        & 70       & 54                 & 66                     &25 \\
    Cleaning dates (year)   & -        & 2021, 2023         & 2021, 2022, 2023        & - \\
    Fuel (kg)               & -        & 3,820,188.77       & 4,471,948.72          & - \\
    Cost (\$)               & -        & \$2,984,016.96     & \$3,511,412.08        & - \\
    $\Delta$ Fuel (kg)      & -        & (2.1\%)\hfill 81,944.81   & (1.9\%) 88,849.46    & - \\
    $\Delta$ Cost (\$)      & -        & (1.8\%)\hfill \$53,794.03 & (1.1\%) 39,169.31  & - \\
    \midrule
\bottomrule
\end{tabular}
\end{table}

\vspace{-0.25cm}


Altogether, the results of our study demonstrate that cleaning vessels more regularly does have a financial benefit for our industrial partner. 
Furthermore, performing those cleanings according to an optimal solution of \eqref{eq:cleaning-2} provides significant financial gains over a more naive schedule.  
In the future, when the company can provide a forecast of the upcoming voyages, Algorithm~\ref{alg:dynamic_programming} can then be applied to obtain optimal dates to schedule the next cleanings.



\subsection{Deployment as a web app}
\label{subsec:web_app}


To allow other parties, such as our industrial partner, to interact with these results, a web app was developed.
In this subsection, we highlight the key functionalities, design choices and how such a tool can foster fruitful collaboration. 
The app's front page is accessed at~\url{https://seaflint.org/} although the dashboard and other features can only be accessed with login credentials.
Access to the app can be made available upon request from the authors.


The website's back-end was developed in Django~\citep{Django} which was chosen for its well-established systems for authentication and its smooth integration with our Object–relational mapping style of database.
The machine learning uses Scikit-learn~\citep{Pedregosa2011} and XGBoost~\citep{Chen2016} libraries, the use of which is discussed in more detail in Appendix~\hyperref[subsec:Linear_Regression]{\ref*{subsec:Linear_Regression}}--\hyperref[subsec:XGBoost]{\ref*{subsec:XGBoost}}.
The front-end consists of some static HTML pages as well as Plotly~\citep{Plotly} data tables and dashboards. 
Plotly is considered one of the best visualization tools for data-heavy projects and creates highly interactive graphics.
The site is then deployed using Docker because it provides reproducible, portable deployment between our development environment and the  University of Southampton servers.


\begin{figure}[H]
    \centering
    \includegraphics[width=\linewidth]{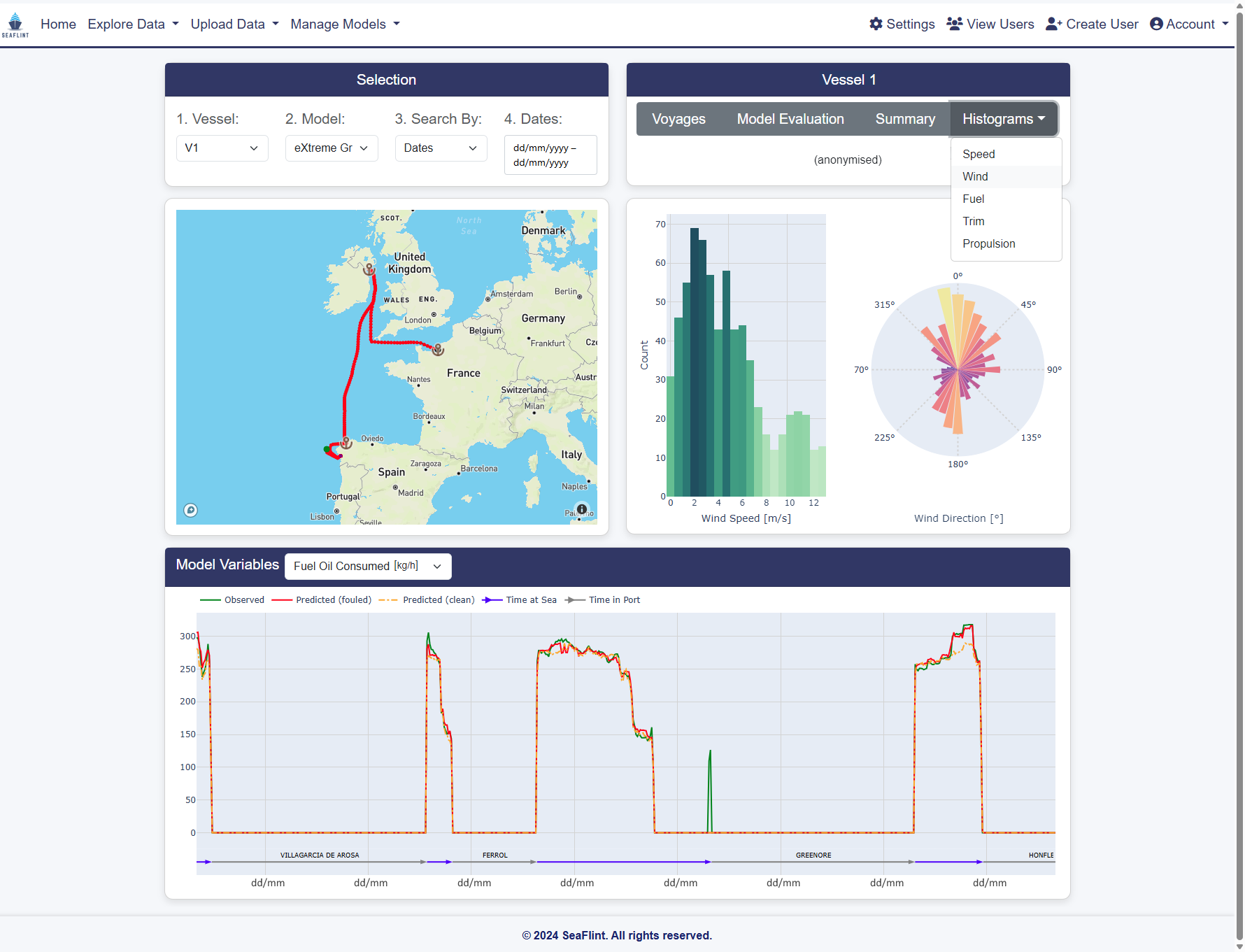}
    \caption{
    A screenshot of the dashboard as seen by an admin after logging in, selecting vessel $V_1$ and filtering for 2 weeks date range. 
    The image has been edited to anonymize the vessel and dates.
    Furthermore, some minor adjustments to the spacing were made to allow the whole dashboard to appear on a single screen without scrolling. 
    }
    \label{fig:dashboard}
\end{figure}


The capabilities of the web app are best described through the four user groups: researcher, fleet manager, analyst and admin. 
Firstly, a \emph{researcher} can upload new machine learning models and data representations though serialized Python objects. 
For example, this commonly takes the form of a Scikit-learn pipeline. 
Permitted class types and input validation is properly enforced to prevent security issue regarding the ability to upload arbitrary code.
This system allow the researchers to experiment freely in their own environment while seamlessly integrating their newest ideas into the web app. 
Secondly, a \emph{fleet manager} can upload and update the data record for the vessels in their fleet. 
For rare events, such as recording a dry dock, this is completed through a form.
For the hourly ship sensor data, this is uploaded as an Excel or CSV file or connected via agreed use of proprietary \glspl{API}. 
Those files are then automatically preprocessed and linked to weather data according to the steps outlined in Section~\ref{sec:empirical_data}.
Next an \emph{analyst} has permissions to view and interact with the above data and models without the ability to upload or update data or models. 
Finally, an \emph{admin user} can set up new users and link their accounts to the right vessels.


All users then have access to the dashboard, seen in Figure~\ref{fig:dashboard}, where the main functionality of the app lies. 
In the top left, is the selection panel, which can be used to select a vessel, select a \gls{ML} model and filter by dates of interest.
The two panels in the right-hand column link together.
They show basic statistics about the selected vessel and allows the user to switch between different visualization such as the histogram and wind rose seen in Figure~\ref{fig:dashboard}. 
Clicking the ``Model Evaluation'' button, shows the R$^2$, \gls{RMSE} and \gls{MAE} computed against ground-truth observations on the selected data.
This allows the fleet managers to upload their most recent data and validate the performance of the model on unseen data independently from our reporting.


In the middle can be seen the map.
This helps contextualize each data point.  
The user can hover over any of the red coordinates dots or port markers to see the date, time and position of that data point.
Furthermore, data points that are excluded for any reason are shown in a different color.
For example, the purple dot in Figure~\ref{fig:dashboard} indicates that the code was unable to connect the hourly sensor data to a valid Weather and Marine data entry due to it being in the Ulla river, Spain, rather than the sea.
In fact all data points in rivers are excluded for this reason. 
These interactions can help our industrial partner build confidence in our methodology.
The map is rendered using OpenStreetMap data according to their Open Data Commons Open Database License~(ODbL).


At the bottom of Figure~\ref{fig:dashboard} is the timeseries plot.
Along the horizontal axis, gray intervals show the time in port and blue intervals show the time at sea.
The drop-down box allows any of the dependent or independent variables to be plotted on the vertical axis. 
A key functionality when plotting the \gls{FOC}, is that the user sees a comparison between three lines.
In green they see the observed \gls{FOC}; in red the predicted \gls{FOC} for the known biofouling metric; and in orange the predicted \gls{FOC} with the biofouling metric set to zero simulating a clean vessel.
A traffic light warning symbol, based on agreed thresholds, then informs the fleet manager if a significant difference is predicted between the bio-fouled and clean predictions with the corresponding percentage and confidence intervals.
The magnitude of this gap informs decisions on dry-dock scheduling.
As a whole, these interactions allow the user to scrutinize our models and their vessels' performance on a voyage-by-voyage basis beyond what a static report could provide.

\section{Conclusions}
\label{sec:conclusion}


In this paper, we presented a novel optimization framework that determines vessel-cleaning schedules dynamically based on real-world voyage data. 
Compared to the fixed-time policies commonly used in industry, this framework not only quantifies the loss in fuel efficiency due to hull fouling mathematically but also adapts cleaning decisions to the vessel’s operational state. 
Specifically, the core components of the methodology are (i) a data-driven function that predicts vessel fuel consumption based on several input variables, including approximations of hull fouling, and (ii) an objective function that induces vessel cleaning when the associated costs can be offset by the resulting increase in fuel efficiency. 
Theoretically, we proved that the most general formulation of the \gls{CSO} problem is NP-hard and that, if the fuel cost functions can only be accessed through evaluations, every exact algorithm requires at least $2^{n+1} - 2$ evaluations for $n$ voyages.
To ensure practicality of the approach for real‑world deployment, we introduced a fouling-state model, in which the cleaning history affects the fuel costs only through a vector of biofouling measures, and showed that this model can be solved exactly via dynamic programming.
In particular, we derived a Bellman equation for this model, which yields the algorithm presented in Algorithm~\ref{alg:dynamic_programming}.
It requires only $n(n+3)/2$ evaluations of the voyage fuel costs and thus ensures computational efficiency even for hourly data over several years.


In our numerical analysis, we evaluated the performance of the proposed optimization framework using real-world vessel data. 
Through a cooperation agreement with a UK‑based shipping company, we obtained access to voyage records from ten tramp‑trading vessels.
In the academic literature, these datasets had never been used and analyzed before, and offered more vessel-specific fouling profiles than data from ships operating on fixed schedules and routes.
This enabled us to contribute a novel performance comparison of several \gls{ML} techniques deployed for each vessel to predict fuel consumption and evaluated as candidates for the data-driven function in our optimization framework.
Across several linear and non-linear models, the tree-based approaches \gls{RF}, \gls{ET} and \gls{XGB} achieved the best performance.
In contrast, the linear models, \gls{LinReg} and \gls{LASSO}, generated much worse predictions, which pointed to substantial non-linear patterns and interaction effects in the underlying data.
This supposition was further supported by \gls{SHAP} dependence plots generated for each vessel based on the \gls{XGB} model, which outperformed the other approaches in terms of \gls{RMSE}, \gls{MAE} and $\text{R}^2$. 
In this part of the analysis, we also examined the effects of the variables constructed to approximate hull fouling and found predominately plausible impacts on fuel consumption, providing empirical evidence for the usefulness of the variables.
After selecting \gls{XGB} as data-driven function for the \gls{CSO}, we then also applied the dynamic programming algorithm to the data and derived vessel specific cleaning schedules. 
Over a time period of four years, these schedules showed that our optimization framework can reduce fuel consumption by up to 5\%, even when accounting for the costs of one or two additional cleaning events.
Additionally, we presented how the whole pipeline of data cleaning though to visualization was automated and deployed as a web app allowing for independent and ongoing validation of our methodology by the industrial partner.


The empirical results confirm that the proposed optimization framework provides a meaningful step towards data‑driven cleaning schedules for sustainable and profitable ship operation.
Despite these promising results, several directions for future research and development remain.
For instance, the data-driven fuel prediction function could be combined with physical models to further improve predictive accuracy. 
Similarly, the approximation of hull fouling could be enhanced by incorporating and evaluating images from inexpensive underwater cameras. 
In terms of data, the performance of the optimization framework could also be assessed using ship data with repeating destinations, where fouling patterns are likely more predictable.

\vspace{-0.3cm}
\addcontentsline{toc}{section}{Acknowledgments}
\section*{Acknowledgments}

The authors would like to thank Carisbrooke Shipping Limited\footnote{
\href{https://carisbrooke.co/}{Carisbrooke Shipping Limited~\faExternalLink}, 38 Medina Road, Cowes, Isle of Wight, PO31 7DA, United Kingdom.\\
},
represented by Captain Simon Merritt and Natalia Walker, for the pleasant collaboration and the provided data. 

Furthermore, the authors would like to thank Selin Damla Ahipa{\c{s}}ao{\u{g}}lu for helpful remarks and discussions on the development of the dynamic programming approach.

Finally, we would like to acknowledge James Court, Mehtap {\"O}zbey Arabac{\i}, John Robinson, and Stephen Pooley from the \gls{SRSG} for their contributions to the web app, and the \gls{SMMI} and \gls{EPSRC} for funding this development.

\vspace{-0.2cm}


\addcontentsline{toc}{section}{References}

\setcitestyle{numbers}
\renewcommand{\bibnumfmt}[1]{#1.}

\bibliographystyle{dcu}
\setlength{\bibsep}{4.0pt}
\bibliography{03_references}

@inproceedings{MEPC_Resolutions2023,
    author = {IMO},
    title = {RESOLUTION {MEPC}.377(80) 2023 IMO STRATEGY ON REDUCTION OF GHG EMISSIONS FROM SHIPS},
    booktitle = {Marine Environment Protection Committee (MEPC) Resolutions},
    year = {2023},
    note = {\url{https://wwwcdn.imo.org/localresources/en/KnowledgeCentre/IndexofIMOResolutions/MEPCDocuments/MEPC.377(80).pdf}},
}

@inproceedings{IMO2016,
    author = {IMO},
    title = {RESOLUTION {MEPC}.278(70) AMENDMENTS TO THE ANNEX OF THE PROTOCOL OF 1997 TO AMEND THE INTERNATIONAL CONVENTION FOR THE PREVENTION OF POLLUTION FROM SHIPS, 1973, AS MODIFIED BY THE PROTOCOL OF 1978 RELATING THERETO},
    booktitle = {Marine Environment Protection Committee (MEPC)},
    year = {2016},
    note = {\url{https://wwwcdn.imo.org/localresources/en/KnowledgeCentre/IndexofIMOResolutions/MEPCDocuments/MEPC.278(70).pdf}},
}

@inproceedings{MRV,
    author = {EU},
    title = {Commission Delegated Regulation (EU) 2016/2071},
    booktitle = {Official Journal of the European Union},
    year = {2016},
    note = {\url{http://data.europa.eu/eli/reg_del/2016/2071/oj}},
}

@article{Adland2018,
         title = {The energy efficiency effects of periodic ship hull cleaning},
         journal = {Journal of Cleaner Production},
         volume = {178},
         number = {},
         pages = {1--13},
         year = {2018},
         note = {\url{https://doi.org/10.1016/j.jclepro.2017.12.247}},
         author = {Roar Adland and Pierre Cariou and Haiying Jia and Francois-Charles Wolff}
}

@article{AGAND2023,
         title = {Fuel consumption prediction for a passenger ferry using machine learning and in-service data: A comparative study},
         journal = {Ocean Engineering},
         volume = {284},
         number = {},
         pages = {115271},
         year = {2023},
         note = {\url{https://doi.org/10.1016/j.oceaneng.2023.115271}},
         author = {Pedram Agand and Allison Kennedy and Trevor Harris and Chanwoo Bae and Mo Chen and Edward J. Park}
}

@Article{alsawaftah2022,
         AUTHOR = {AlSawaftah, Nour and Abuwatfa, Waad and Darwish, Naif and Husseini, Ghaleb A.},
         TITLE = {{A Review on Membrane Biofouling: Prediction, Characterization, and Mitigation}},
         JOURNAL = {Membranes},
         VOLUME = {12},
         YEAR = {2022},
         NUMBER = {12},
         pages = {1271},
         note = {\url{https://doi.org/10.3390/membranes12121271}}
}

@article{bakka2022,
         title={Estimating the effect of biofouling on ship shaft power based on sensor measurements},
         author={Bakka, Haakon and Rognebakke, Hanne and Glad, Ingrid and Haff, Ingrid Hob{\ae}k and Vanem, Erik},
         journal={Ship Technology Research},
         pages={209--221},
         year={2022},
         volume = {70},
         number = {3},
         note={\url{https://doi.org/10.1080/09377255.2022.2159108}}
}

@article{Bellman1955,
         title={{Equipment Replacement Policy}},
         author={Bellman, Richard},
         journal={Journal of the Society for Industrial and Applied Mathematics},
         volume={3},
         number={3},
         pages={133--136},
         year={1955},
         note={\url{https://doi.org/10.1137/0103011}}
}

@book{Bellman1957,
      title={{Dynamic Programming}},
      author={Bellman, Richard},
      year={1957},
      publisher={Princeton University Press},
      address={Princeton, NJ, USA}
}

@article{bergstra2012,
         author  = {James Bergstra and Yoshua Bengio},
         title   = {{Random Search for Hyper-Parameter Optimization}},
         journal = {Journal of Machine Learning Research},
         year    = {2012},
         volume  = {13},
         number  = {10},
         pages   = {281--305},
         note     = {\url{http://jmlr.org/papers/v13/bergstra12a.html}}
}

@book{Bertsekas2017,
  author    = {Bertsekas, Dimitri P.},
  title={{Dynamic Programming and Optimal Control}},
  volume    = {1},
  edition   = {4th},
  year      = {2017},
  publisher = {Athena Scientific},
  address   = {Belmont, MA, USA},
  isbn      = {978-1-886529-43-4}
}

@book{Bishop2006,
      author = {Bishop, Christopher M.},
      title = {{Pattern Recognition and Machine Learning}},
      series = {Information Science and Statistics},
      publisher = {Springer},
      year = {2006},
      note = {ISBN 978-0-387-31073-2},
      address = {New York},
      edition   = {1st}
}

@article{bloomfield2021,
         title={Automating the assessment of biofouling in images using expert agreement as a gold standard},
         author={Bloomfield, Nathaniel J and Wei, Susan and A. Woodham, Bartholomew and Wilkinson, Peter and Robinson, Andrew P},
         journal={Scientific Reports},
         volume={11},
         number = {},
         pages={2739},
         year={2021},
         note= {\url{https://doi.org/10.1038/s41598-021-81011-2}}
}

@article{botchkarev2019,
         title={{A New Typology Design of Performance Metrics to Measure Errors in Machine Learning Regression Algorithms}},
         author={Botchkarev, Alexei},
         journal={Interdisciplinary Journal of Information, Knowledge, and Management},
         volume={14},
         number = {},
         pages={45--79},
         year={2019},
         note={\url{https://doi.org/10.28945/4184}}
}

@incollection{Bressy2009,
              title = {18 - Tin-free self-polishing marine antifouling coatings},
              author = {C. Bressy and A. Margaillan and F. Fa{\"y} and I. Linossier and K. R{\'e}hel},
              editor = {Claire Hellio and Diego Yebra},
              booktitle = {Advances in Marine Antifouling Coatings and Technologies},
              publisher = {Woodhead Publishing},
              pages = {445-491},
              year = {2009},
              address = {Cambridge, UK},
              series = {Woodhead Publishing Series in Metals and Surface Engineering},
              note = {\url{https://doi.org/10.1533/9781845696313.3.445}},
}

@article{chai2014,
         title={{Root mean square error (RMSE) or mean absolute error (MAE)? -- Arguments against avoiding RMSE in the literature}},
         author={Chai, Tianfeng and Draxler, Roland R},
         journal={Geoscientific Model Development},
         volume={7},
         number={3},
         pages={1247--1250},
         year={2014},
         note={\url{https://doi.org/10.5194/gmd-7-1247-2014}}
}

@article{CHAPMAN2007,
         title = {Transport and climate change: a review},
         journal = {Journal of Transport Geography},
         volume = {15},
         number = {5},
         pages = {354--367},
         year = {2007},
         note = {\url{https://doi.org/10.1016/j.jtrangeo.2006.11.008}},
         author = {Lee Chapman},
}

@inproceedings{Chen2016,
               author = {Chen, Tianqi and Guestrin, Carlos},
               title = {{XGBoost: A Scalable Tree Boosting System}},
               year = {2016},
               publisher = {Association for Computing Machinery},
               address = {New York, NY},
               booktitle = {Proceedings of the 22nd ACM SIGKDD International Conference on Knowledge Discovery and Data Mining},
               pages = {785--794},
               numpages = {10},
               location = {San Francisco, California, USA},
               series = {KDD '16},
               note = {\url{https://doi.org/10.1145/2939672.2939785}},
}

@article{CORADDU2017,
         title = {Vessels fuel consumption forecast and trim optimisation: A data analytics perspective},
         journal = {Ocean Engineering},
         volume = {130},
         number = {},
         pages = {351--370},
         year = {2017},
         note = {\url{https://doi.org/10.1016/j.oceaneng.2016.11.058}},
         author = {Andrea Coraddu and Luca Oneto and Francesco Baldi and Davide Anguita}
}

@article{CORADDU2019,
         title = {A novelty detection approach to diagnosing hull and propeller fouling},
         journal = {Ocean Engineering},
         volume = {176},
         number = {},
         pages = {65--73},
         year = {2019},
         note = {\url{https://doi.org/10.1016/j.oceaneng.2019.01.054}},
         author = {Andrea Coraddu and Serena Lim and Luca Oneto and Kayvan Pazouki and Rose Norman and Alan John Murphy},
}

@article{CORADDU2019b,
         title = {Data-driven ship digital twin for estimating the speed loss caused by the marine fouling},
         journal = {Ocean Engineering},
         volume = {186},
         number = {},
         pages = {106063},
         year = {2019},
         note = {\url{https://doi.org/10.1016/j.oceaneng.2019.05.045}},
         author = {Andrea Coraddu and Luca Oneto and Francesco Baldi and Francesca Cipollini and Mehmet Atlar and Stefano Savio}
}

@book{Cormen2022,
      title={Introduction to Algorithms},
      author={Cormen, Thomas H and Leiserson, Charles E and Rivest, Ronald L and Stein, Clifford},
      year={2022},
      address = {Cambridge, Massachusetts},
      edition = {4th},
      publisher={The MIT Press},
      note={ISBN 978-0262046305}
}

@article{Dantzig1957,
         title={{Discrete-Variable Extremum Problems}},
         author={Dantzig, George B},
         journal={Operations Research},
         volume={5},
         number={2},
         pages={266--288},
         year={1957},
         publisher={INFORMS},
         note = {\url{https://doi.org/10.1287/opre.5.2.266}}
}

@book{Darwin1851,
      title = {A monograph on the sub-class Cirripedia, with figures of all the species},
      series ={Ray Society Publication},
      volume = {1},
      address = {London, UK},
      publisher = {Ray Society},
      author = {Darwin, Charles},
      year = {1851},
      note = {\url{https://doi.org/10.5962/bhl.title.2104}},
}

@book{Dasgupta2006,
      title={Algorithms},
      author={Dasgupta, Sanjoy and Papadimitriou, Christos H and Vazirani, Umesh Virkumar},
      year={2006},
      edition = {1st},
      address = {New York, NY, USA},
      publisher={McGraw-Hill Higher Education},
      note={ISBN 978-0073523408}
}

@inproceedings{deHaas2023,
               title={{Power Increase due to Marine Biofouling: a Grey-box Model Approach}},
               author={de Haas, Matthew and Coraddu, Andrea and El Mouhandiz, Abdel-Ali and Dimitra Charisi, Nikoleta and Kana, Austin A.},
               booktitle={Proceedings of the 4th International Conference on Modelling and Optimisation of Ship Energy Systems},
               year={2023},
               note = {\url{https://doi.org/10.59490/moses.2023.661}}
}

@article{DEMIREL2017,
         title = {Predicting the effect of biofouling on ship resistance using CFD},
         journal = {Applied Ocean Research},
         volume = {62},
         number = {},
         pages = {100--118},
         year = {2017},
         note = {\url{https://doi.org/10.1016/j.apor.2016.12.003}},
         author = {Yigit Kemal Demirel and Osman Turan and Atilla Incecik}
}

@article{DEMIREL2019,
         title = {Practical added resistance diagrams to predict fouling impact on ship performance},
         journal = {Ocean Engineering},
         volume = {186},
         number = {},
         pages = {106112},
         year = {2019},
         note = {\url{https://doi.org/10.1016/j.oceaneng.2019.106112}},
         author = {Yigit Kemal Demirel and Soonseok Song and Osman Turan and Atilla Incecik}
}

@Article{Din2023,
         AUTHOR = {Din, Ashraf Ud and Ur Rahman, Imran and Vega-Muñoz, Alejandro and Elahi, Ehsan and Salazar-Sepúlveda, Guido and Contreras-Barraza, Nicolás and Alhrahsheh, Rakan Radi},
         TITLE = {How Sustainable Transportation Can Utilize Climate Change Technologies to Mitigate Climate Change},
         JOURNAL = {Sustainability},
         VOLUME = {15},
         YEAR = {2023},
         NUMBER = {12},
         ARTICLE-NUMBER = {9710},
         note = {\url{https://doi.org/10.3390/su15129710}}
}

@misc{Django,
  author = {{Django Software Foundation}},
  title = {Django},
  address = {Lawrence, Kansas},
  year = {2019},
  note = {Version 5.0.6. \url{https://djangoproject.com}}
}

@article{dunn1961,
         title={{Multiple Comparisons Among Means}},
         author={Dunn, Olive Jean},
         journal={Journal of the American Statistical Association},
         volume={56},
         number={293},
         pages={52--64},
         year={1961},
         note={\url{https://doi.org/10.1080/01621459.1961.10482090}}
}

@article{Fagerholt2010,
         author = {Fagerholt, K and Laporte, G and Norstad, I},
         title = {Reducing fuel emissions by optimizing speed on shipping routes},
         journal = {Journal of the Operational Research Society},
         volume = {61},
         number = {3},
         pages = {523--529},
         year = {2010},
         note = {\url{https://doi.org/10.1057/jors.2009.77}}
}

@book{Fahrmeir2021,
      title={{Regression: Models, Methods and Applications}},
      author={Fahrmeir, Ludwig and Kneib, Thomas and Lang, Stefan and Marx, Brian D.},
      edition={2nd},
      year={2021},
      publisher={Springer},
      address={Berlin},
      note={\url{https://doi.org/10.1007/978-3-662-63882-8}}
}

@article{first2021,
         title={Rapid quantification of biofouling with an inexpensive, underwater camera and image analysis},
         author={First, Matthew R and Riley, Scott C and Islam, Kazi Aminul and Hill, Victoria and Li, Jiang and Zimmerman, Richard C and Drake, Lisa A},
         journal={Management of Biological Invasions},
         volume={12},
         number={3},
         pages={599--617},
         year={2021},
         note={\url{https://doi.org/10.3391/mbi.2021.12.3.06}}
}

@article{Floerl2003,
         author = {Floerl, Oliver and Inglis, Graeme J.},
         title = {Boat harbour design can exacerbate hull fouling},
         journal = {Austral Ecology},
         volume = {28},
         number = {2},
         pages = {116--127},
         note = {\url{https://doi.org/10.1046/j.1442-9993.2003.01254.x}},
         year = {2003}
}

@book{fox2015,
      title={{Applied Regression Analysis and Generalized Linear Models}},
      author={Fox, John},
      year={2015},
      edition= {3rd},
      publisher={SAGE Publications},
      address={Thousand Oaks, CA}
}

@book{geron2022,
      title={{Hands-on Machine Learning with Scikit-Learn, Keras, and TensorFlow}},
      author={G{\'e}ron, Aur{\'e}lien},
      year={2022},
      edition = {3rd},
      publisher={O'Reilly Media},
      address = {Sebastopol, CA}
}

@book{Gianfagna2021,
      author = {Gianfagna, Leonida and Di Cecco, Antonio},
      title = {{Explainable AI with Python}},
      publisher = {Springer},
      year = {2021},
      note = {\url{https://doi.org/10.1007/978-3-030-68640-6}},
      address = {Cham},
      edition   = {1st}
}

@article{GKEREKOS2019,
         title = {{Machine learning models for predicting ship main engine Fuel Oil Consumption: A comparative study}},
         journal = {Ocean Engineering},
         volume = {188},
         pages = {106282},
         year = {2019},
         note = {\url{https://doi.org/10.1016/j.oceaneng.2019.106282}},
         author = {Christos Gkerekos and Iraklis Lazakis and Gerasimos Theotokatos}
}

@inproceedings{GorenHuber2017,
               author = {Goren Huber, Lilach and Kunz, Simon and Dettling, Marcel},
               title = {{Condition-Based Maintenance Decision Making: a Practical Approach for Marine Vessels}},
               year = {2017},
               publisher = {Jost Institute for Tribotechnology},
               address = {Lancashire, UK},
               booktitle = {30th Conference for Condition Monitoring and Diagnostic Engineering Managemen (COMADEM)},
               pages = {377--386},
               note = {\url{https://digitalcollection.zhaw.ch/handle/11475/8014}},
}

@misc{Granhag2023,
      title={Best practice for cleaning of ship hulls},
      author={Granhag, Lena and Javadi, Mehran and Ytreberg, Erik},
      year={2023},
      note={\emph{Swedish Maritime and Water Authority}.  \url{https://research.chalmers.se/publication/535739}}
}

@article{GUPTA2022,
         title = {Ship performance monitoring using machine-learning},
         journal = {Ocean Engineering},
         volume = {254},
         number = {},
         pages = {111094},
         year = {2022},
         note = {\url{https://doi.org/10.1016/j.oceaneng.2022.111094}},
         author = {Prateek Gupta and Adil Rasheed and Sverre Steen}
}

@article{handayani2023,
         title={{Navigating Energy Efficiency: A Multifaceted Interpretability of Fuel Oil Consumption Prediction in Cargo Container Vessel Considering the Operational and Environmental Factors}},
         author={Handayani, Melia Putri and Kim, Hyunju and Lee, Sangbong and Lee, Jihwan},
         journal={Journal of Marine Science and Engineering},
         volume={11},
         number={11},
         pages={2165},
         year={2023},
         note={\url{https://doi.org/10.3390/jmse11112165}}
}

@misc{hansen2012gameover,
      author    = {James Hansen},
      title     = {Game Over for the Climate},
      publisher = {The New York Times},
      year      = {2012},
      month     = {May},
      day       = {9},
      note      = {The New York Times. Accessed: 2025-12-29. \url{https://www.nytimes.com/2012/05/10/opinion/game-over-for-the-climate.html}},
}

@book{hastie2009,
      title={{The Elements of Statistical Learning: Data Mining, Inference, and Prediction}},
      author={Hastie, Trevor and Tibshirani, Robert and Friedman, Jerome H},
      edition={2nd},
      series = {Springer Series in Statistics},
      year={2009},
      note={\url{https://doi.org/10.1007/b94608}},
      address ={New York, NY},
      publisher={Springer}
}

@article{HUA2018,
         title = {En-route operated hydroblasting system for counteracting biofouling on ship hull},
         journal = {Ocean Engineering},
         volume = {152},
         number = {},
         pages = {249--256},
         year = {2018},
         note = {\url{https://doi.org/10.1016/j.oceaneng.2018.01.050}},
         author = {Jian Hua and Yin-Sen Chiu and Chin-Yang Tsai}
}

@article{HUANG2022,
         title = {Machine learning in sustainable ship design and operation: A review},
         journal = {Ocean Engineering},
         volume = {266},
         number = {},
         pages = {112907},
         year = {2022},
         note = {\url{https://doi.org/10.1016/j.oceaneng.2022.112907}},
         author = {Luofeng Huang and Blanca Pena and Yuanchang Liu and Enrico Anderlini}
}

@Article{Huotari2021,
         AUTHOR = {Huotari, Janne and Manderbacka, Teemu and Ritari, Antti and Tammi, Kari},
         TITLE = {{Convex Optimisation Model for Ship Speed Profile: Optimisation under Fixed Schedule}},
         JOURNAL = {Journal of Marine Science and Engineering},
         VOLUME = {9},
         YEAR = {2021},
         NUMBER = {7},
         pages = {730},
         note = {\url{https://doi.org/10.3390/jmse9070730}},
}

@online{IMO2023,
        author    = {{IMO}},
        title     = {{Marine Environment: Biofouling [Online]}},
        year      = {2023},
        note      = {Accessed: 2026-01-13, \url{https://www.imo.org/en/ourwork/environment/pages/biofouling.aspx}}
}

@book{James2021,
      title={{An Introduction to Statistical Learning with Applications in R}},
      author={James, Gareth and Witten, Daniela and Hastie, Trevor and Tibshirani, Robert},
      edition={2nd},
      year={2021},
      publisher={Springer},
      address={New York, NY},
      note={\url{https://doi.org/10.1007/978-1-0716-1418-1}}
}

@book{Kamath2021,
      title={{Explainable Artificial Intelligence: An Introduction to Interpretable Machine Learning}},
      author={Kamath, Uday and Liu, John},
      edition={1st},
      year={2021},
      publisher={Springer},
      address={Cham},
      note={\url{https://doi.org/10.1007/978-3-030-83356-5}}
}

@Incollection{Karp1972,
              author="Karp, Richard M.",
              editor="Miller, Raymond E. and Thatcher, James W. and Bohlinger, Jean D.",
              title="Reducibility among Combinatorial Problems",
              bookTitle="Complexity of Computer Computations",
              series="The IBM Research Symposia Series",
              year="1972",
              publisher="Springer",
              address="Boston, MA",
              edition="1st",
              pages="85--103",
              note={\url{https://doi.org/10.1007/978-1-4684-2001-2_9}}
}

@book{Kellerer2004,
      title={Knapsack Problems},
      author={Kellerer, H. and Pferschy, U. and Pisinger, D.},
      edition = {1st},
      year={2004},
      publisher={Springer},
      address = {Berlin, Heidelberg},
      note={\url{https://doi.org/10.1007/978-3-540-24777-7}},
}

@article{LANG2022110387,
         title = {Comparison of supervised machine learning methods to predict ship propulsion power at sea},
         journal = {Ocean Engineering},
         volume = {245},
         number = {},
         pages = {110387},
         year = {2022},
         note = {\url{https://doi.org/10.1016/j.oceaneng.2021.110387}},
         author = {Xiao Lang and Da Wu and Wengang Mao}
}

@article{LAURIE2021,
         title = {Machine learning for shaft power prediction and analysis of fouling related performance deterioration},
         journal = {Ocean Engineering},
         volume = {234},
         number = {},
         pages = {108886},
         year = {2021},
         note = {\url{https://doi.org/10.1016/j.oceaneng.2021.108886}},
         author = {Anastasia Laurie and Enrico Anderlini and Jesper Dietz and Giles Thomas}
}

@inproceedings{Lewis2009,
               author = {Lewis, John A. and Coutts, Ashley D. M.},
               editor={D{\"u}rr, S. and Thomason, J.C.},
               publisher = {John Wiley \& Sons},
               address = {Hoboken, NJ},
               title = {{Biofouling Invasions}},
               booktitle = {Biofouling},
               chapter = {24},
               pages = {348--365},
               note = {\url{https://doi.org/10.1002/9781444315462.ch24}},
               year = {2009},
}

@ARTICLE{Loyola2019,
         author={Loyola-González, Octavio},
         journal={IEEE Access}, 
         title={{Black-Box vs. White-Box: Understanding Their Advantages and Weaknesses From a Practical Point of View}}, 
         year={2019},
         volume={7},
         number = {},
         pages={154096--154113},
         note={\url{https://doi.org/10.1109/ACCESS.2019.2949286}}
}

@inproceedings{Lundberg2017,
               author = {Lundberg, Scott M. and Lee, Su-In},
               title = {A unified approach to interpreting model predictions},
               year = {2017},
               isbn = {9781510860964},
               publisher = {Curran Associates Inc.},
               address = {Red Hook, NY},
               booktitle = {Proceedings of the 31st International Conference on Neural Information Processing Systems},
               pages = {4768--4777},
               numpages = {10},
               series = {NIPS'17}
}

@article{ma2023,
         title={{An Interpretable Gray Box Model for Ship Fuel Consumption Prediction Based on the SHAP Framework}},
         author={Ma, Yiji and Zhao, Yuzhe and Yu, Jiahao and Zhou, Jingmiao and Kuang, Haibo},
         journal={Journal of Marine Science and Engineering},
         volume={11},
         number={5},
         pages={1059},
         year={2023},
         note={\url{https://doi.org/10.3390/jmse11051059}},
}

@article{Mathews1896,
         title={{On the Partition of Numbers}},
         author={Mathews, George B},
         journal={Proceedings of the London Mathematical Society},
         volume={s1-28},
         number={1},
         pages={486--490},
         year={1896},
         note = {\url{https://doi.org/10.1112/plms/s1-28.1.486}}
}

@book{molnar2022,
      title      = {Interpretable Machine Learning},
      author     = {Christoph Molnar},
      year       = {2022},
      subtitle   = {A Guide for Making Black Box Models Explainable},
      publisher  = {Independently published},
      edition    = {2nd},
      note       = {\url{https://christophm.github.io/interpretable-ml-book}}
}

@article{Molnar2008,
         title={Assessing the global threat of invasive species to marine biodiversity},
         author={Molnar, Jennifer L and Gamboa, Rebecca L and Revenga, Carmen and Spalding, Mark D},
         journal={Frontiers in Ecology and the Environment},
         volume={6},
         number={9},
         pages={485--492},
         year={2008},
         note = {\url{https://doi.org/10.1890/070064}}
}

@article{moreira2021,
         title={{Neural Network Approach for Predicting Ship Speed and Fuel Consumption}},
         author={Moreira, L{\'u}cia and Vettor, Roberto and Guedes Soares, Carlos},
         journal={Journal of Marine Science and Engineering},
         volume={9},
         number = {2},
         pages={119},
         year={2021},
         note={\url{https://doi.org/10.3390/jmse9020119}}
}

@incollection{Nakamura2025,
              author="Nakamura, Junji and Ishida, Masayoshi",
              title={{Global Warming and CO2 Emissions}},
              bookTitle="Blueprint for a Methanol Society: Toward Carbon-Neutrality",
              year="2025",
              publisher="Springer Nature",
              edition={1st},
              address="Singapore",
              pages="3--26",
              note={\url{https://doi.org/10.1007/978-981-95-3465-4_1}}
}

@article{Nikolaidis2022,
         author = {Nikolaidis, George and Themelis, Nikos},
         title = {Examining the performance of retrofit measures in real ship operation using data-driven models},
         journal = {Ship Technology Research},
         volume = {69},
         number = {3},
         pages = {170--180},
         year = {2022},
         note = {\url{https://doi.org/10.1080/09377255.2022.2109327}}
}

@article{Oliveira2022,
         title = {A novel tool for cost and emission reduction related to ship underwater hull maintenance},
         journal = {Journal of Cleaner Production},
         volume = {356},
         number = {},
         pages = {131882},
         year = {2022},
         author = {Olivera, D. R. and Lagerstr\"om, M. and Granhag, L. and Werner, S. and Larsson, A. I. and Ytreberg, E.},
         note = {\url{https://doi.org/10.1016/j.jclepro.2022.131882}},
}

@book{Papadimitriou1994,
      author    = {Papadimitriou, Christos H.},
      title     = {Computational Complexity},
      publisher = {Addison-Wesley},
      address   = {Reading, Mass.},
      edition   = {1st},
      year      = {1994},
      note      = {ISBN 978-0201530827}
}

@article{parviainen2018,
         title={How can stakeholders promote environmental and social responsibility in the shipping industry?},
         author={Parviainen, Tuuli and Lehikoinen, Annukka and Kuikka, Sakari and Haapasaari, P{\"a}ivi},
         journal={WMU Journal of Maritime Affairs},
         volume={17},
         number = {},
         pages={49--70},
         year={2018},
         note = {\url{https://doi.org/10.1007/s13437-017-0134-z}}
}

@article{Pedregosa2011,
         author = {Pedregosa, Fabian and Varoquaux, Ga\"{e}l and Gramfort, Alexandre and Michel, Vincent and Thirion, Bertrand and Grisel, Olivier and Blondel, Mathieu and Prettenhofer, Peter and Weiss, Ron and Dubourg, Vincent and Vanderplas, Jake and Passos, Alexandre and Cournapeau, David and Brucher, Matthieu and Perrot, Matthieu and Duchesnay, \'{E}douard},
         title = {Scikit-learn: Machine Learning in Python},
         year = {2011},
         publisher = {JMLR.org},
         volume = {12},
         number = {},
         journal = {Journal of Machine Learning Research},
         pages = {2825--2830},
         numpages = {6},
         note = {\url{https://doi.org/10.5555/1953048.2078195}}
}

@misc{Plotly, 
    author = {{Plotly Technologies Inc.}}, 
    title = {Collaborative data science}, 
    publisher = {Plotly Technologies Inc.}, 
    address = {Montreal, QC}, 
    year = {2015}, 
    note = {Version 5.22.0. \url{https://plot.ly}}
}

@book{ravichandiran2019,
      title={{Hands-On Deep Learning Algorithms with Python}},
      author={Ravichandiran, Sudharsan},
      edition={1st},
      year={2019},
      address ={Birmingham},
      publisher={Packt Publishing},
      note = {ISBN 978-1789344158}
}

@book{Rothman2020,
      title={{Hands-On Explainable AI (XAI) with Python}},
      author={Rothman, Denis},
      year={2020},
      edition = {1st},
      publisher={Packt Publishing},
      address={Birmingham},
      note ={ISBN 978-1800202764}
}

@incollection{Rumala2023,
              title = {{How You Split Matters: Data Leakage and Subject Characteristics Studies in Longitudinal Brain MRI Analysis}},
              author = {Rumala, Dewinda},
              booktitle = {Medical Image Computing and Computer Assisted Intervention -- MICCAI 2023},
              series = {Lecture Notes in Computer Science},
              volume = {14227},
              pages = {235--245},
              year = {2023},
              publisher = {Springer},
              address = {Cham},
              note = {\url{https://doi.org/10.1007/978-3-031-45249-9_23}}
}

@article{Schultz2011,
         title={Economic impact of biofouling on a naval surface ship},
         author={Schultz, Michael P and Bendick, JA and Holm, ER and Hertel, WM},
         journal={Biofouling},
         volume={27},
         number={1},
         pages={87--98},
         year={2011},
         note = {\url{https://doi.org/10.1080/08927014.2010.542809}}
}

@article{Smola2004,
         title={A tutorial on support vector regression},
         author={Smola, Alex J and Sch{\"o}lkopf, Bernhard},
         journal={Statistics and Computing},
         volume={14},
         number={3},
         pages={199--222},
         year={2004},
         note={\url{https://doi.org/10.1023/B:STCO.0000035301.49549.88}}
}

@article{soner2019,
         title={Statistical modelling of ship operational performance monitoring problem},
         author={Soner, Omer and Akyuz, Emre and Celik, Metin},
         journal={Journal of Marine Science and Technology},
         volume={24},
         pages={543--552},
         year={2019},
         note={\url{https://doi.org/10.1007/s00773-018-0574-y}}
}

@article{song2020,
         title={{Review of Underwater Ship Hull Cleaning Technologies}},
         author={Song, C and Cui, W},
         journal={Journal of Marine Science and Application},
         volume={19},
         number = {},
         pages={415--429},
         year={2020},
         note={\url{https://doi.org/10.1007/s11804-020-00157-z}}
}

@article{Srinivasan2007,
         title={Managing the use of copper-based antifouling paints},
         author={Srinivasan, Mridula and Swain, Geoffrey W},
         journal={Environmental Management},
         volume={39},
         pages={423--441},
         year={2007},
         note = {\url{https://doi.org/10.1007/s00267-005-0030-8}}
}

@ARTICLE{Tamburri2020,
         AUTHOR={Tamburri, Mario N.  and Davidson, Ian C.  and First, Matthew R.  and Scianni, Christopher  and Newcomer, Katherine  and Inglis, Graeme J.  and Georgiades, Eugene T.  and Barnes, Janet M.  and Ruiz, Gregory M. },
         TITLE={{In-Water Cleaning and Capture to Remove Ship Biofouling: An Initial Evaluation of Efficacy and Environmental Safety}},
         JOURNAL={Frontiers in Marine Science},
         VOLUME={7},
         YEAR={2020},
         NOTE={\url{https://doi.org/10.3389/fmars.2020.00437}}
}

@article{Tibshirani1996,
         author = {Tibshirani, Robert},
         title = {{Regression Shrinkage and Selection Via the Lasso}},
         journal = {Journal of the Royal Statistical Society: Series B (Methodological)},
         volume = {58},
         number = {1},
         pages = {267--288},
         year = {1996},
         note = {\url{https://doi.org/10.1111/j.2517-6161.1996.tb02080.x}}
}

@article{Townsin2003,
         author = {Townsin, R L},
         title = {{The Ship Hull Fouling Penalty}},
         journal = {Biofouling},
         volume = {19},
         number = {sup1},
         pages = {9--15},
         year = {2003},
         publisher = {Taylor \& Francis},
         note = {\url{https://doi.org/10.1080/0892701031000088535}}
}

@article{UZUN2019,
         title = {Time-dependent biofouling growth model for predicting the effects of biofouling on ship resistance and powering},
         journal = {Ocean Engineering},
         volume = {191},
         number = {},
         pages = {106432},
         year = {2019},
         note = {\url{https://doi.org/10.1016/j.oceaneng.2019.106432}},
         author = {Dogancan Uzun and Yigit Kemal Demirel and Andrea Coraddu and Osman Turan}
}

@article{VALCHEV2022,
         title = {Numerical methods for monitoring and evaluating the biofouling state and effects on vessels’ hull and propeller performance: A review},
         journal = {Ocean Engineering},
         volume = {251},
         number = {},
         pages = {110883},
         year = {2022},
         note = {\url{https://doi.org/10.1016/j.oceaneng.2022.110883}},
         author = {Iliya Valchev and Andrea Coraddu and Miltiadis Kalikatzarakis and Rinze Geertsma and Luca Oneto}
}

@book{wade2020,
      title={{Hands-On Gradient Boosting with XGBoost and scikit-learn}},
      author={Wade, Corey},
      year={2020},
      edition = {1st},
      note = {ISBN 978-1839218354},
      publisher={Packt Publishing},
      address={Birmingham}
}

@article{WANG2023,
         title = {Innovative approaches to addressing the tradeoff between interpretability and accuracy in ship fuel consumption prediction},
         journal = {Transportation Research Part C: Emerging Technologies},
         volume = {157},
         number = {},
         pages = {104361},
         year = {2023},
         note= {\url{https://doi.org/10.1016/j.trc.2023.104361}},
         author = {Haoqing Wang and Ran Yan and Shuaian Wang and Lu Zhen}
}

@book{Wichert2021,
      author = {Wichert, Andreas and Sa-Couto, Luis},
      title = {{Machine Learning -- A Journey to Deep Learning}},
      publisher = {World Scientific},
      year = {2021},
      note = {\url{https://doi.org/10.1142/12201}},
      address = {Singapore},
      edition   = {1st}
}

@article{Wu2022,
         title={Economic analysis of ship operation using a new antifouling strategy},
         author={Wu, Y and Hua, J and Wu, D L},
         journal={Ocean Engineering},
         volume={266},
         number = {},
         pages={113038},
         note = {\url{https://doi.org/10.1016/j.oceaneng.2022.113038}},
         year={2022}
}

@article{YAN2020,
         title = {Development of a two-stage ship fuel consumption prediction and reduction model for a dry bulk ship},
         journal = {Transportation Research Part E: Logistics and Transportation Review},
         volume = {138},
         number = {},
         pages = {101930},
         year = {2020},
         note = {\url{https://doi.org/10.1016/j.tre.2020.101930}},
         author = {Ran Yan and Shuaian Wang and Yuquan Du},
}

@article{ZHOU2022,
         title = {An adaptive hyper parameter tuning model for ship fuel consumption prediction under complex maritime environments},
         journal = {Journal of Ocean Engineering and Science},
         volume = {7},
         number = {3},
         pages = {255--263},
         year = {2022},
         note = {\url{https://doi.org/10.1016/j.joes.2021.08.007}},
         author = {Tianrui Zhou and Qinyou Hu and Zhihui Hu and Rong Zhen},
}

@software{Zippenfenig2023,
          title = {{Open-Meteo.com Weather API [Computer software]. Zenodo}},
          author = {Zippenfenig, Patrick},
          note = {\url{https://doi.org/10.5281/zenodo.14582479}},
          licence = {CC-BY-4.0},
          year = {2023},
          copyright = {Creative Commons Attribution 4.0 International}
}

@article{Zis2020,
         title = {Ship weather routing: A taxonomy and survey},
         journal = {Ocean Engineering},
         volume = {213},
         number = {},
         pages = {107697},
         year = {2020},
         author = {Thalis P. V. Zis and Harilaos N. Psaraftis and Li Ding},
         note = {\url{https://doi.org/10.1016/j.oceaneng.2020.107697}}
}

\newpage
\appendix

\renewcommand\thefigure{\thesection.\arabic{figure}}   
\setcounter{figure}{0}    

\renewcommand\thetable{\thesection.\arabic{table}}   
\setcounter{table}{0}   

\setcitestyle{authoryear,open={(},close={)}}

\section{Appendix}\label{sec:appendix}

\counterwithin{theo}{section}


\subsection{Oracle Complexity of CLEANING-1}
\label{subsec:oracle_complexity}

In this subsection of the Appendix, we treat the fuel cost functions $f_1, \ldots, f_n$ of \eqref{eq:cleaning} as black boxes and count the number of their evaluations that are required to find a global minimum.
Proposition~\ref{prop:oracle} shows that the brute-force search in Algorithm~\ref{alg:brute_force} is optimal up to a factor of about $n/2$ in this respect, so that every exact algorithm requires an exponential number of evaluations.

\vspace{0.2cm}


\begin{prop}[Oracle complexity]{prop:oracle}
    Consider instances of \eqref{eq:cleaning} in which $f_1, \ldots, f_n$ can only be accessed through function evaluations.
    \begin{enumerate}[label=(\roman*), noitemsep, topsep=0pt]
        \item Algorithm~\ref{alg:brute_force} finds a global minimum of \eqref{eq:cleaning} with exactly $n \cdot 2^n$ function evaluations.
        \item Every deterministic algorithm that finds a global minimum of \eqref{eq:cleaning} for all instances evaluates, on every instance, each $f_j$ at each of its $2^j$ arguments.
        In particular, it requires at least $\sum_{j=1}^{n} 2^j = 2^{n+1} - 2$ function evaluations.
        \item The bound in (ii) is tight, i.e., there exists an algorithm that finds a global minimum of \eqref{eq:cleaning} with exactly $2^{n+1} - 2$ function evaluations.
    \end{enumerate}
\end{prop}


\begin{prf}[Proposition~\ref{prop:oracle}]{prf:oracle}
    (i) Algorithm~\ref{alg:brute_force} loops through all $2^n$ schedules and, for each of them, evaluates $f_1, \ldots, f_n$ once in Line~\ref{line:brute_obj_evaluation}.
    Thus, it performs exactly $n \cdot 2^n$ evaluations and returns a schedule with the lowest objective value.

    (ii) Let $\mathcal{A}$ be a deterministic algorithm that returns a global minimum for all instances, and let $I$ be an instance with fuel cost functions $f_1, \ldots, f_n$ and objective function $J_I$.
    Suppose, towards a contradiction, that $\mathcal{A}$, when run on $I$, never evaluates $f_j$ at some argument $\mathbf{h} \in \{0,1\}^j$, and let $\zz^*$ be the schedule returned by $\mathcal{A}$.
    We construct an instance $I'$ that coincides with $I$ except for the value of $f_j$ at $\mathbf{h}$, which is replaced by $f_j(\mathbf{h}) + \varepsilon$ for some $\varepsilon \in \mathbb{R}$.
    Then, $J_{I'}(\zz) = J_I(\zz) + \varepsilon$ if $\zz_{1:j} = \mathbf{h}$, and $J_{I'}(\zz) = J_I(\zz)$ otherwise.
    Since $\mathcal{A}$ is deterministic and all function evaluations it performs return the same values on $I$ and $I'$, it performs the same steps on both instances and also returns $\zz^*$ for $I'$.
    If $\zz^*_{1:j} \neq \mathbf{h}$, we choose any $\zz' \in \{0,1\}^n$ with $\zz'_{1:j} = \mathbf{h}$ and set $\varepsilon \coloneqq J_I(\zz^*) - J_I(\zz') - 1$, which yields $J_{I'}(\zz') = J_I(\zz^*) - 1 < J_{I'}(\zz^*)$.
    If $\zz^*_{1:j} = \mathbf{h}$, we choose any $\zz' \in \{0,1\}^n$ with $\zz'_{1:j} \neq \mathbf{h}$, e.g., by flipping $z^*_j$, and set $\varepsilon \coloneqq J_I(\zz') - J_I(\zz^*) + 1$, which yields $J_{I'}(\zz^*) = J_I(\zz') + 1 > J_{I'}(\zz')$.
    In both cases, $\zz^*$ is not a global minimum for $I'$, which contradicts the assumption on $\mathcal{A}$.
    Hence, $\mathcal{A}$ evaluates each $f_j$ at all of its $2^j$ arguments, which amounts to $\sum_{j=1}^{n} 2^j = 2^{n+1} - 2$ evaluations.

    (iii) We first evaluate each $f_j$ once at each of its $2^j$ arguments and store the results.
    Afterwards, we run Algorithm~\ref{alg:brute_force}, where the stored values are looked up instead of evaluating $f_1, \ldots, f_n$ again.
    This algorithm finds a global minimum with exactly $2^{n+1} - 2$ evaluations.
\end{prf}


\subsection{Linear Regression}
\label{subsec:Linear_Regression}

A widely used statistical method for modeling the relationship between a dependent variable~$\mathbf{y}$ and a set of independent variables is the \gls{LinReg} [\citet[p. 24]{Fahrmeir2021}]. 
The main idea of the approach is to express the response variable as a non-deterministic linear function of the covariates. 
For an empirical dataset, this can be translates into equations of the following form
\begin{equation*}
    \mathbf{y} = \mathbf{X} \boldsymbol{\beta} + \boldsymbol{\varepsilon},
\end{equation*}
where $ \boldsymbol{\beta} \in \mathbb{R}^k$ denotes the true unknown regression coefficients, and $\boldsymbol{\varepsilon} \in \mathbb{R}^m$ is the unknown random error term. 
To estimate the regression coefficients, a popular approach is to minimize the sum of squared residuals that is given as
\begin{equation*}
    \underset{\boldsymbol{\beta}}{\text{minimize}} \ \sum_{t=1}^m \bigl( y_t - \mathbf{x}_t \boldsymbol{\beta} \bigr)^2,
\end{equation*}
and is also known as Ordinary Least Squares [\citet[p. 26]{Fahrmeir2021}]. 
The main advantages of the \gls{LinReg} are its simplicity, and interpretability [\citet[p. 43]{hastie2009}].


\subsection{Least Absolute Shrinkage and Selection Operator}

The \gls{LASSO} regression introduced by \citet{Tibshirani1996} is another common statistical method to linearly model the connection between a set of covariates and a continuous outcome. 
To estimate the regression coefficients, \gls{LASSO} minimizes the penalized least squares criterion, given as
\begin{equation*}
    \underset{\boldsymbol{\beta}}{\text{minimize}} \ \frac{1}{2} \cdot \sum_{t=1}^m \bigl(y_t - \mathbf{x}_t \boldsymbol{\beta} \bigr)^2 + \lambda \cdot \sum_{l=1}^k \lvert \beta_l \rvert,
\end{equation*}
where $\lambda \cdot \sum_{l=1}^k \lvert \beta_l \rvert$ with $\lambda > 0$ represents the penalty term. 
This approach is similar to \gls{LinReg} but additionally incorporates regularization to prevent overfitting and automatic feature selection [\citet[p. 170]{Fahrmeir2021}]. 
Note that the user has to specify the parameter $\lambda$ that decides over the strength of the penalization. 
In our numerical experiments, we perform a random search of $\lambda$ values between 0.1 and 5.


\subsection{Multilayer Perceptron}

A \gls{ML} algorithm that allows for more complex relationships between input and output is the \gls{MLP}. 
The main idea of this method is to use a series of non-linear transformations instead of just calculating weighted sums of the independent variables [\citet[p. 427]{Wichert2021}]. 
The general structure of \gls{MLP} then can be summarized based on three components: input layer, hidden layer(s), and output layer [\citet[p. 11]{ravichandiran2019}]. 
While the former component is only used to enter the variables, the hidden layers perform the non-linear transformations of the features based on activation functions [\citet[pp. 12--13]{ravichandiran2019}]. 
For example, the procedure at the first hidden layer can be described as
\begin{equation*}
    h_1 \bigl(\mathbf{X} \mathbf{W}_1 + \mathbf{B}_1 \bigr) = \mathbf{Z}_1,
\end{equation*}
where $h_1: \mathbb{R}^{m \times k} \rightarrow \mathbb{R}^{m \times u_1}$ denotes the activation function with $u_1$ being the number of neurons in the layer, $\mathbf{W}_1 \in \mathbb{R}^{k \times u_1}$ represents the weight matrix, and $\mathbf{B}_1 \in \mathbb{R}^{m \times u_1}$ is the so-called bias term. 
Furthermore, $\mathbf{Z}_1 \in \mathbb{R}^{m \times u_1}$ denotes the output of the first layer which replaces $\mathbf{X}$ as input in the second layer. 
The procedure then continues until the output layer is reached, which has to ensure that the outcome of the model has the correct dimensions. 
Due to this set-up, the \gls{MLP} is generally more powerful than the two linear models, but it comes at the price that the user must specify several hyperparameters. 
Unfortunately, the optimal choices for these parameters are often unknown as they highly depend on the structure of the underlying problem. 
For example, there exist several options for the activation function, the number of hidden layers, and the amount of neurons per layer that significantly influence the performance of \gls{MLP}. 
In practice, the hyperparameters are often selected by evaluating different combinations.
In our numerical experiments, we try each of the following hyperparameter combinations:
\begin{verbatim}
 "hidden_layer_sizes": [(50, 50, 50), (50, 100, 50), (25, 50, 100, 100, 50, 25)],
 "activation": ["relu", "tanh"],
\end{verbatim}

\subsection{Support Vector Regression}
\label{subsec:Support_Vector_Regression}

Another alternative approach is \gls{SVR} [\citet{Smola2004}], which aims to find a function that approximates the data within a specified tolerance, rather than minimizing squared residuals. 
Formally, \gls{SVR} solves the following optimization problem:
\begin{align*}
    \label{eq:SVR}
    \underset{\boldsymbol{\beta}, b, \boldsymbol{\xi}, \boldsymbol{\xi}^*}{\text{minimize}}\quad
    & \frac{1}{2} \|\boldsymbol{\beta}\|^2 + C \cdot  \sum_{t=1}^m (\xi_t + \xi_t^*),\\
    \text{subject to}\quad
    &y_t - (\mathbf{x}_t \boldsymbol{\beta} + b) \leq \varepsilon + \xi_t, \\
    &(\mathbf{x}_t \boldsymbol{\beta} + b) - y_t \leq \varepsilon + \xi_t^*, \\
    &\xi_t, \xi_t^* \geq 0,
\end{align*}
where $\xi_t, \xi_t^* \in \mathbb{R}$ are slack variables allowing violations of the $\varepsilon$-insensitive margin, and $C > 0$ is a regularization parameter controlling the trade-off between model complexity and tolerance to deviations. 
The main advantages of \gls{SVR} comes from its robustness to outliers, and its ability to handle non-linear relationships through its dual formulation and kernel trick [\citet[p. 193]{Smola2004}].
However, this approach also introduces several hyperparameters.
In our numerical experiments, we use the \gls{SVR} implementation of the Python library \verb|scikit-learn| [\citet{Pedregosa2011}] with the following hyperparameter search space:

\begin{verbatim}
    "C": (0.01, 0.1, 1, 10, 100, 1000),
    "gamma": (10, 1, 0.1, 0.01, 0.001, 0.0001),
    "kernel": ("rbf",)
\end{verbatim}


\subsection{k-Nearest Neighbors}

All previously explained algorithms require the estimation of parameters that constrain the model to a predefined functional form. 
One approach that circumvents this inflexibility is the \gls{kNN}. 
It is one of the simplest non-parametric methods and essentially only requires the storage of the whole dataset [\citet[p. 105]{James2021}]. 
The outcome of any observation~$\mathbf{x}_0$ is then predicted by a comparison to all other instances, where similarity is measured based on a distance measure. 
A common choice, in this regard, is the Minkowski metric (p-norm), calculated as
\begin{equation*}
    \bigl\lVert \mathbf{x}_0 - \mathbf{x}_t \bigr\rVert^p = \biggl ( \sum_{l=1}^q \lvert x_{0l} - x_{tl} \rvert ^p\biggl)^{\frac{1}{p}},
\end{equation*}
where $p \in \mathbb{N}$ is the choice of norm, and $q \in \mathbb{N}$ represents the number of features (not to be confused with $k \in \mathbb{N}$, which, in this case is the number of nearest neighbors). 
Based on such a metric a similarity ranking can then be constructed, which is used to create a neighborhood $\mathcal{N}_k(\mathbf{x}_{0})$ of the $k$ most similar observations. 
Afterwards, the outcome for the instance $\mathbf{x}_0$ is calculated as an average of the response variables from this neighborhood [\citet[p. 14]{hastie2009}]. 
Due to this rather simple set-up, \gls{kNN} only requires the selection of two hyper-parameters: $k$ and $p$.
We tried a range values of $k$ between 1 and 101, and $p$ as either the manhattan or euclidean norm.


\subsection{Random Forest}

The \gls{RF} is a non-parametric approach that is based on regression trees. 
It is an ensemble learning technique that combines multiple \glspl{DT} to create an overall more accurate and robust model. 
The main idea of a \gls{DT} is to ask a series of yes-no questions based on the input variables that divide the observations into partitions [\citet[p. 664]{Bishop2006}]. 
The outcome of an instance is then predicted by the average outcome in the respective group. 
\gls{RF} extends this methodology by simultaneously growing multiple de-correlated trees based on a set of bootstrap samples [\citet[pp. 587--588]{hastie2009}]. 
The overall prediction is then aggregated by averaging over the individual tree outputs. 
Another peculiarity is the introduction of more randomness by selecting random subsets of variables per yes-no-question, which leads to a greater tree diversity [\citet[pp. 588--589]{hastie2009}]. 
Note that due to this more sophisticated set-up, the user again has to provide several hyperparameters such as the number of de-correlated trees, the percentage of randomly selected variables that are considered per split, or the allowed depth per \gls{DT}.
In our analysis, we perform a random search of the following space:

\begin{verbatim}
    "n_estimators": [200],
    "max_features": np.arange(0.2, 1.05, 0.05),
    "max_depth": randint(4, 30),
    "min_samples_leaf": randint(5, 50),
\end{verbatim} 


\subsection{Extremely Randomized Trees}

The \gls{ET} are another tree-based method that can be seen as an extension of \gls{RF}. 
It introduces even more randomness by determining the binary decision per yes-no question based on random thresholds for each variable [\citet[p. 169]{geron2022}]. 
Due to this adaption, the trees become even more diverse, and the algorithm also gets much faster than \gls{RF}, which aims to find optimal thresholds. 
Nevertheless, most of the hyperparameters of \gls{RF} remain and accordingly must be specified before the model estimation.
In our numerical experiments, we search the same space of hyperparameters as given for \gls{RF} above.


\subsection{Extreme Gradient Boosting}
\label{subsec:XGBoost}

The last non-parametric and tree-based approach that is used in this paper is \gls{XGB}. 
While \gls{RF} and \gls{ET} grow multiple independent \glspl{DT}, the main idea of \gls{XGB} is to fit a series of trees that aim to correct the predictions of their predecessors [\citet[pp. 110--112]{wade2020}]. 
Formally, for an arbitrary instance $t$ this procedure can be expressed as
\begin{equation*}
    \hat{y}_t = \sum_{T=1}^{\bar{T}} f_T(\mathbf{x}_t) \quad \text{with} \quad f_T \in \mathcal{F},
\end{equation*}
where $\bar{T} \in \mathbb{N}$ denotes the number of additive functions, and $\mathcal{F} = \bigl\{f_T(\mathbf{x}_t) = w_{q(\mathbf{x}_t)} \bigr\} \bigl(q: \mathbb{R}^k \rightarrow d \in \{1, \ldots, D \}, w_d \in \mathbb{R} \bigr)$ represents the space of regression trees with $D \in \mathbb{N}$. 
Within this space, $q$ denotes the structure of each tree and allocates an instance to a leaf index $d$ with a corresponding leaf score $w_d$. 
The leaf scores $w_d$ of the $T^{th}$ tree are then estimated based on minimizing a regularized objective function that also contains a penalty term for the complexity of the additional tree [\citet[p. 786--787]{Chen2016}]. 
Note that for this \gls{ML} algorithm also several hyperparameters must be specified by the user [\citet[ch. 6]{wade2020}]. 
For example, the number of additive trees $\bar{T}$, the allowed depth of an individual tree or the strength of the regularization terms must be selected.
In our numerical experiments, we use the \gls{XGB} implementation of the Python library \verb|XGBoost| [\citet{Chen2016}] with the following hyperparameter search space:
\begin{verbatim}
    "colsample_bytree": np.arange(0.2, 1.05, 0.05),
    "gamma": uniform(0, 2.5),
    "learning_rate": uniform(0.01, 0.3),
    "reg_alpha": uniform(40, 180),
    "reg_lambda": uniform(0, 1),
    "max_depth": randint(2, 18),
    "n_estimators": randint(150, 700),
    "subsample": np.arange(0.2, 1.05, 0.05)
\end{verbatim} 


\subsection{K-Fold Cross Validation}

Figure~\ref{fig:k-fold-cross-val} illustrates the main idea of K-fold cross validation with K = 6.
In this framework, one of the six folds is set aside as test dataset for the final model evaluation. 
The remaining folds are used for variable selection and hyperparameter tuning.
In particular, the models are fitted five times for each configuration by cycling through the folds one to five as the validation set while using the remaining four folds for training.
The overall validation score is then computed as the average of the individual scores and serves to assess the quality of a given feature set or hyperparameter choice [\citet[pp. 32--33]{Bishop2006}; \citet[pp. 241--242]{hastie2009}].

\begin{figure}[H]
    \centering
    \includegraphics[width=0.5\linewidth]{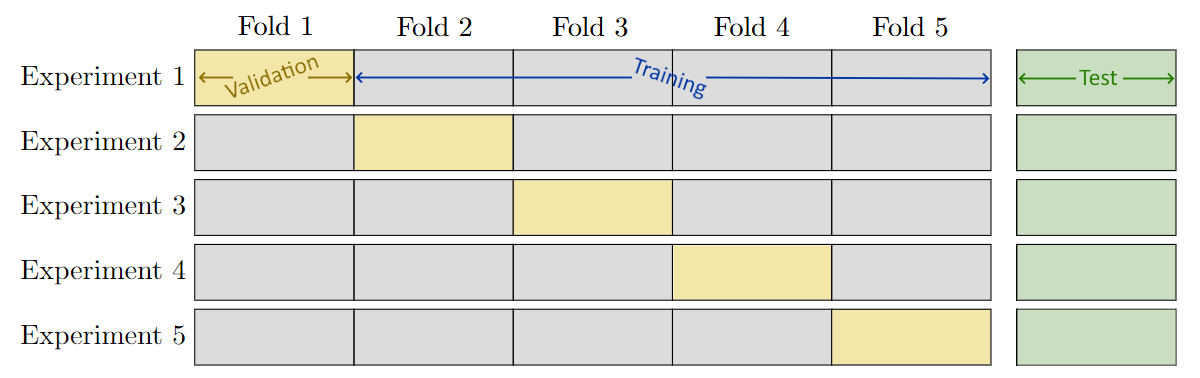}
    \caption{Visualization of K-fold cross validation with a final test set.}
    \label{fig:k-fold-cross-val}
\end{figure}

\subsection{Further Regression Scores}

To evaluate the performance of the \gls{ML} algorithms, we use three performance metrics: $\text{R}^2$, \gls{RMSE} and \gls{MAE}. 
While the results of the $\text{R}^2$ score are directly incorporated in Subsection~\ref{sec:selection_of_fuel_prediction_function}, the \gls{RMSE} and \gls{MAE} values are reported in this subsection of the Appendix for the sake of brevity.
Based on the results shown in Table~\ref{tab:summary_RMSE}, it can be observed that the model rankings based on the \gls{RMSE} and the  $\text{R}^2$ score are identical.
Overall, \gls{XGB} outperforms all other \gls{ML} models.

\begin{table}[H]
    \scriptsize
    \centering
    \captionsetup{justification=centering}
    \caption{\gls{RMSE} scores of the \gls{ML} models on the test datasets.}
\begin{tabular}{l|ccccccccc}
\toprule
& \acrshort{LinReg}
& \acrshort{LASSO}
& \acrshort{kNN}
& \acrshort{MLP}
& \acrshort{SVR}
& \acrshort{ET}
& \acrshort{RF}
& \acrshort{XGB}\\
\midrule
$V_1$   &  26.8676 &  26.7702 & 29.9897 & 17.3455 & 10.8206 & 9.4063 & 8.7529 & \textbf{6.2752}\\
$V_2$   &  20.5689 &  20.6605 & 30.6290 & 18.1058 & 11.0696 & 9.0086 & 9.2723 & \textbf{7.7433}\\
$V_3$   &  20.2513 &  20.2396 & 30.8415 & 12.5006 & 6.5161  & 5.8514 & 5.9280 & \textbf{4.0179}\\
$V_4$   &  11.4218 &  11.5367 & 20.9961 & 11.4799 & 12.2190 & 8.5980 & 8.9163 & \textbf{6.8672}\\
$V_5$   &  16.0106 &  16.0452 & 27.4599 & 15.2965 & 10.6673 & 7.8016 & 6.9768 & \textbf{6.6745}\\
$V_6$   &  47.5881 &  47.4514 & 49.4213 & 27.7973 & 14.0532 & 9.7829 & 9.3958 & \textbf{8.5962}\\
$V_7$   &  21.7915 &  21.8496 & 28.2506 & 14.1761 & 8.4934  & 8.6389 & 9.7555 & \textbf{8.2419}\\
$V_9$   &  20.6359 &  20.8467 & 26.0503 & 13.7851 & 8.8230  & 6.2017 & 6.0882 & \textbf{4.1277}\\
$V_{10}$&  16.9664 &  16.9650 & 26.4931 & 14.5940 & 11.5048 & 5.8268 & 5.6535 & \textbf{5.4880}\\
\bottomrule
\end{tabular}

    \label{tab:summary_RMSE}
\end{table}

Table~\ref{tab:summary_MAE} reports the \gls{MAE} scores for each vessel and model. 
In contrast to the \gls{RMSE} results, \gls{XGB} does not consistently outperform the other approaches, as the other two tree-based methods provide slightly better outcomes for a few vessels. 
In particular, \gls{RF} is the best model in three out of nine cases, and \gls{ET} outperforms all other \gls{ML} algorithms in two out of nine cases. 
These ranking differences can be explained by the importance of larger prediction errors. 
Specifically, the \gls{RMSE} and the $\text{R}^2$ score penalize high errors more strongly by squaring all errors, while the \gls{MAE} only considers the absolute difference.

\begin{table}[H]
    \scriptsize
    \centering
    \captionsetup{justification=centering}
    \caption{\gls{MAE} scores of the \gls{ML} models on the test datasets.}
\begin{tabular}{l|ccccccccc}
\toprule
& \acrshort{LinReg}
& \acrshort{LASSO}
& \acrshort{kNN}
& \acrshort{MLP}
& \acrshort{SVR}
& \acrshort{ET}
& \acrshort{RF}
& \acrshort{XGB}\\
\midrule
$V_1$   & 20.4051 & 19.3193 & 19.8012 &  13.0955 & 8.8206 & 8.4063           & 7.7529          & \textbf{6.1702}\\
$V_2$   & 9.3461  & 9.6350  & 19.4036 &  13.4157 & 7.1978  & 7.1606          & \textbf{6.9578} & 7.5517\\
$V_3$   & 9.3176  & 9.2890  & 17.7586 &  10.1653 & 6.3763  & 4.5847          & 4.4440          & \textbf{3.9315}\\
$V_4$   & 7.0739  & 7.1368  & 10.5170 &  11.7771 & 5.5238  & \textbf{4.4080} & 4.7042          & 5.5446\\
$V_5$   & 8.1376  & 8.1762  & 17.3489 &  14.3571 & 7.3773  & 6.9749          & \textbf{6.5662} & 6.8213\\
$V_6$   & 26.8383 & 27.4823 & 19.0256 &  14.9325 & 12.4432 & 7.3608          & \textbf{6.5948} & 7.8703\\
$V_7$   & 11.7862 & 11.8897 & 14.6954 &  9.6867  & 8.1088  & \textbf{5.4271} & 6.2481          & 5.8392\\
$V_9$   & 11.1912 & 11.8410 & 15.5353 &  10.4029 & 5.9223  & 4.4587          & 4.4701          & \textbf{2.9728}\\
$V_{10}$& 10.7153 & 10.5808 & 18.2649 &  9.0718  & 9.1003  & 5.9203          & 5.8910          & \textbf{5.0052}\\
\bottomrule
\end{tabular}

    \label{tab:summary_MAE}
\end{table}


\subsection{Bootstrap Confidence Intervals of Performance Metrics}
\label{sec:bootstrapping}

For the sake of brevity, Table~\ref{tab:XGB_detailed} in Subsection~\ref{sec:selection_of_fuel_prediction_function} only reported point estimates for the three performance metrics: \gls{RMSE}, \gls{MAE} and $\text{R}^2$ score.
To study the uncertainty behind these results, this section of the Appendix provides corresponding \glspl{CI}.
For this purpose, we bootstrap 100,000 independent samples with replacement based on the training and testing split of each vessel, and repeatedly score the model's performance [\citet[pp. 655--656]{fox2015}].
In total, this leads to 36 \glspl{CI} that are simultaneously constructed for each evaluation criterion. 
To ensure that the familywise error rate stays at $\alpha = 0.05$, the confidence level is adjusted based on the Bonferroni correction [\citet{dunn1961}] leading to $\alpha = \frac{0.05}{36} \approx0.0014$.  
The low ($\alpha$), medium and high ($1-\alpha$) percentiles are presented in Tables~\ref{tab:XGB_bootstrapping_RMSE}, \ref{tab:XGB_bootstrapping_MAE} and~\ref{tab:XGB_bootstrapping_R2}.

\begin{table}[H]
    \scriptsize
    \centering
    \caption{\gls{RMSE} bootstrap \glspl{CI} results for the \gls{XGB} model whose training is described in Section~\ref{sec:numerical_experiments}.}
    \begin{tabular}{l | ccc | ccc }
\toprule
& \multicolumn{3}{c}{Training \gls{RMSE}} &  \multicolumn{3}{c}{Test \gls{RMSE}} \\
& Low $0.14\%$ & Med $50.00\%$ & High $99.86\%$ & Low $0.14\%$ & Med $50.00\%$ & High $99.86\%$ \\
\midrule
$V_1$    & 1.9131   & 2.0067   & 2.1675   & 5.8628   & 6.2659   & 6.6767 \\
$V_2$    & 3.5962   & 3.8042   & 4.0527   & 7.0472   & 7.5049   & 7.9726 \\
$V_3$    & 0.5668   & 0.6189   & 0.6792   & 3.5733   & 4.0010   & 4.4711 \\
$V_4$    & 2.5193   & 2.7142   & 2.9633   & 6.4384   & 6.9332   & 7.6156 \\
$V_5$    & 1.0277   & 1.1175   & 1.2120   & 5.9388   & 6.5944   & 7.7847 \\
$V_6$    & 1.6050   & 1.6928   & 1.7873   & 7.7066   & 8.6141   & 9.4907 \\
$V_7$    & 3.4994   & 3.6774   & 3.8838   & 7.1598   & 8.1078   & 9.1830 \\
$V_9$    & 2.3398   & 2.5132   & 2.7144   & 3.7714   & 4.0364   & 4.3271 \\
$V_{10}$ & 1.6891   & 1.7962   & 1.9136   & 5.3827   & 5.6037   & 5.8318 \\
\midrule
\bottomrule
\end{tabular}
    \label{tab:XGB_bootstrapping_RMSE}
\end{table}

\begin{table}[H]
    \scriptsize
    \centering
    \caption{\gls{MAE} bootstrap \glspl{CI} given for the \gls{XGB} model whose training is described in Section~\ref{sec:numerical_experiments}.}
    \begin{tabular}{l | ccc | ccc }
\toprule
& \multicolumn{3}{c}{Training \gls{MAE}} &  \multicolumn{3}{c}{Test \gls{MAE}} \\
& Low $0.14\%$ & Med $50.00\%$ & High $99.86\%$ & Low $0.14\%$ & Med $50.00\%$ & High $99.86\%$ \\
\midrule
$V_1$    & 1.1843   & 1.2530   & 1.3209   & 5.5687   & 6.2680   & 6.8778  \\
$V_2$    & 4.5828   & 4.8522   & 5.0774   & 6.6710   & 7.3887   & 8.2529  \\
$V_3$    & 0.5827   & 0.6131   & 0.6429   & 3.3106   & 3.9388   & 4.5517  \\
$V_4$    & 1.5849   & 1.6863   & 1.7934   & 4.9216   & 5.5938   & 6.0942  \\
$V_5$    & 1.0781   & 1.1233   & 1.1773   & 6.0968   & 6.7239   & 7.5140  \\
$V_6$    & 1.7144   & 1.8600   & 2.0268   & 6.7174   & 7.9455   & 8.9072  \\
$V_7$    & 1.9282   & 2.0127   & 2.1030   & 5.0456   & 5.6326   & 6.1831  \\
$V_9$    & 1.3578   & 1.4203   & 1.4810   & 2.4520   & 2.8794   & 3.2993  \\
$V_{10}$ & 1.4199   & 1.5016   & 1.5938   & 4.9048   & 5.2469   & 5.5879  \\
\midrule
\bottomrule
\end{tabular}
    \label{tab:XGB_bootstrapping_MAE}
\end{table}

\begin{table}[H]
    \scriptsize
    \centering
    \caption{$\text{R}^2$ bootstrap \glspl{CI} given for the \gls{XGB} model whose training is described in Section~\ref{sec:numerical_experiments}.}
    \begin{tabular}{l | ccc | ccc }
\toprule
& \multicolumn{3}{c}{Training $\text{R}^2$} &  \multicolumn{3}{c}{Test $\text{R}^2$} \\
& Low $0.14\%$ & Med $50.00\%$ & High $99.86\%$ & Low $0.14\%$ & Med $50.00\%$ & High $99.86\%$ \\
\midrule
$V_1$    & 0.9964   & 0.9972   & 0.9977   & 0.9617   & 0.9681   & 0.9739 \\
$V_2$    & 0.9799   & 0.9806   & 0.9813   & 0.9053   & 0.9069   & 0.9084 \\
$V_3$    & 0.9988   & 0.9996   & 0.9996   & 0.9827   & 0.9869   & 0.9905 \\
$V_4$    & 0.9914   & 0.9919   & 0.9922   & 0.9139   & 0.9172   & 0.9197 \\
$V_5$    & 0.9982   & 0.9987   & 0.9991   & 0.9721   & 0.9803   & 0.9838 \\
$V_6$    & 0.9873   & 0.9910   & 0.9940   & 0.9026   & 0.9237   & 0.9284 \\
$V_7$    & 0.9834   & 0.9837   & 0.9977   & 0.9355   & 0.9410   & 0.9452 \\
$V_9$    & 0.9958   & 0.9959   & 0.9994   & 0.9920   & 0.9937   & 0.9952 \\
$V_{10}$ & 0.9972   & 0.9974   & 0.9985   & 0.9749   & 0.9828   & 0.9896 \\
\midrule
\bottomrule
\end{tabular}
    \label{tab:XGB_bootstrapping_R2}
\end{table}

\newpage
\subsection{Python implementation}\label{sec:python_imple}
Here we provide a self-contained Python 3.12 implementation of Algorithms~\ref{alg:brute_force} and~\ref{alg:dynamic_programming} and an example problem instance.
The most up-to-date code should be accessed through the project's GitHub repository:
\url{https://github.com/SamuelWardPhD/Cleaning-Schedule-Optimization}

\begin{lstlisting}[style=custom,language=Python]
import itertools
import numpy as np
from typing import Callable, Sequence


def main():
    # Get the parameters of the example problem instance
    c, b_bar, Delta, P, F, f = example_problem_instance()

    # Solve with brute force
    algorithm_1_brute_force_search(c, f)

    # Solve with dynamic programming
    algorithm_2_dynamic_cleaning_schedule_optimiser(c, b_bar, Delta, P, F)


def example_problem_instance():
    """
    Example parameterization with n = 5 voyages of 10 hourly observations each.
    """
    n = 5                                                       # Number of voyages
    c = np.array((45_000, 44_000, 39_000, 36_000, 42_000))      # Cost c_j of a cleaning before voyage j (USD)
    b_bar = np.array((100, 100))                                # Initial fouling state b_bar (days)
    Delta = np.array([[25, 25], [25, 25], [25, 25], [25, 25]])  # Accumulations Delta_1, ..., Delta_{n-1} (days)
    P = np.diag((0, 1))                                         # Cleaning resets the days since in-water survey
    p = 0.493                                                   # Fuel price p_t (USD/kg)

    # Variables u_t = (e.g. speed through water in kn, draught/wave height in m) for the hours t in T_j of voyage j
    speed = np.array((
        (8.5, 9.2, 10.0, 11.3, 12.1, 10.8, 9.7, 11.9, 13.0, 12.4),
        (9.0, 9.6, 10.4, 11.0, 12.0, 12.3, 11.1, 12.7, 13.2, 12.8),
        (8.2, 8.9, 9.5, 10.1, 10.7, 10.3, 9.6, 10.8, 11.5, 11.0),
        (9.5, 10.2, 10.8, 11.6, 12.5, 12.0, 11.3, 12.9, 13.5, 13.0),
        (8.8, 9.4, 10.1, 10.9, 11.7, 11.2, 10.5, 11.8, 12.6, 12.1),
    ))
    draught = np.array((
        (9.8, 10.1, 10.0, 10.4, 10.6, 10.3, 10.2, 10.7, 10.9, 10.8),
        (9.7, 9.9, 10.0, 10.1, 10.2, 10.1, 9.8, 10.3, 10.4, 10.2),
        (10.0, 10.3, 10.5, 10.7, 10.8, 10.6, 10.4, 10.9, 11.0, 10.8),
        (9.6, 9.8, 10.0, 10.1, 10.2, 10.0, 9.9, 10.3, 10.4, 10.2),
        (9.9, 10.0, 10.2, 10.3, 10.5, 10.4, 10.1, 10.6, 10.7, 10.5),
    ))

    # Within-voyage accumulation delta_t, i.e. both measures increase by 1/24 day for every hour of the voyage
    delta = np.repeat(np.arange(10).reshape(-1, 1)/24, 2, axis=1)

    # Prediction g(u_t, s_t) of the fuel oil consumption (kg/h) for all hours of a voyage at once,
    # where fouling is assumed to build up faster on an older hull coating
    def g(speed_t, draught_t, s_t):
        return draught_t*speed_t**2 + 18*s_t[:, 1]*(1 + s_t[:, 0]/365)

    # Equation 3: F_j(b) = sum_{t in T_j} p_t*g(u_t, b + delta_t)
    F = [lambda b, j=j: np.sum(p*g(speed[j - 1], draught[j - 1], b + delta)) for j in range(1, n + 1)]

    # Equation 2: b_j = (I - z_j*P)(b_{j-1} + Delta_{j-1}), where b_0 = 0 and Delta_0 = b_bar
    Delta_from_0 = np.vstack((b_bar, Delta))  # Row j - 1 is Delta_{j-1}

    def b(z_1_to_j):
        b_j = np.zeros(2)
        for j, z_j in enumerate(z_1_to_j, start=1):
            b_j = (np.eye(2) - z_j*P)@(b_j + Delta_from_0[j - 1])
        return b_j

    # Fuel costs of voyage j as a function of the cleaning history, f_j(z_1:j) = F_j(b_j), for (CLEANING-1)
    f = [lambda z_1_to_j, j=j: F[j - 1](b(z_1_to_j)) for j in range(1, n + 1)]

    return c, b_bar, Delta, P, F, f


def algorithm_1_brute_force_search(
        c: Sequence,
        f: Sequence[Callable],
        verbose=3,
):
    """
    Algorithm 1 of the paper, i.e. a brute-force search for (CLEANING-1)
        minimise:  J(z) = sum_{j=1,...,n} f_j(z_1:j) + c_j*z_j    over z in {0, 1}^n.

    :param c:          Numpy array where c[j - 1] is the cost c_j of a cleaning before voyage j;
    :param f:          Sequence where f[j - 1] is the function f_j that maps the cleaning history
                       z_1:j = (z_1, ..., z_j) to the fuel costs of voyage j;
    :param verbose:    Level of output. Set to zero to suppress all print statements;
    """

    # Number of voyages
    n = len(c)

    # Validation
    assert len(f) == n, f"Invalid parameterization: len(f)={len(f)} should be equal to len(c)={n}."

    # Title
    if verbose:
        print("\n===== Algorithm 1: Brute Force =====")

    # The objective function of (CLEANING-1)
    # J(z) = sum_{j=1,...,n}( f_j(z_1:j) + c_j*z_j )
    def J(z):
        return sum(f[j - 1](z[:j]) + c[j - 1]*z[j - 1] for j in range(1, n + 1))

    # Initialise variables that will store the best solution found so far
    z_star = None
    J_star = np.inf

    # Iterate through every cleaning schedule z in {0, 1}^n
    # I.e. (0, ..., 0, 0) then (0, ..., 0, 1) then (0, ..., 1, 0) and so on
    for progress, z in enumerate(itertools.product((0, 1), repeat=n)):
        J_z = J(z)

        # Maintain the best solution so far
        if J_z < J_star:
            J_star = J_z
            z_star = z

        # Update
        if verbose >= 2 and (progress + 1) in (((2**n)*p)//100 for p in range(20, 101, 20)):
            print(f"   {progress + 1:>3} out of {2**n} combinations searched ({(progress + 1)/(2**n):.1%})")

    # Output results to console
    if verbose:
        print(f"Optimal cleaning schedule z*:   {z_star}")
        print(f"Optimal objective value J*:     {J_star:.2f}")

    return z_star, J_star


def algorithm_2_dynamic_cleaning_schedule_optimiser(
        c: Sequence,
        b_bar: np.ndarray,
        Delta: np.ndarray,
        P: np.ndarray,
        F: Sequence[Callable],
        verbose=3,
):
    """
    Algorithm 2 of the paper, i.e. dynamic programming for (CLEANING-2).

    :param c:          Numpy array where c[j - 1] is the cost c_j of a cleaning before voyage j;
    :param b_bar:      Initial fouling state at the start of voyage 1 if no cleaning takes place;
    :param Delta:      Array where Delta[j - 1] is the accumulation Delta_j from the start of voyage j to voyage j+1;
    :param P:          Diagonal 0-1 matrix that selects the biofouling measures reset by a cleaning;
    :param F:          Sequence where F[j - 1] is the function F_j that maps the fouling state to the fuel costs;
    :param verbose:    Level of output. Set to zero to suppress all print statements;
    """
    n = len(c)

    # Validation
    assert len(F) == n, f"Invalid parameterization: len(F)={len(F)} should be equal to len(c)={n}."
    assert len(Delta) >= n - 1, f"Invalid parameterization: len(Delta)={len(Delta)} should be at least n-1={n - 1}."
    assert P.shape == (len(b_bar), len(b_bar)), f"Invalid parameterization: P.shape={P.shape} should be (d, d)."

    # Title
    if verbose:
        print("\n===== Algorithm 2: Dynamic Programming =====")

    # Row j of Lambda is the accumulated fouling Lambda_j = sum_{i=0,...,j-1} Delta_i, where Delta_0 = b_bar
    Lambda = np.vstack((np.zeros(len(b_bar)), np.cumsum(np.vstack((b_bar, Delta[:n - 1])), axis=0)))

    # Phi[j, l] is the optimal cost-to-go Phi_j(l) of the voyages j, ..., n for 1 <= j <= n+1 and 0 <= l < j
    # pi[j, l] is the optimal decision pi_j(l) for 1 <= j <= n and 0 <= l < j
    # In both cases, l is the index of the voyage before which the most recent cleaning took place (l = 0: none)
    Phi = np.full((n + 2, n + 1), np.nan)
    pi = np.zeros((n + 1, n), dtype=int)

    # Boundary condition Phi_{n+1}(l) = 0 as there are no voyages after voyage n
    Phi[n + 1, :] = 0

    # Backward pass: Bellman equation
    for j in reversed(range(1, n + 1)):

        if verbose >= 2:
            print(f"   --- Voyage j = {j} ---")

        # The cost kappa of a cleaning before voyage j is the sum of
        #   + The cost of cleaning c_j
        #   + The fuel costs F_j(b_j^(j)) of voyage j directly after a cleaning, where b_j^(j) = Lambda_j - P*Lambda_j
        #   + The optimal cost-to-go Phi_{j+1}(j) of the voyages after j
        kappa = c[j - 1] + F[j - 1](Lambda[j] - P@Lambda[j]) + Phi[j + 1, j]

        # Assume that the most recent cleaning took place before voyage l
        for l in reversed(range(j)):

            # The cost nu of not cleaning before voyage j is the sum of
            #   + The fuel costs F_j(b_j^(l)) of voyage j with fouling state b_j^(l) = Lambda_j - P*Lambda_l
            #   + The optimal cost-to-go Phi_{j+1}(l) of the voyages after j
            b_j_l = Lambda[j] - P@Lambda[l]
            nu = F[j - 1](b_j_l) + Phi[j + 1, l]

            # Record the optimal cost-to-go and decision, where cleaning is preferred in case of a tie
            Phi[j, l] = min(kappa, nu)
            pi[j, l] = 1 if kappa <= nu else 0

            if verbose >= 2:
                print("".join((
                    f"l: {l:<4}",
                    f"b_j^(l): {str(b_j_l):<16}",
                    f"kappa (clean): {kappa:<12.2f}",
                    f"nu (do not clean): {nu:<12.2f}",
                    f"pi_j(l): {pi[j, l]}",
                )))

    # Forward pass: recover an optimal schedule z_hat by following the optimal decisions pi_j(l)
    z_hat = []
    l = 0
    for j in range(1, n + 1):
        z_hat.append(int(pi[j, l]))
        if z_hat[-1] == 1:
            l = j

    # Output results to console
    if verbose:
        print("   --- Results ---")
        print(f"Optimal cleaning schedule z_hat:   {tuple(z_hat)}")
        print(f"Optimal objective value Phi_1(0):  {Phi[1, 0]:.2f}")

    return tuple(z_hat), Phi[1, 0]


if __name__ == '__main__':
    main()

__author__ = "Samuel Ward"
__maintainer__ = "Samuel Ward"
__email__ = "s.ward@soton.ac.uk"
__credits__ = ["Samuel Ward", "Marah-Lisanne Thormann"]
__version__ = "8.0.0"
\end{lstlisting}


\addcontentsline{toc}{section}{Declaration of LLM Usage}

\section*{Declaration of LLM Usage}

During the preparation of this work, the authors made use of Large Language Models (LLMs) for conducting literature searches, formatting LaTeX tables and equations, and proofreading.
We are aware that these models may hallucinate content that is incorrect and references that do not exist. 
Whenever changes were suggested by the LLMs, we verified them either based on our own domain knowledge or by searching for supporting references.
Ultimately, the authors made all decisions regarding the inclusion, revision, or exclusion of content.
All the theorems and innovations were conceptualized and implemented by the authors.



\end{document}